\documentclass[a4paper]{article}
\usepackage{calc, amsmath, amsthm, a4, latexsym, amssymb, layout, subfigure, color, bm, xspace}
\usepackage[dvips]{graphicx}

\newtheorem{thm}{Theorem}[section]
\newtheorem{lemma}[thm]{Lemma}

\newtheorem{prop}[thm]{Proposition}

\theoremstyle{definition}

\newtheorem{rem}[thm]{Remark}

\newcommand{\be}[1]{\begin{equation}\label{#1}}
\newcommand{\ee}{\end{equation}}
\newcommand{\ba}{\begin{array}}
\newcommand{\ea}{\end{array}}
\newcommand{\bal}{\begin{aligned}}
\newcommand{\eal}{\end{aligned}}

\newcommand{\ie}{i.e.\@\xspace}
\newcommand{\R}{\mathbb{R}}

\newcommand{\Z}{\mathbb{Z}}
\newcommand{\E}{\mathbb{E}}

\newcommand{\calC}{\mathcal{C}}

\newcommand{\calE}{\mathcal{E}}

\newcommand{\calM}{\mathcal{M}}

\newcommand{\1}{1{\hskip -2.5 pt}\hbox{I}} %{{\bm 1}}
\newcommand{\la}{\lambda}

\newcommand{\eps}{\varepsilon}

\renewcommand{\phi}{\varphi}

\newcommand{\invis}[1]{}

\newcommand{\ra}{\rightarrow}
\newcommand{\ua}{\uparrow} % 15/05/07
\newcommand{\da}{\downarrow}
\newcommand{\weakconv}{\Rightarrow}

\newcommand{\diff}{\mathrm{d}}
\newcommand{\Prob}{\mathrm{Prob}}
\newcommand{\argmax}{\mathrm{argmax}}
\newcommand{\ox}{\overline{x}}
\newcommand{\oy}{\overline{y}} 
\newcommand{\tB}{\tilde{B}}
\newcommand{\restrH}{H^*}

\newcommand{\oz}{\overline{z}}

\renewcommand{\phi}{\varphi}

\renewcommand{\d}{\mathrm{d}}

\newcommand{\heap}[2]{\genfrac{}{}{0pt}{}{#1}{#2}}
\newcommand{\ssup}[1] {{\scriptscriptstyle{({#1}})}}
\newcommand{\sfrac}[2]{\mbox{$\frac{#1}{#2}$}}

\begin{document}

\begin{center}
{\LARGE \bf Ageing in the parabolic Anderson model}\\
\vspace{0.7cm}
\textsc{Peter M\"orters\footnote{Department of Mathematical Sciences, University of Bath, Bath BA2 7AY, United Kingdom.}, 
Marcel Ortgiese\footnote{Institut f\"ur Mathematik, Technische Universit\"at Berlin, Stra\ss e des 17. Juni 136, 10623 Berlin, Germany.}} 
and \textsc{Nadia Sidorova\footnote{Department of Mathematics, 
University College London, Gower Street, London WC1E 6BT, United Kingdom.}}\\
%\vspace{0.1cm}
%{\small(Draft of \version)} 
%\vspace{0.5cm}\\
%\version
\end{center}

\vspace{0.3cm}

%\centerline{\LARGE \bf Ageing in the parabolic Anderson model}

%\vspace{0.5cm}

%\thispagestyle{empty}

\begin{quote}{\small {\bf Abstract: }
The parabolic Anderson model is the Cauchy problem for the heat equation with
a random potential. We consider this model in a setting which is continuous in time 
and discrete in space, and focus on time-constant, independent and identically distributed 
potentials with polynomial tails at infinity. We are concerned with the long-term temporal 
dynamics of this system. Our main result is that the periods, in which the profile of the 
solutions remains nearly constant, are increasing linearly over time, a phenomenon 
known as ageing. We describe this phenomenon in the weak sense, by looking at the asymptotic probability
of a change in a given time window, and in the strong sense, by identifying the almost sure upper 
envelope for the process of the time remaining until the next change of profile. We also prove functional 
scaling limit theorems for profile and growth rate of the~solution of the parabolic Anderson model.}
\end{quote}
\vspace{0.5cm}

{\footnotesize{\bf MSc classification (2000):} Primary 60K37 Secondary 82C44 \\[-5mm] 

}

% \tableofcontents

%%%%%%%%%%%%%%%%%%% Introduction %%%%%%%%%%%%%%%%%%%%%%%%%%%%%%%
\section{Introduction}

\subsection{Motivation and overview}

The long term dynamics of disordered complex systems out of equilibrium have been the subject 
of great interest in the past decade. A key paradigm in this research programme is the notion of 
ageing. %which can be characterized in terms of scaling properties of a suitable two-point function of the system. 
Roughly speaking, in an ageing system the probability that there is no 
essential change of the state between time $t$ and time $t+s(t)$ is of constant order for a 
period $s(t)$ which depends increasingly, and often linearly, on the time~$t$.
Hence, as time goes on, in an ageing system changes become less likely and the typical time scales
of the system are increasing. Therefore, as pointed out in~\cite{BF04}, ageing can be associated to 
the existence of infinitely  many time-scales that are inherently relevant to the system. In that respect, 
ageing systems are distinct from metastable systems, which are characterized by a finite number of well 
separated time-scales, corresponding to the lifetimes of different metastable states. % inspired by Bovier and Faggionato
\medskip

Ageing systems are typically rather difficult to analyse analytically. Most results to date
concern either the Langevin dynamics of relatively simple mean field spin glasses, 
see e.g.~\cite{BDG01}, or phenomenological models like the class of \emph{trap models}, see e.g.~\cite{B92, C06, GMW09}. 
The idea behind the latter is to represent a physical system as a particle moving in a random energy landscape with infinitely 
many valleys, or traps. Given the landscape, the particle moves according to a 
continuous time random walk remaining at each trap for an exponential time
with a rate proportional to its depth. While there is good experimental evidence for
the claim that trap models capture the dynamical behaviour of many more complex systems,
a rigorous mathematical derivation of this fact exists only in very few cases. 
% inspired by Bovier and Faggionato
\medskip
\pagebreak[3]

Two recent papers, Dembo and Deuschel~\cite{DD07} and Aurzada and Doering~\cite{AD09},
investigate weaker forms of ageing based on correlations. Both deal with a class of models which includes as a special case a parabolic Anderson
model with time-variable potential and show absence of correlation-based ageing in this case. While this approach
is probably the only way to deal rigorously with complicated models, it is not established that the effect
picked up by these studies is actually really due to the existence or absence of ageing in our sense, or whether
other moment effects are accountable. 
\medskip

In the present work we show that the parabolic Anderson model exhibits ageing behaviour, 
at least if the underlying random potential is sufficiently heavy-tailed. As a lattice
model with random disorder the parabolic Anderson model is a model of significant complexity, 
but its linearity and strong localization features make it considerably easier to study than, 
for example, the dynamics of most non-mean field spin glass models.%
\medskip

Our work has led to three main results. The first one, Theorem~\ref{ageing_for_u}, shows that the probability
that during the time window $[t,t+\theta t]$ the profiles of the solution of the parabolic Anderson problem 
remain within distance $\eps>0$ of each other converges to a constant~$I(\theta)$, which is strictly between 
zero and one. This shows that ageing holds on a linear time scale. Our second main result, Theorem~\ref{asymptotics_R_X}, 
is an almost sure ageing result. We define a function $R(t)$ which characterizes the waiting time starting 
from time~$t$ until the profile changes again. We determine the precise almost sure upper envelope of $R(t)$
in terms of an integral test. The third main result, Theorem~\ref{spatial_limit_u}, is a functional scaling limit theorem 
for the location of the peak, which determines the profile, and for the growth rate of the solution. We give the 
precise statements of the results in Section~\ref{se.results}, and in Section~\ref{se.guide} we provide a rough
guide to the proofs.
\medskip

\subsection{Statement of the main results}\label{se.results}

%\section{Overview}

%Ageing is a phenomenon that can be observed in many physical systems that are out of equilibrium. 
%Vaguely speaking, if starting at time $t$ we observe for how long the system stays in the same state, then if the answer
%depends on $t$, the system ages. 
%Mathematical treatments have focused mainly on the dynamical properties of spin glasses and related trap models, see~\cite{BBC08} for an overview. Also in~\cite{DD07}, the occurrence of ageing is investigated  for different models of interacting diffusion processes. In our project we show that ageing occurs in the parabolic Anderson model.  

The parabolic Anderson model is given by the heat equation on the lattice $\Z^d$ with a random potential, \ie we consider the solution 
$u \colon   ( 0,\infty) \times \Z^d \ra [0,\infty)$ of the Cauchy problem
\[ \begin{array}{llll} \displaystyle\frac{\partial}{\partial t} u(t,z) & = &\Delta u(t,z) + \xi(z) u(t,z) \, ,  
& \quad \mbox{ for } (t,z) \in (0,\infty) \times \Z^d \, , \\[3mm]
\displaystyle  \lim_{t\downarrow 0 } u(t,z) & = & \1_0(z) \, , &\quad \mbox{ for } z \in \Z^d \,. \end{array} \]
Here $\Delta$ is the discrete Laplacian
\[ \Delta f(x) = \sum_{\heap{y \in \Z^d}{y \sim x}} \big( f(y) - f(x) \big) \, , \vspace{-2mm}\]
and $y \sim x$ means that $y$ is a nearest-neighbour of site $x$. The potential $\xi = (\xi(z) \colon z\in \Z^d)$ 
is a collection of independent, identically distributed random variables, which we assume to be Pareto-distributed for some $\alpha > d$, \ie
\[ \Prob \{ \xi(z) \leq x \} = 1 - x^{- \alpha} \, , \quad \mbox{for } x \geq 1 \, . \] 
The condition $\alpha>d$ is necessary and sufficient for the Cauchy problem to have a unique, nonnegative solution, 
see~\cite{GM90}. We write
$$U(t)=\sum_{z\in\Z^d} u(t,z) \qquad \mbox{ for } t\ge 0,$$
for the \emph{total mass} of the solution (which is finite at all times) and
$$v(t,z)= \frac{u(t,z)}{U(t)} \qquad \mbox{ for } t\ge 0, z\in\Z^d,$$
for its~\emph{profile}. It is not hard to see that the total mass grows  
superexponentially in time. Our interest is therefore focused on the changes
in the profile of the solution. 
\bigskip

\subsubsection{Ageing: a weak limit theorem}

Our first ageing result is a weak limit result. We show that for an
observation window whose size is growing linearly in time, the probability 
of seeing no change during the window converges to a nontrivial value. The same
limit is obtained when only the states at the endpoints of the observation
window are considered.
\medskip
%\newpage

\begin{thm}\label{ageing_for_u} For any $\theta>0$ there exists $I(\theta)>0$ such that, for all sufficiently small $\eps >0$, 
$$\begin{aligned}
\lim_{t \ra \infty} \Prob & \Big\{ \sup_{z \in \R^d} \sup_{s \in [t, t+t\theta ]} \big| {v(t,z)} - {v(s,z)} \big| < \eps \Big\}\\
& = \lim_{t \ra \infty} \Prob \Big\{ \sup_{z \in \R^d}  \big| v(t,z) - v(t+t\theta,z) \big| < \eps \Big\} \\ 
& = I (\theta). \end{aligned}$$
\end{thm}

\begin{rem}\ \\[-8mm]
\begin{itemize}
\item 
Note that we only have one ageing regime, which is contrast to the behaviour of the unsymmetric trap 
models described in~\cite{BC05}
\item
An integral representation of $I(\theta)$ will be given in Proposition~\ref{I}, which shows that
the limit is not derived from the generalized arcsine law as in the universal scheme for trap models 
described in~\cite{BC08}. In Proposition~\ref{asymptotics_I_large_c}, we show that there are
 positive constants $C_0, C_1$ such that
$$\lim_{\theta\downarrow 0} \theta^{-1} \, \big(1-I(\theta)\big) =C_0 
\qquad\mbox{ and } \qquad
\lim_{\theta\uparrow \infty} \theta^{d} \, I(\theta) =C_1 \, . $$ \\[-9mm]
\end{itemize}
\end{rem}

%Denote by $B(a,b)$ the Beta function with parameters $a,b$ and 
%by $\tilde{B}(x,a,b)$ as the normalized incomplete Beta function 
%\be{incomplete_beta_function}
%\tilde{B}(x,a,b) = \tfrac{1}{B(a,b)} \int_0^x v^{a-1} (1-v)^{b-1} \diff t \, .\ee
%Then
%$$I(\theta) = \tfrac{1}{B(\alpha - d+1,d)} \int_0^1  v^{\alpha - d} (1-v)^{d-1} \phi_\theta(v) \diff v \, ,$$
%where the weight $\phi_\theta(v)$ is defined by
%\be{weight} \tfrac1{\phi_\theta(v)} = 1- \tilde{B}(v,\alpha - d,d) + (1+\theta)^\alpha \, \big(\tfrac\theta{v}+1\big)^{d - \alpha} 
%\, \tilde{B} \big( \tfrac{v+\theta}{1+\theta}, \alpha - d, d\big) \, . \ee

\subsubsection{Ageing: an almost-sure limit theorem}

The crucial ingredient in our ageing result is the fact that in the case of Pareto distributed potentials
the profile of the solution of the parabolic Anderson problem can be essentially described by one parameter, 
the location of its \emph{peak}. This is due to the one-point localization theorem \cite[Theorem 1.2]{KLMS09}
which states that, for any $\Z^d$-valued process $(X_t \colon t\ge 0)$ 
with the property that $v(t,X_t)$ is the maximum value of the profile at time~$t$, we have 
\begin{equation}\label{one-point}
v(t,X_t) \ra 1 \mbox{ in probability. } 
\end{equation}
In other words, asymptotically the profile becomes completely localized in its peak.
Assume for definiteness that $t\mapsto X_t$ is right-continuous and define the 
\emph{residual lifetime} function by $R(t) = \sup\{ s \geq 0 \colon X_t = X_{t + s} \}$, 
for $t\ge 0$. Roughly speaking, $R(t)$ is the waiting time, at time~$t$, until the next change of peak, 
see the schematic picture in Figure~\ref{remaining_lifetime}.
We have shown in Theorem~\ref{ageing_for_u} that the law of $R(t)/t$ converges to the  law given by 
the distribution function~$1-I$. In the following theorem, we describe the smallest asymptotic upper 
envelope for the process~$(R(t) \colon t\ge 0)$.
\smallskip

\begin{thm}[Almost sure ageing]\label{asymptotics_R_X} 
For any nondecreasing function $h\colon(0,\infty)\to(0,\infty)$ we have, almost surely, 
$$\limsup_{t \ra\infty} \frac{R(t)}{t h(t)} = 
\left\{ \begin{array}{ll} 0 & \mbox{ if } \displaystyle \int_1^\infty \frac{\d t}{t h(t)^d} <\infty,\\[3mm]
\infty & \mbox{ if } \displaystyle \int_1^\infty \frac{\d t}{t h(t)^d} =\infty.\\
\end{array}\right.$$
\end{thm}

\begin{figure}[htbp]
\centering 
\includegraphics[height=4cm]{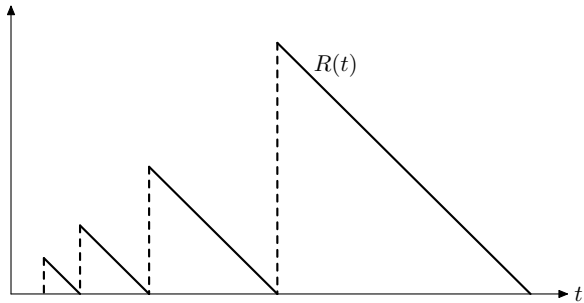} 
\caption{A schematic representation of the residual lifetime function $R$.}
\label{remaining_lifetime}
\end{figure}
	
%\begin{thm}\label{asymptotics_sigma} For all $n$ sufficiently large, for $\gamma, \gamma' \gg 1$, 
%\[ (\log \sigma_n)^{-\gamma} \leq \frac{R(\sigma_n)}{\sigma_n} \leq (\log \sigma_n)^{\gamma'} \,  . \]
%\end{thm}

\subsubsection{A functional scaling limit theorem}

To complete the discussion of the temporal behaviour of the solution it is natural to look for a 
functional limit theorem under suitable space-time scaling of the solution. From~\cite[Theorem 1.2]{HMS08} 
we know that there are heavy fluctuations even in the \emph{logarithm} of the total mass, as we have for $t\uparrow\infty$,
\begin{align}\label{mass} \frac{(\log t)^{\frac{d}{\alpha-d}}}{t^{\frac{\alpha}{\alpha-d}}} \log U(t) 
\Rightarrow Y \, , \end{align}
where $Y$ is a random variable of extremal Fr\'echet type with shape parameter~$\alpha-d$. We 
therefore focus on  the profile of the solution and extend it to $(0,\infty)\times \R^d$ by taking the integer parts of
the second coordinate, letting $v(t,x)=v(t, \lfloor x \rfloor)$. Taking nonnegative measurable functions on $\R^d$ as 
densities with respect to the Lebesgue measure, we can interpret $a^{d} v(t,ax)$ for any $a,t>0$ as an element of the 
space $\calM(\R^d)$ of probability measures on $\R^d$. % equipped with the weak~topology.
By $\delta(y)\in\calM(\R^d)$ we denote the Dirac point mass located in~$y\in\R^d$.
\medskip

%\newpage
\begin{prop}[Convergence of the scaled profile to a wandering point mass]\label{classical}
There exists a nondegenerate  stochastic process $(Y_t \colon t>0)$ such that,
as $T\uparrow\infty$, the following functional scaling limit holds,
\begin{equation}\label{scalim}
\Big(  \big(\tfrac{T}{\log T}\big)^{\frac{\alpha d}{\alpha-d}} \, v\big(tT, \big(\tfrac{T}{\log T}\big)^{\frac{\alpha}{\alpha-d}}x\big)  
\colon t>0 \Big)
\weakconv  \big(  \delta(Y_t)  \colon t  > 0 \big) \, ,
\end{equation}
in the sense of convergence of finite dimensional distributions on the space 
$\calM(\R^d)$ equipped with the weak topology.
\end{prop}
%\medskip
%
\begin{rem}
The process $(Y_t \colon t>0)$ will be described explicitly in and after Remark~\ref{comments_spatial_limit_u}\,(iii).
\end{rem}

In this formulation of a scaling limit theorem the mode of convergence is not
optimal. Also, under the given scaling, islands of diameter $o((\frac{t}{\log t})^{\frac{\alpha}{\alpha-d}})$
at time~$t$ would still be mapped onto single points, and hence the spatial scaling is not sensitive to the 
one-point localization described in the previous section. 
%\smallskip
%
We now state an optimal result in the form of a  functional scaling limit theorem
in the Skorokhod topology for the localization point itself. Additionally, we prove joint convergence of the localization
point together with the value of the potential there. This leads to a Markovian limit process
which is easier to describe, and from which the non-Markovian process  \mbox{$(Y_t \colon t>0)$} can be derived by projection.
This approach also yields an extension of~\eqref{mass} to a functional limit theorem. Here and in the following
we denote by $|x|$ the $\ell^1$-norm of $x\in\R^d$.%
\medskip

\pagebreak[3]

\begin{thm}[Functional scaling limit theorem]\label{spatial_limit_u}\ \\
There exists a time-inhomogeneous Markov process $((Y^{\ssup 1}_t,Y^{\ssup 2}_t) \, : \, t  > 0)$ on $\R^d \times \R$
such that,\\[-7mm]
\begin{itemize}
\item[(a)] as $T\to\infty$, we have 
\[\begin{aligned} \Big(\big(\big(\tfrac{\log T}{T}\big)^{\frac{\alpha}{\alpha-d}} X_{tT},	\big(\tfrac{\log T}{T}\big)^{\frac{d}{\alpha-d}} \xi(X_{tT})\big)\, : \, t > 0\Big) 
\weakconv  \Big(\big(Y^{\ssup 1}_t,Y^{\ssup 2}_t+\tfrac{d}{\alpha-d} |Y^{\ssup 1}_t|\big)\, : \, t  > 0 \Big) \, , \end{aligned}\]
in distribution on the space~$D(0,\infty)$ of c\`adl\`ag functions $f\colon (0,\infty)\to \R^d \times \R$ with respect to the 
Skorokhod topology  on compact subintervals;
\item[(b)]
as $T\to\infty$, we have 
\[\begin{aligned} \Big(	\big(\tfrac{\log T}{T}\big)^{\frac{d}{\alpha-d}}\tfrac{\log U(tT)}{tT} \, : \, t > 0\Big) 
\weakconv  \big( Y^{\ssup 2}_t+\tfrac{d}{\alpha-d}\big(1-\tfrac{1}{t}\big)|Y^{\ssup 1}_t|\, : \, t  > 0 \big) \, , \end{aligned}\]
in distribution on the space~$C(0,\infty)$ of continuous functions $f\colon (0,\infty)\to \R$ with respect to the uniform topology on compact subintervals.
\end{itemize}
\end{thm} 
\smallskip

\begin{rem}\label{comments_spatial_limit_u}\ \\[-7mm]
\begin{itemize}
\item[(i)] Projecting the process onto the first component at time $t=1$ we recover the result of
\cite[Theorem~1.3]{KLMS09}. This result shows in particular  that the peak~$X_t$ of the profile escapes with superlinear speed.
%\item[(ii)] The scaling limit in \eqref{scalim} is an immediate consequence of Theorem~\ref{spatial_limit_u}\,(a). 
\item[(ii)] From the proof of this result it is easy to see that the convergence in both parts of
Theorem~\ref{spatial_limit_u} also holds simultaneously on the space of c\`adl\`ag functions $f\colon (0,\infty)\to \R^d \times \R \times \R$ 
with respect to the Skorokhod topology on compact subintervals.
\item[(iii)] The process $(Y_t \colon t>0)$ in Proposition~\ref{classical} is
is equal to the projected process $(Y_t^{\ssup 1} \colon t>0)$. 
\end{itemize}
\end{rem}

\begin{figure}[htbp]\label{fig_definition_of_Y}
\centering 
\subfigure[$t<1$.]{ 	
		\includegraphics[width=6cm]{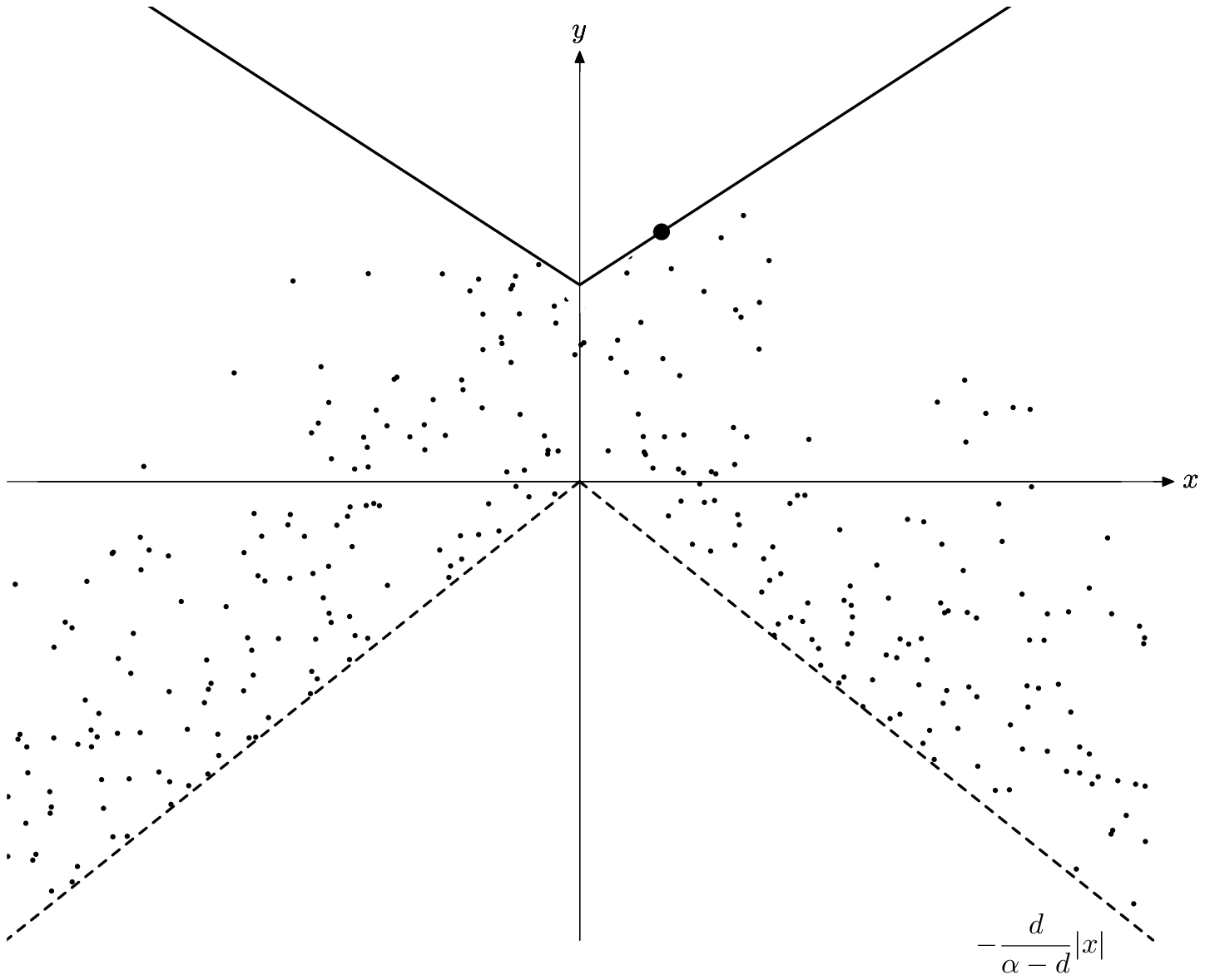}
		\label{fig_definition_of_Y_small_t}  
		}
\subfigure[$t>1$.]{ 	
		\includegraphics[width=6cm]{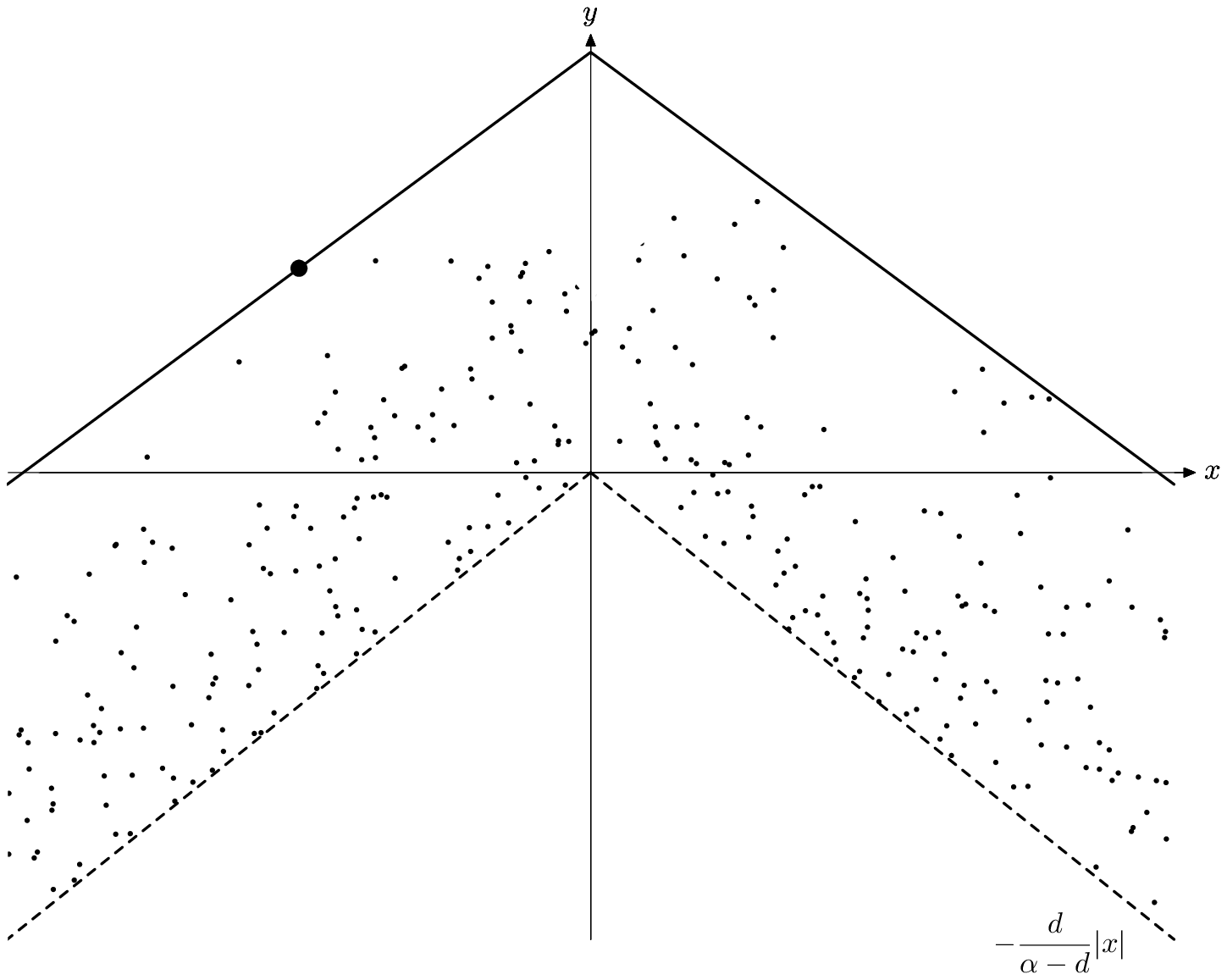}
		\label{fig_definition_of_Y_large_t} 
		}
%\subfigure[$V$ uniformly distributed on \mbox{$[0,1]$}.]{ 	
%		\includegraphics[height=6cm, width=6cm]{spectrum_three.eps}
%		\label{f_V_uniform} }
%\subfigure[$V$ standard normally distributed.]{ 	
%		\includegraphics[height=6cm, width=6cm]{spectrum_four.eps}
%		\label{f_V_normal} }
\caption[The definition of the process $Y_t$ in terms of the point process.]{The definition of the process $(Y_t^{\ssup 1},Y_t^{\ssup 2})$ in terms of the point process $\Pi$. Note that $t$ parametrizes the opening angle of the cone, see (a) for $t<1$ and (b) for $t > 1$.}
\end{figure}

In order to describe the limit process we need to introduce some notation. Denote by $\Pi$ a Poisson point process on 
$H^0 = \{ (x,y)\in\R^d\times \R \colon y > -\frac{d}{\alpha-d}|x| \}$  with intensity measure 
\begin{align*}
\nu(\d x\,\d y)=\d x\otimes \frac{\alpha\d y}{(y+\frac{d}{\alpha-d}|x|)^{\alpha+1}}.
\end{align*}
%where as in the remainder of the paper $|\cdot|$ denotes the $\ell^1$ norm. 
Given the point process, we can define an $\R^d$-valued process $Y^{\ssup 1}_t$ and an $\R$-valued process $Y^{\ssup 2}_t$ 
in the following way. Fix $t>0$ and define the open cone with tip $(0,z)$
$$\calC_t(z) = \big\{ (x,y) \in \R^d \times \R \, : \, y + \tfrac{d}{\alpha-d}(1-\tfrac{1}{t}) |x| > z \big\} \, ,$$
and let 
$$\calC_t = \mathrm{cl}\bigg( \bigcup_{\heap{z>0}{\Pi(\calC_t(z)) = 0}} \calC_t(z) \bigg) \, .$$
Informally, $\calC_t$ is the closure of the first cone $\calC_t(z)$ that `touches' the point process as we decrease~$z$ from infinity. Since $\calC_t\cap \Pi$ contains at most two points, we can define $(Y_t^{\ssup 1}, Y_t^{\ssup 2})$ as the point in this intersection whose projection on the first component has the largest $\ell^1$-norm, see Figures~\ref{fig_definition_of_Y_small_t} and~\ref{fig_definition_of_Y_large_t} for  illustration. The resulting process $((Y_t^{\ssup 1}, Y_t^{\ssup 2}) \colon t > 0)$ is an element of $D(0,\infty)$.% 
\smallskip
% the space of c\`adl\`ag functions on $(0,\infty)$ taking values in $\R^{d}\times \R$. 
% We will equip this space with the Skorokhod topology on compact subintervals of~$(0,\infty)$.

%\newpage
The derived processes in Theorem~\ref{spatial_limit_u} can be described as follows: % (see also Figure~\ref{limit_process}):
\begin{itemize}
\item $((Y^{\ssup 2}_t+\frac{d}{\alpha-d} |Y^{\ssup 1}_t|\big)\colon t  > 0)$ corresponds to the \emph{vertical distance} 
of the point $(Y^{\ssup 1}_t, Y^{\ssup 2}_t)$ to the boundary of the domain given by the curve~$y = -\frac{d}{\alpha-d}|x|$;
\item $((Y_t^{\ssup 2}+(1-\tfrac{1}{t})|Y_t^{\ssup 1}|) \colon t>0)$ corresponds to the \emph{$y$-coordinate of the tip} of the cone~$\calC_t$. 
\end{itemize}
%\smallskip

%\begin{figure}[htbp]
%\centering 
%\subfigure[The time evolution for $t>1$.]{ 	
%		\includegraphics[width=6.5cm]{time_evolution_of_Y.eps} 
%		\label{fig_time_evolution}  }
%\subfigure[The trace of the limiting process.]{ 	
%		\includegraphics[width=6.5cm]{trace_of_limit.eps} 
%		\label{fig_trace}  }
%\caption[The time evolution of $Y_t$.]{(a) The time evolution for $t > 1$ of the jump process $Y_t = (Y^{\ssup 1}_t,Y^{\ssup 2}_t)$ (large dots), %defined in terms of the Poisson point process (small dots). (b) The trace of the limiting process $(Y^{\ssup 1}_t,Y^{\ssup 2}_t + %\frac{d}{\alpha-d}(1-\frac{1}{t})|Y_t^{\ssup 1}|)$. The second component corresponds to the $y$-coordinate of the tip of the defining cone.
%}\label{limit_process}
%\end{figure}
% Maybe replace with two pictures showing the derived processes more clearly.

\begin{rem}\emph{Time evolution of the process.}\\ %[-6mm]
% \begin{itemize} \item[(i)] 
$(Y^{\ssup 1}_1,Y^{\ssup 2}_1)$ is the `highest' point of the Poisson point process $\Pi$.
%\item[(ii)] 
Given~$(Y^{\ssup 1}_t,Y^{\ssup 2}_t)$ and $s\ge t$ we consider the surface given by all $(x,y) \in \R^{d}\times\R$ such that 
$$y = Y_t^{\ssup 2} - \tfrac{d}{\alpha-d}\big(1-\tfrac{1}{s}\big)(|x| - |Y_t^{\ssup 1}|) \, .$$
For $s=t$ there are no points of $\Pi$ above this surface, while $(Y^{\ssup 1}_{t},Y^{\ssup 2}_{t})$ (and possibly
one further point) is lying on it.  We now increase the parameter $s$ until the surface hits a further point of $\Pi$. At this 
time~$s>t$ the process jumps to this new point~$(Y^{\ssup 1}_{s},Y^{\ssup 2}_{s})$.
Geometrically, increasing~$s$ means opening the cone further, while keeping the point~$(Y^{\ssup 1}_t,Y^{\ssup 2}_t)$ on the boundary
and moving the tip upwards on the $y$-axis. % see Figure~\ref{fig_time_evolution}. % maybe replace figure
%\item[(iii)]
Similarly, given the point~$(Y^{\ssup 1}_t,Y^{\ssup 2}_t)$ one can go backwards in time 
by decreasing~$s$, or equivalently closing the cone and moving the tip downwards on the $y$-axis.
%, see also Figure~\ref{fig_time_evolution}.
The general independence properties of Poisson processes ensure that this procedure yields a process
$((Y^{\ssup 1}_t,Y^{\ssup 2}_t) \colon t>0)$ which is Markovian in both the forward and backward direction.
%\item[(iv)] 
The process $(Y^{\ssup 2}_t+\tfrac{d}{\alpha-d}(1-\frac1t)|Y^{\ssup 1}_t| \colon t >0 )$ is continuous, which can be seen 
directly from its interpretation\vspace{-1mm} as the $y$-coordinate of the tip of the cone. % see Figure~\ref{fig_trace}.
An animation of the process $((Y^{\ssup 1}_t,Y^{\ssup 2}_t) \colon t>0)$ can be found on the second author's homepage at 
{\texttt http://people.bath.ac.uk/maspm/animation\_ ageing.pdf}.
%Marcel might want to create a separate movie for his/our homepages.
%\end{itemize}
\end{rem}

\subsection{Strategy of the proofs and overview}\label{se.guide}

Let us first collect some of the key ingredients common to the proofs of our three main results.
It is shown in~\cite{KLMS09} that, almost surely, for all large~$t$ the total mass $U(t)$ can be 
approximated by a variational problem. More precisely,
\begin{equation}\label{key}
\frac{1}{t} \log U(t) \sim \max_{z \in \Z^d} \Phi_t(z) \, , 
\end{equation}
where, for any $t\ge 0$, the functional $\Phi_t$ is defined as
$$\Phi_t(z) =  \xi(z) - \frac{|z|}{t} \log \xi(z) + \frac{\eta(z)}{t},$$ 
for $z\in\Z^d$ with $t\xi(z)\geq |z|$, and $\Phi_t(z)=0$ for other values of~$z$. Here 
$\eta(z)$ is the logarithm of the number of paths of 
length $|z|$ leading from $0$ to $z$.

%\cite{KLMS09} continue their analysis by showing that the peak $X_t$ of the profile agrees for most 
%times~$t$ with the maximizer $Z_t$ of the functional~$\Phi_t$. 
Furthermore,~\cite{KLMS09} show that the peaks $X_t$ agree for most times~$t$ with 
the maximizer $Z_t$ of the functional~$\Phi_t$.
This maximizer is uniquely defined, 
if we impose the condition that $t \mapsto Z_t$ is right-continuous.
Defining the two scaling functions
$$r_t = \big(\sfrac{t}{\log t}\big)^{\frac{\alpha}{\alpha-d}} \quad \mbox{and} \quad 
a_t = \big(\tfrac{t}{\log t}\big)^{\frac{d}{\alpha-d}},$$
it is shown in~\cite{KLMS09}, refining the argument of~\cite{HMS08}, that, 
as $t \ra \infty$, the point process
\begin{equation}\label{pitdef}
\Pi_t= \sum_{\heap{z \in \Z^d}{t \xi(z) \geq |z|}} \delta_{(\frac{z}{r_t}, \frac{\Phi_t(z)}{a_t})}
\end{equation}
converges (in a suitable sense) to the Poisson point process~$\Pi$ on~$H_0$ defined above.
\pagebreak[3]

Section~\ref{se:weak_ageing} is devoted to the proof of the `annealed' ageing result, Theorem~\ref{ageing_for_u}.
We show in Section~\ref{ageing_for_solution}, see Lemma~\ref{weak_ageing}, that
$$\begin{aligned}
\lim_{t \ra \infty} \Prob & \Big\{ \sup_{z \in \R^d} \sup_{s \in [t, t+t\theta ]} 
\big| {v(t,z)} - {v(s,z)} \big| < \eps \Big\}\\
& = \lim_{t \ra \infty} \Prob\big\{ Z_t=Z_{t+t\theta} \big\}.
\end{aligned}$$
Therefore we begin this proof, in Section~\ref{ageing_for_variational_maximizer}, by discussing 
the limit on the right hand side. To this end we approximate the probability %on the right hand side 
in terms of the point process~$\Pi_t$. We are able to write
\begin{equation}\label{appO}
\frac{\Phi_{t+\theta t}(z)}{a_t}= \frac{\Phi_{t}(z)}{a_t} + \frac{\theta}{1+\theta}\,\frac{d}{\alpha-d}\, \frac{|z|}{r_t} + 
{\rm error},
\end{equation}
where the error can be suitably controlled, see Lemma~\ref{Phi_t(1+c)_in_Phi_t}. Hence (in symbolic notation)
$$\begin{aligned}
\Prob& \big\{ Z_t=Z_{t+t\theta} \big\}\\
& \approx \iint \Prob\Big\{ \Pi_t(\d x \,  \d y)>0, \Pi_t\{ (\bar{x},\bar{y}) \colon  \bar{y} > y\}=0,\\ 
& \qquad \qquad\qquad\qquad  \Pi_t\{ (\bar{x},\bar{y}) \colon |\bar{x}| > |x| \mbox{ and }
\bar{y} > y - \tfrac{d}{\alpha - d}\tfrac{\theta}{1+\theta} (|\bar{x}| - |x|) \}=0 \Big\},
\end{aligned}$$
where the first line of conditions on the right means that~$x$ is a maximizer of $\Phi_t$
with maximum~$y$, and the second line means that $x$ is also a maximizer of~$\Phi_{t+\theta t}$.
As $t\uparrow\infty$ the point process~$\Pi_t$ is replaced by~$\Pi$ and we can evaluate the probability.
\medskip

Section~\ref{se:almost_sure_ageing} is devoted to the `quenched' ageing result, Theorem~\ref{asymptotics_R_X}.
This proof is technically more involved, because we cannot exploit the point process approach and have to do
significant parts of the argument from first principles. We now have to consider events 
$$\Prob\Big\{ \frac{R(t)}{t} \ge \theta_t\Big\} \approx \Prob\big\{ Z_t=Z_{t+t\theta_t}\big\},$$
for $\theta_t\uparrow\infty$. We have to significantly refine the argument above and 
replace the convergence of $\Prob\{Z_t=Z_{t+t\theta}\}$ by a moderate deviation statement, see
Section~\ref{sect_moderate_deviations}. Indeed, 
for $\theta_t\uparrow \infty$ not too fast we show that
$$\Prob\big\{ Z_t=Z_{t+t\theta_t} \big\} %\sim I(\theta_t) 
\sim C\,\theta_t^{-d},$$
for a suitable constant~$C>0$, see Proposition~\ref{moderate_deviations}. Then, if $\varphi(t) = t h(t)$, this allows us to show 
in Sections~\ref{asymptotics_tau} and~\ref{sect_asymptotics_sigma} that,
for any $\eps>0$, the series $\sum_n \Prob\{ R(e^n) \ge \eps \varphi(e^n)\}$ converges if $\sum_n h(e^n)^{-d}$ converges, which
is essentially equivalent to $\int h(t)^{-d} \d t/t<\infty$. By Borel-Cantelli we get %, %for any such function~$h$, 
that\vspace{-2mm}
$$\limsup_{n\to\infty} \frac{R(e^n)}{\varphi(e^n)}=0,$$
which implies the upper bound in Theorem~\ref{asymptotics_R_X}, and the lower bound follows similarly 
using a slightly more delicate second moment estimate, see Lemma~\ref{asymptotic_independence}.
\medskip

The proofs of the scaling limit theorems, Proposition~\ref{classical} and Theorem~\ref{spatial_limit_u} 
are given in Section~\ref{sect_spatial_limit_theorem}. By~\eqref{appO} we can describe~$Z_{tT}$
approximately as the maximizer of 
$$\frac{\Phi_{T}(z)}{a_T}+ \frac{d}{\alpha-d}\, \Big(1-\frac1t\Big) \,\frac{|z|}{r_T}.$$
Instead of attacking the proof of Theorem~\ref{spatial_limit_u} directly, we first show
in Sections~\ref{sect_finite_diml} and~\ref{sect_tightness} a limit theorem for 
\be{real_process} \Big( \big( \tfrac{Z_{tT}}{r_T} , \tfrac{\Phi_{tT}(Z_{tT})}{a_T} \big) \, : \, t > 0 \Big) \, , \ee
see Proposition~\ref{spatial_limit}. Informally, we obtain
$$\begin{aligned}
P & \big\{ \tfrac{Z_{tT}}{r_T} \in A, \, \tfrac{\Phi_{tT}(Z_{tT})}{a_T} \in B \big\}\\
& \approx \iint\limits_{\substack{{x\in A}, \\ {y + q(1-\frac{1}{t})|x| \in B}}}  \Prob\Big\{ \Pi_T(\d x \,  \d y)>0,\,
%\\ & \qquad\qquad \qquad\qquad\qquad\qquad 
 \Pi_T\big\{ (\bar{x},\bar{y}) \colon 
\bar{y} - y  > \tfrac{d}{\alpha - d}\, \big( 1-\tfrac1t\big) \,(|{x}| - |\bar{x}|) \big\}=0 \Big\},
\end{aligned}\pagebreak[3]$$
where the first line of conditions on the right means that there is a site $z\in\Z^d$ such that
$x=z/r_T\in A$ and $y=\Phi_T(z)/a_T \in B - q(1-\frac{1}{t})|x|$ , and the second line means that~$\Phi_{tT}(z)$ is not 
surpassed by~$\Phi_{tT}(\bar{z})$ for any other site~$\bar{z}\in\Z^d$ with $\bar{x}=\bar{z}/r_T$.
We can then use the convergence of~$\Pi_T$ to~$\Pi$ inside the formula to 
give a limit theorem for the one-dimensional distributions of~(\ref{real_process}). A minor
strengthening of this argument given in Section~\ref{sect_finite_diml} shows convergence of the finite dimensional
distributions, see Lemma~\ref{finite_dimensional_distributions}. 
In Section~\ref{sect_tightness} we check a tightness criterion in Skorokhod space, see Lemma~\ref{Prob_T_tight}, 
and thus complete the proof of the convergence
\[ \Big( \big( \tfrac{Z_{tT}}{r_T} , \tfrac{\Phi_{tT}(Z_{tT})}{a_T} \big) \, : \, t > 0 \Big) \weakconv 
\Big( \big(Y_t^{\ssup 1}, Y_t^{\ssup 2} + \tfrac{d}{\alpha-d}(1-\tfrac{1}{t}) |Y_t^{\ssup 1}| \big) \, : \, t > 0 \Big)\, . \]
Based on this result we complete the proof of the scaling limit results in Section~\ref{transfer_spatial_limit}.
Theorem~\ref{spatial_limit_u}\,(b) follows using~\eqref{key} and projecting on the second component.
Observe that the convergence in~(b) automatically holds in the uniform sense, as all involved processes are continuous.
We note further that % for $z\in\Z^d$ such that $z/r_T$ is bounded away from zero and infinity,
$$\frac{\xi(z)}{a_T} = \frac{\Phi_T(z)}{a_T} + \frac{d}{\alpha-d} \, \frac{|z|}{r_T} + {\rm error},$$
see Lemma~\ref{superappro}.
This allows us to deduce Theorem~\ref{spatial_limit_u}\,(a), and Proposition~\ref{classical} is an easy consequence 
of this. %Theorem~\ref{spatial_limit_u}\,(a).
%approximate the events of interest 
%associated with the random variables
%$$\Big(\frac{Z_{tT}}{r_T},\frac{\Phi_{tT}(Z_{tT})}{a_T}\Big),$$
%with events involving only the point process~$\Pi_T$. 
%Informally, we obtain
%$$\begin{aligned}
%P & \big\{ \tfrac{Z_{tT}}{r_T} \in A, \, \tfrac{\xi(Z_{tT})}{a_T} \in B \big\}\\
%& \approx \iint_{{x\in A}, {y+\frac{d}{\alpha-d}|x| \in B}} P\Big\{ \Pi_T(\d x \, , \d y)>0,\\ 
%& \qquad\qquad \qquad\qquad\qquad\qquad  \Pi_T\big\{ (\bar{x},\bar{y}) \colon 
%\bar{y} - y  > \tfrac{d}{\alpha - d}\, \big( 1-\tfrac1t\big) \,(|{x}| - |\bar{x}|) \big\}=0 \Big\},
%\end{aligned}$$
%where the first line of conditions on the right means that there is a site $z\in\Z^d$ such that
%$x=z/r_T\in A$ and $\xi(z)/a_T\in B$ with $y=\Phi_T(z)/a_T$ , and the second line means that~$\Phi_{tT}(z)$ is not 
%surpassed by~$\Phi_{tT}(\bar{z})$ for any other site~$\bar{z}\in\Z^d$ with $\bar{x}=\bar{z}/r_T$.
%We can then use the convergence of~$\Pi_T$ to~$\Pi$ inside the formula to 
%give a limit theorem for the one-dimensional distributions of
%$$\Big(\big(\tfrac{Z_{tT}}{r_T},	\tfrac{\xi(Z_{tT})}{a_T}\big)\, : \, t > 0\Big),$$
%and analogous reasoning leads to the convergence of the finite-dimensional distributions.

\section{Ageing: a weak limit theorem}\label{se:weak_ageing}

This section is devoted to the proof of Theorem~\ref{ageing_for_u}. In Section~\ref{ageing_for_variational_maximizer}
we show ageing for the two point function of the process $(Z_t \colon t\ge 0)$ of maximizers of the variational problem~$\Phi_t$, 
using the point process approach which was developed in~\cite{HMS08} and extended in~\cite{KLMS09}. In Section~\ref{ageing_for_solution} 
we use this and the localization of the profile in $Z_t$ to complete the proof.

\subsection{Ageing for the maximizer of $\Phi_t$}\label{ageing_for_variational_maximizer}

In this section, we prove ageing for the two point function of the process $(Z_t \colon t\ge 0)$, which
from now on is chosen to be left-continuous. The value $I(\theta)$ will be given by the formula 
in Proposition~\ref{I} below.

\begin{prop}\label{ageing_for_Z} Let $\theta >0$, then
$\displaystyle\lim_{t \ra \infty} \Prob \big\{ Z_t = Z_{t+\theta t} \big\} = I(\theta) \in (0,1).$
% where $I(\theta)\in(0,1)$ is given by the formula in Proposition~\ref{I} below.
\end{prop}

Throughout the proofs we use the abbreviation 
$$q=\frac{d}{\alpha-d}\, . $$ For any $t >0$ consider the point process~$\Pi_t$ on $\R^d \times \R$ 
defined in \eqref{pitdef}.
%given by
%$$\Pi_t = \sum_{\heap{z \in \Z^d}{t \xi(z) \geq |z|}} \delta_{\big(\frac{z}{r_t}, \frac{\Phi_t(z)}{a_t} \big)} \, , $$
%We can restrict the sum to points $z \in \Z^d$ such that $t \xi(z) \geq |z|$ because, 
%as we will see later, these are the only relevant points for our calculation in the large $t$ limit.
Define a locally compact Borel set
$$\widehat{H} = \dot{\R}^{d+1} \setminus \big( \{ (x,y) \in \R^d \times \R \colon y < -q(1-\eps) |x| \} \cup \{ 0\} \big) \, , $$
where $0< \eps < \frac{1}{1+\theta}$ and $\dot{\R}^{d+1}$ is the one-point compactification of $\R^{d+1}$.
%\dot{\R}^{d+1} \setminus \big( (\R^d \times (-\infty,0) ) \cup \{ (0,0)\} \Big) \, . \]
As in Lemma~6.1\vspace{-1mm} 
of~\cite{KLMS09} one can show that the point process $\Pi_t$ restricted to the domain $\widehat H$ 
converges in law to a Poisson process $\Pi$ on $\widehat{H}$ with intensity measure
\be{definition_of_Pi} \nu(\diff x \, \diff y) =  \frac{\alpha\, \diff x \,\diff y}{(y + q |x|)^{\alpha + 1}}  
  \, . \ee
Here, $\Pi_t$ and $\Pi$ are random elements of the set of point measures on $\widehat H$, which is given the topology of vague convergence. For more background on point processes and similar arguments, see~\cite{HMS08}.

Our strategy is to express the condition $Z_t = Z_{t+\theta t}$ in terms of the point process $\Pi_t$. 
In order to be able to bound error functions that appear in our calculations, we have to restrict 
our attention to the point process $\Pi$ on a large box. To this end, define the two boxes 
\[ \bal B_N & = \{ (x,y) \in \R^d \times [0,\infty) \, : \, |x| \leq N , \, \tfrac{1}{N} \leq y \leq N \} \, , \\
 \widehat{B}_N & = \{ ( x, y) \in \widehat{H} \, : \, | x| \leq N , y \leq N \} \, . \eal \]
Now note that the condition $Z_t = Z_{t+\theta t}$ means that 
\be{c1} \Phi_{t+\theta t} (z) \leq \Phi_{t+\theta t}(Z_t) \, , \ee 
for all $z \in \Z^d$. We now show that it suffices to guarantee that this condition holds 
for all $z$ in a sufficiently large bounded box. % (growing in time).

\begin{lemma}\label{restriction_to_box} Define the event
\[  A(N,t) = \Big\{ \big(\tfrac{Z_t}{r_t},  \tfrac{\Phi_t(Z_t)}{a_t}\big) \in B_N\, ,  \Phi_{t+\theta t} (z)  \leq \Phi_{t+\theta t}(Z_t) \, \forall z  \in \Z^d \, \textrm{s.t.} \, \big(\tfrac{|z|}{r_t},\tfrac{\Phi_t(z)}{a_t}\big) \in \widehat{B}_N \Big\} \, .  \]
Then, provided the limit on the right-hand side exists, we find that
\[ \lim_{t \ra \infty} \Prob\{ Z_t = Z_{t+\theta t} \} = \lim_{N \ra \infty} \lim_{t \ra \infty} \Prob(A(N,t)) \, . \]
\end{lemma}

\begin{proof} We have the lower bound,
\[ \bal \Prob \{ Z_t = Z_{t+\theta t} \}  &\geq \Prob \big\{ Z_t = Z_{t+\theta t} \, , 
\, \big(\tfrac{Z_t}{r_t},  \tfrac{\Phi_t(Z_t)}{a_t}\big) \in B_N \big\} \\
& \geq \Prob(A(N,t)) -  \Prob\big\{ \tfrac{|Z_{t+\theta t}|}{r_t} > N \big\}  
%+ \Prob\big\{ \tfrac{\Phi_{t}(Z_{t+\theta t})}{a_t} > N \big\}
%\Big)
\, . 
%\\ & \geq \Prob(A(N,t)) - \Big( \Prob\big\{ \tfrac{|Z_{t+\theta t}|}{r_t} > N \big\}  
%+ \Prob\big\{ \tfrac{\Phi_{t+\theta t}(Z_{t +\theta t})}{a_t} > N \big\} 
%\\ & \hspace{6.6cm} 
%+ \Prob\{ \xi(Z_{t+\theta t}) \leq \log d \} \Big) \, . 
\eal \]
%If $\frac{\Phi_{t}(Z_{t +\theta t})}{a_t} > N$, then by definition $t \xi(Z_{t+\theta t}) \geq |Z_{t+\theta t}|$. 
%For any $z$ such that $t \xi(z) \geq |z|$ the map $t \mapsto \Phi_t(z)$ is increasing if $\xi(z) > \log d$, 
%since $\eta(z) \leq |z| \log d$, and therefore
%$$\Prob\big\{ \tfrac{\Phi_{t}(Z_{t+\theta t})}{a_t} > N \big\} \le
%\Prob\big\{ \tfrac{\Phi_{t+\theta t}(Z_{t +\theta t})}{a_t} > N \big\} + 
%\Prob\{ \xi(Z_{t+\theta t}) \leq \log d \}.$$
Recall that, by~\cite[Lemma 6.2]{KLMS09}, we have that
\be{convergence_Z_Phi_of_Z} \big( \tfrac{Z_t}{r_t}, \tfrac{\Phi_t(Z_t)}{a_t} \big) \weakconv (Y^{\ssup 1}, Y^{\ssup 2}) \, , \ee
where $(Y^{\ssup 1}, Y^{\ssup 2})$ is a random variable on $\R^d \times [0,\infty)$  with an explicit density. In particular, 
we find that since $r_{t+\theta t} = (1+\theta)^{q+1} r_t (1+o(1))$ 
\[ \lim_{t \ra \infty} \Prob\big\{ \tfrac{|Z_{t+\theta t}|}{r_t} > N \big\} = 
\Prob \big\{ |Y^{\ssup 1}| > \tfrac{N}{(1+\theta)^{q+1}} \big\} \, ,  \]
which converges to zero as $N \ra \infty$. 
%Similarly, as $a_{t+\theta t} = (1+\theta)^{q} a_t (1+o(1))$, we get 
%\[ \lim_{t \ra \infty} \Prob\big\{ \tfrac{\Phi_{t+\theta t}(Z_{t+\theta t})}{a_t} > N \big\} 
%= \Prob \big\{ |Y^{\ssup 2}| > \tfrac{N}{(1+\theta)^{q}} \big\} \, \stackrel{N\to\infty}{\longrightarrow} 0.  \]
%Since, by Lemma 3.2(i) in ~\cite{KLMS09}, we have for any $\eps>0$, eventually for all $t$ that 
%$\xi(Z_t) > a_t (\log t)^{-\eps}$, the probability $\Prob\{ \xi(Z_{t+\theta t}) \leq \log d \}$ also tends to 
%zero as $t \ra \infty$.

Now, for an upper bound on $\Prob\{ Z_t = Z_{t(1+\theta)} \}$ we find that
\[  \Prob\{ Z_t = Z_{t(1+\theta)} \} 
%\\ & \leq \Prob \Big\{ Z_t = Z_{t(1+\theta)}, \Big(\frac{Z_t}{r_t},  \frac{\Phi_t(Z_t)}{a_t}\Big) \in B_N \Big\} + \Prob\Big\{ \frac{|Z_t|}{r_t} > N\} + \Prob\Big\{ \frac{1}{N} \leq \frac{\Phi_t(Z_t)}{a_t} \leq N\Big \}\\
 \leq \Prob (A(N,t)) + \Prob\big\{ \tfrac{|Z_t|}{r_t} \geq N\big\} + \Big(1-\Prob\big\{ \tfrac{1}{N} \leq \tfrac{\Phi_t(Z_t)}{a_t} \leq N\big\}
 \Big) \, . \]
As above, using the convergence~(\ref{convergence_Z_Phi_of_Z})  one can show that the limit of the last two summands is 
zero when taking first $t \ra \infty$ and then $N \ra \infty$, which completes the proof of the lemma.
%. This completes the proof of the claim that
%\[ \lim_{t \ra \infty} \Prob\{ Z_t = Z_{t(1+\theta)} \} = \lim_{N \ra \infty} \lim_{t \ra \infty} \Prob(A(N,t)) \, , \]
%provided the latter limit exists.
\end{proof}

We would like to translate the condition~(\ref{c1}) into a condition on the point process $\Pi_t$. Therefore, we have to express $\Phi_{t+\theta t}(z)$ in terms of $\Phi_t(z)$. 

\begin{lemma}\label{Phi_t(1+c)_in_Phi_t} For any $z\in\Z^d$ such that $(\frac{z}{r_t},\frac{\Phi_t(z)}{a_t}) \in \widehat{B}_N$ 
and $t \xi(z) \geq |z|$, 
\[ \frac{\Phi_{t+\theta t}(z)}{a_t} =  \frac{\Phi_t(z)}{a_t} + \frac{q\theta}{1+\theta}\frac{|z|}{r_t} + \delta_\theta
\big(t,\tfrac{|z|}{r_t}, \tfrac{\Phi_t(z)}{ a_t}\big) \, , \]
where the error $\delta_\theta$ converges to zero as $t \ra \infty$ uniformly. % for all such $z$. 
Moreover, almost surely, eventually for all large enough $t$, for all $z\in\Z^d$ such that  
$(\frac{z}{r_t},\frac{\Phi_t(z)}{a_t})\in \widehat B_N$ and $t \xi(z) < |z|$, we have that 
$\Phi_{t+\theta t}(z) \leq 0$, and such a $z\in\Z^d$ will automatically satisfy~(\ref{c1}).
\end{lemma}

\begin{proof}
Consider any $z$ such that $(\frac{z}{r_t},\frac{\Phi_t(z)}{a_t}) \in \widehat{B}_N$ and $t \xi(z) \geq |z|$.
%, as we will see this latter condition is satisfied by all the relevant points $z \in \widehat{B}_N$. 
 %\1{\Big\{ (1+\theta)\frac{\xi(z)}{a_t} \geq (\log t)^{-1} \frac{|z|}{r_t} \Big\}}
Then, using that $r_t = \frac{t}{\log t} \, a_t$  we obtain
\be{Phi_t(1+c)} \begin{aligned} \frac{\Phi_{t+\theta t}(z)}{a_t} & =  \frac{\xi(z)}{a_t} - \frac{1}{a_t t+\theta t}\big(|z| \log \xi(z) - \eta(z)\big)   \\ 
& =  \frac{\Phi_t(z)}{a_t} + \frac{\theta}{1+\theta} \frac{|z|}{r_t \log t}\log a_t + \frac{\theta}{1+\theta}\Big(\frac{|z|}{r_t \log t} \log \frac{\xi(z)}{a_t} - \frac{\eta(z)}{t a_t}\Big)  \\
& = \frac{\Phi_t(z)}{a_t} + \frac{\theta q}{1+\theta}\frac{|z|}{r_t} + \delta_\theta'\big(t,\tfrac{z}{r_t}, \tfrac{\xi(z)}{ a_t}\big) \, ,
\end{aligned} \ee
where using that $\log a_t = (q+o(1)) \log t$ and $0\le \eta(z) \le |z| \log d$, we can write 
\be{error_term} \delta_\theta'\big(t,\tfrac{z}{r_t}, \tfrac{\xi(z)}{ a_t}\big) = \frac{\theta}{1+\theta}\Big( \frac{|z|}{r_t\log t} \log \frac{\xi(z)}{a_t} + o(1)\frac{|z|}{r_t} \Big) \, . \ee
First of all, we have to show that this expression is of the form $\delta_\theta(t,z/r_t,\Phi_t(z)/a_t)$ for some suitable error function. With this in mind, using that $a_t t = r_t \log t$, we obtain for $z$ such that $t \xi(z) \geq |z|$ 
\[ \bal \frac{\Phi_t(z)}{a_t} & = \frac{\xi(z)}{a_t} - \frac{|z|}{r_t \log t} \log \xi(z) + \frac{\eta(z)}{a_t t} \\
& = \frac{\xi(z)}{a_t} - (q + o(1)) \frac{|z|}{r_t} - \frac{|z|}{r_t \log t} \log \frac{\xi(z)}{a_t} + \frac{\eta(z)}{a_t t} \\
& = \chi_\rho\big(\tfrac{\xi(z)}{a_t}\big) - (q + o(1)) \frac{|z|}{r_t} ,\\   %+ \frac{\eta(z)}{a_t t} 
\eal \, , \]
where $\chi_\rho(x) = x - \rho \log x$ and $\rho = \frac{|z|}{r_t \log t}$. Note that $\chi_\rho$ is strictly increasing on $[\rho,\infty)$ and also that %since $r_t = \frac{t}{\log t} a_t$ we find that 
$\xi(z)/a_t > \rho$ is equivalent to $t \xi(z) > |z|$ which is satisfied by assumption. Therefore, we can write
\[ \frac{\xi(z)}{a_t} = \chi_\rho^{-1} \big(  \tfrac{\Phi_t(z)}{a_t} + (q + o(1)) \tfrac{|z|}{r_t}\big)\, , \] % - \tfrac{\eta(z)}{a_t t} 
and obtain that the error in~(\ref{error_term}) is of the required form
\be{error_in_required_form}  \bal \delta_\theta'\big(t,\tfrac{z}{r_t}, \tfrac{\xi(z)}{ a_t}\big) & = \frac{\theta}{1+\theta}\Big( \frac{|z|}{r_t\log t} \log \chi_\rho^{-1} \big(  \tfrac{\Phi_t(z)}{a_t} + (q + o(1)) \tfrac{|z|}{r_t} \big)  + o(1)\frac{|z|}{r_t} \Big) \\
&  =:  \delta_\theta\big(t,\tfrac{z}{r_t}, \tfrac{\Phi_t(z)}{ a_t}\big) \, . \eal \ee
We now show that this error tends to zero uniformly for all $z$ satisfying $t\xi(z) > |z|$ and 
$(\frac{z}{r_t},\frac{\Phi_t(z)}{a_t}) \in \widehat B_N$. For a lower bound we first use that
%On the one hand we can use that 
%$\frac{\eta(z)}{t a_t} \leq \frac{|z|}{r_t} \frac{\log d}{\log t}$, and also that 
$x\log x \geq -e^{-1}$ to obtain %since $t\xi(z) > |z|$
\begin{align*}
\frac{|z|}{r_t\log t} \log & 
\chi_\rho^{-1} \big(  \tfrac{\Phi_t(z)}{a_t} + (q + o(1)) \tfrac{|z|}{r_t}\big) \\
& \geq \frac{|z|}{r_t \log t}\log \frac{|z|}{r_t \log t} \geq -\frac{1}{\log t} \,e^{-1} - \frac{\log \log t}{\log t} \frac{|z|}{r_t} 
\geq -\frac{1}{\log t} \,e^{-1} - \frac{\log \log t}{\log t} N  \, . 
\end{align*}
To bound the expression in~(\ref{error_in_required_form}) from above note 
that $\rho = \frac{|z|}{r_t \log t} \leq \frac{N}{\log t}$ and we can thus assume that $\rho<1$, which implies that for $x > 1$ we find $\chi_1(x) \leq \chi_\rho(x)$. % so that $x \leq \chi_1^{-1} \chi_\rho(x)$. 
Hence, either
\[  \chi_\rho^{-1} \big(  \tfrac{\Phi_t(z)}{a_t} + (q + o(1)) \, \tfrac{|z|}{r_t} \big) \leq 1 \, , \]
or we can estimate
\[\chi_\rho^{-1} \big(  \tfrac{\Phi_t(z)}{a_t} + (q + o(1)) \,\tfrac{|z|}{r_t}\big) \leq 
\chi_1^{-1} \big(  \tfrac{\Phi_t(z)}{a_t} + (q + o(1)) \,\tfrac{|z|}{r_t}\big)
\leq \chi_1^{-1} \big(  (N(1 + 2q) \big) \, . \]
which completes the proof of the first part of the lemma.

For the second part, recall that for all $t>0$ we have $\Phi_t(Z_t) > 0$, since $\Phi_t(0) > 0$.
%see e.g.~\cite[Lemma 3.2]{KLMS09}. 
Suppose $t \xi(z) < |z|$, then $\Phi_t(z) = 0$ and hence $z \neq Z_t$.
We want to show that $\Phi_{t+\theta t}(z) \leq 0$ which ensures that $z$ satisfies~(\ref{c1}).
Indeed, if $(t+\theta t) \xi(z) < |z|$, then this is true as $\Phi_{t+\theta t}(z) = 0$, and otherwise
%but if $(t+\theta t) \xi(z) \geq |z|$, then 
we can estimate as above that 
\[ \frac{\Phi_{t+\theta t}(z)}{a_t} = \frac{\xi(z)}{a_t} - \frac{q\theta}{1+\theta} \frac{|z|}{r_t} + \tilde{\delta}_\theta\big(t,\tfrac{|z|}{a_t}, \tfrac{\Phi_t(z)}{a_t}\big) \, , \]
where $\tilde{\delta}_\theta(t,x,y)$ converges to zero 
%as $t \ra \infty$ 
uniformly in $(x,y) \in \widehat B_N$. In particular, %since $t \xi(z) < |z|$ 
it follows that
\[ \frac{\Phi_{t+\theta t}(z)}{a_t} \leq \Big( - \frac{q\theta}{1+\theta} + \frac1{\log t}\Big)\, 
\frac{|z|}{r_t} + \tilde{\delta}_\theta\big(t,\tfrac{|z|}{a_t}, \tfrac{\Phi_t(z)}{a_t}\big) \, , \]
which is negative for all $t$ large enough, uniformly for all $z$ such that  $(\frac{z}{r_t},\frac{\Phi_t(z)}{a_t}) \in B_N$.
\end{proof}

%According to Lemma~\ref{restriction_to_box}, we need to first 
We now calculate $\Prob(A(t,N))$ in the limit as $t \ra \infty$, i.e.
we are interested in
\[ \underset{{(x,y) \in B_N}}{\int\!\!\int}  \Prob \Big\{ \tfrac{Z_t}{r_t} \in \diff x, \tfrac{\Phi_t(Z_t)}{a_t} \in \diff y, 
\Phi_{t+\theta t}(z) \leq \Phi_{t+\theta t}(Z_t) \, \forall z\in\Z^d \mbox{ s.t. } 
\big(\tfrac{|z|}{r_t}, \tfrac{\Phi_t(z)}{a_t}\big) \in \widehat B_N \Big\}
 \, . \]
%Before, we continue we need to clarify what we mean by the above notation. We write
%\[ \int\int \Prob \{ X \in \diff x, Y \in \diff y, (x,y) \in A \} \, , \]
%instead of 
%\[ \int\int \Prob\{ (x,y) \in A | (X,Y) = (x,y) \} \Prob_{X,Y} (\diff x \, \diff y) \,, \]
%where $\Prob_{X,Y}$ is the distribution of $X,Y$ and $\Prob\{ (X,Y) \in A | (X,Y) = (x,y) \}$ is a regular conditional probability as defined in~\cite{Br68}. Since, $(X,Y)$ is $\R^{d+1}$ valued, the regular conditional probability always exists, see~\cite[Theorem 4.43]{Br68}.
 
First, we express the probability under the integral for fixed $(x,y) \in B_N$ in terms of 
the point process~$\Pi_t$. Given that $\Pi_t$ contains the point $(x,y)$ we require that  
%The condition $Z_t/r_t \in \diff x, \Phi_t(Z_t)/a_t \in \diff y$ means that $\Pi_t$ should have one point in the set $\diff x \times \diff y$ and 
there are no points in the set $\R^d \times (y,\infty)$, and
requiring~\eqref{c1} for all points $z$ with $({|z|}/{r_t}, {\Phi_t(z)}/{a_t}) \in \widehat B_N$ 
is, by Lemma~\ref{Phi_t(1+c)_in_Phi_t}, equivalent to 
%the requirement that in the limit all points $(\bar x,\bar y)$ of $\Pi$ 
%restricted to $\widehat B_N$ satisfy 
%~\cite[Lemma 3.7]{HMS08}
%\begin{align*}
%\bar y+\frac{q \theta}{1+\theta}|\bar x|\le y+\frac{q \theta}{1+\theta}|x|.
%\end{align*}
%In other words,
the requirement that $\Pi_t$ should have no points in the set 
$$\big\{(\bar x,\bar y)\in \widehat B_N  \colon \bar y+\tfrac{q\theta}{1+\theta}|\bar x|> y+\tfrac{q\theta}{1+\theta}|x|\big\}.$$
Hence, defining the set 
\[ D^N_\theta(r,y) = \big\{ (\ox, \oy) \in \R^d \times \R \, : \, \oy > y \big\} \cup 
\big\{ (\ox, \oy) \in \widehat B_N \, :\,  |\ox| > r, \oy > y - \tfrac{q\theta}{1+\theta} (|\ox| - r ) \big\} \, , \]
we see that, as $t\ra \infty$,
\[ \begin{aligned}
\lim_{t\ra\infty} \Prob(A(N,t)) &= \underset{{(x,y) \in B_N}}{\int\!\!\int}
\Prob\big\{\Pi(\d x\, \d y)=1, \Pi(D_\theta^N(|x|,y))=0\big\}\notag\\
&=\underset{{(x,y) \in B_N}}{\int\!\!\int}e^{-\nu(D^N_\theta(|x|,y))}\nu(\d x\,\d y) \, . 
%\label{c2}
\end{aligned} \]
Taking the limit in this way is justified as $D^N_\theta(|x|,y)$ is relatively compact in $\widehat H$ and $(x,y)$ ranges only over elements in $B_N$. 
Finally, if we similarly define (see also Figure~\ref{point_process})
\[ D_\theta(r,y) = \big\{ (\ox, \oy) \in \R^d \times \R \, :\, |\ox| \leq r, \oy  > y \mbox{ or } |\ox| > r, \oy > y 
- \tfrac{q\theta}{1+\theta} (|\ox| - r ) \big\} \, . \]
we can invoke Lemma~\ref{restriction_to_box} to see that
\[ \begin{aligned} \lim_{t \ra \infty} \Prob \{ Z_t = Z_{t+\theta t} \} & = 
\lim_{N \ra \infty} \lim_{t\ra\infty} \Prob(A(N,t)) \\
& = \lim_{N \ra \infty} \underset{{(x,y) \in B_N}}{\int\!\!\int}e^{-\nu(D^N_\theta(|x|,y))}\nu(\d x\,\d y) \\
& = \int_{y \geq 0} \int_{x \in \R^d} e^{- \nu(D_\theta(|x|,y))} \nu(\diff x\, \diff y) \, , \end{aligned}\]
where the last equality follows by dominated convergence, as the integrand %$\1\{(x,y) \in B_N\}e^{-\nu(D^N_c(|x|,y))}$ 
is dominated by $e^{- \nu(D_0(|x|,y))}$ which is integrable with respect to $\nu$ by the direct calculation in the 
next proposition. %For an illustration of the region $D_\theta(|x|,y)$ see Figure~\ref{point_process}.

\begin{figure}[htbp]
\centering 
\includegraphics[height=6cm]{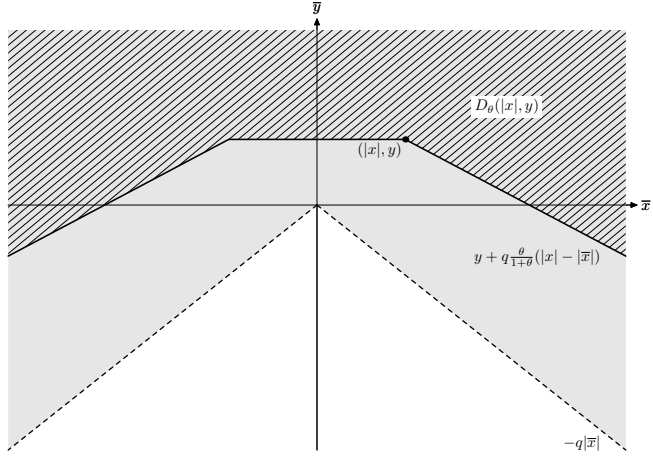} 
\caption[Illustration of the region $D_\theta(|x|,y)$.]{The point process $\Pi$ is defined on the set $\widehat H$ indicated in grey. 
%The shaded area shows the set $D_\theta(|x|,y)$ for a particular choice of $x$ and $y$. 
If we fix $Z_t/r_t=x, \Phi_t(Z_t)/a_t=y$, the condition that $Z_t = Z_{t+\theta t}$ 
corresponds to the requirement that the point process $\Pi$ has no points in the shaded region $D_\theta(|x|,y)$.
}\label{point_process}
\end{figure}

%\subsubsection{Simplification of the formula for $I(\theta)$}

We now simplify the expression that arises from the point process calculation. We
denote by $B(a,b)$ the Beta function with parameters $a,b$ and define the normalized 
incomplete Beta function 
%\be{incomplete_beta_function}
$$\tilde{B}(x,a,b) = \frac{1}{B(a,b)}\, \int_0^x v^{a-1} (1-v)^{b-1} \,\diff v \, .$$

\begin{prop}[Explicit form of $I(\theta)$]\label{I} For any $\theta\geq 0$, we have 
\[  \int_{y \geq 0} \int_{x \in \R^d} e^{- \nu(D_\theta(|x|,y))} \nu(\diff x\, \diff y) =  
I(\theta) := \frac{1}{B(\alpha - d+1,d)} \,\int_0^1  v^{\alpha - d} (1-v)^{d-1} \, \phi_\theta(v) \, \diff v \, ,\]
where the weight $\phi_\theta(v)$ is defined by
\begin{equation}\label{weight}
\tfrac1{\phi_\theta(v)} = 1- \tilde{B}(v,\alpha - d,d) + (1+\theta)^\alpha \, \big(\tfrac \theta v+1\big)^{d - \alpha}
 \tilde{B} \big( \tfrac{v+\theta}{1+\theta}, \alpha - d, d\big) \, .
\end{equation}
\end{prop}

\begin{proof}
First of all, we compute $\nu(D_\theta(r,y))$ for some $r>0$, 
\[ \begin{aligned} \nu(D_c(r,y)) & = \int_{|\ox| \leq r} \int_y^\infty \frac{\alpha \, \diff \ox \, \diff \oy}{(\oy + q |\ox|)^{\alpha+1}}
+ \int_{|\ox| > r} \int_{y - \frac{q\theta}{1+\theta}(|\ox| - r)}^\infty \frac{\alpha \,\diff \ox \, \diff \oy}{(\oy + q|\ox|)^{\alpha + 1}} \\
& = \int_{|\ox| \leq r} \frac{\diff \ox}{(y + q|\ox|)^{\alpha}} + \int_{|\ox| > r} \frac{ \diff \ox}{(y + \frac{q\theta}{1+\theta} r + \frac{q}{1+\theta}|\ox|)^{\alpha}} \, . \end{aligned} \]
Next, we can rewrite the two last summands. We exploit the invariance of the integrand under reflections at the axes, then
for $x_i \geq 0$ we use the substitution $u_1 = \ox_1 + \cdots +\ox_d$, $u_i = \ox_i$ for $i \geq 2$ and then 
the substitution $y + q u_1 = y/v$, so that
\be{turn_into_beta} \begin{aligned} \int_{|\ox| \leq r} \frac{\diff \ox}{(y + q|\ox|)^{\alpha}} & = 
2^d \int_0^{r} \frac{u_1^{d-1}}{(y+q u_1)^\alpha} \Big(\int_{\heap{u_2+\cdots + u_d \leq 1}{u_i \geq 0}} 
\diff u_2 \ldots \diff u_d \Big) \diff u_1 \\
& = \tfrac{2^d}{(d-1)!} \int_0^{r} \frac{u_1^{d-1}}{(y+ q u_1)^\alpha}\, \diff u_1 
 = \tfrac{2^d y^{d-\alpha}}{q^d (d-1)!} \int_{\frac{y}{y+qr}}^1 v^{\alpha - d- 1} (1-v)^{d-1} \diff v \\
& = \vartheta y^{d-\alpha} \Big( 1 - \tilde{B}\big( \tfrac{y}{y+qr}, \alpha - d, d\big) \Big) \, , \end{aligned} \ee
where $\vartheta = \frac{2^d B(\alpha - d,d)}{q^d (d-1)!}$. A similar calculation shows that
\[ \int_{|\ox| > r} \frac{ \diff \ox}{(y + \frac{q\theta}{1+\theta} r + \frac{q}{1+\theta}|\ox|)^{\alpha}} 
= \vartheta (1+\theta)^d \big( y + \tfrac{q\theta}{1+\theta}r\big)^{d - \alpha} \tilde{B} \big(\tfrac{y + \frac{q\theta}{1+\theta} r}{y+qr}, \alpha - d, d\big)  \, . \]
Combining the previous displays, and using the substitution $y + qr = y/v$ yields\vspace{-3mm}
\begin{align}\label{nu_D_c} %\begin{aligned} 
\nu\big( D_\theta(r,y)\big) & = \vartheta y^{d-\alpha} \Big( 1 - \tilde{B}\big( \tfrac{y}{y+qr}, \alpha - d, d\big) +
(1+\theta)^d \big( 1 + \tfrac{q\theta}{1+\theta}\tfrac{r}{y}\big)^{d - \alpha} \tilde{B} \Big(\tfrac{y + \tfrac{q\theta}{1+\theta} r}{y+qr}, \alpha - d, d\Big) \Big) \notag \\
& =  \vartheta y^{d-\alpha} \Big( 1 - \tilde{B}( v, \alpha - d, d) +
(1+\theta)^\alpha \big( 1 + \tfrac{\theta}{v}\big)^{d - \alpha} \tilde{B} \big(\tfrac{v+\theta}{1+\theta}, \alpha - d, d\big) \Big) 
\notag \\ & = \vartheta y^{d-\alpha} \phi_\theta(v)^{-1}\, . 
\end{align}%\end{aligned}\ee
To calculate the integral over $x \in \R^d$ we substitute $r = x_1 + \ldots + x_d$ and $u_i = x_i$ for $i \geq 2$,
\[ \int_{\R^d} e^{- \nu(D_\theta(|x|,y))} \frac{\alpha}{ (y + q|x|)^{\alpha + 1}} \,  \diff x
= \tfrac{2^d}{(d-1)!} \int_0^\infty e^{- \nu(D_\theta(r,y))} \frac{\alpha r^{d-1}}{(y+qr)^{\alpha + 1}} \,  \diff r . \]
Finally, we integrate over $y \geq 0$ and use the above formula for $\nu(D_\theta(r,y))$ together with the substitution $y + qr = y/v$
and $w = \vartheta y^{d-\alpha}$ to obtain
\[ \begin{aligned}\int_{y \geq 0} \int_{x \in \R^d} & e^{- \nu(D_\theta(|x|,y))} \, \nu(\diff x\, \diff y)
 = \tfrac{2^d}{(d-1)!} \int_0^\infty \int_0^\infty e^{- \nu(D_\theta(r,y))} \frac{\alpha r^{d-1} \diff r}{(y+qr)^{\alpha + 1}} \, \diff y\\
& = \tfrac{2^d}{q^d (d-1)!} \int_0^1 \alpha v^{\alpha - d} (1-v)^{d-1}
\int_0^\infty \exp \{ - \vartheta y^{d-\alpha} \phi_\theta(v)^{-1} \} \, y^{d-\alpha-1} \, \diff y \, \diff v \\
& = \tfrac{\alpha}{B(\alpha - d, d)(\alpha - d)} \int_0^1 v^{\alpha - d} (1-v)^{d-1} \int_0^\infty e^{-w \phi_\theta(v)^{-1}} 
\, \diff w \, \diff v \\
& = \tfrac{1}{B(\alpha -d +1,d)} \int_0^1 v^{\alpha -d} (1-v)^{d-1} \phi_\theta(v) \, \diff v
\, , \end{aligned} \]
where we used the identity $B(x+1,y)\,(x+y) = B(x,y)\,x$ for $x,y>0$ in the last~step.
\end{proof}

%\subsubsection{Asymptotics of $I$}

%Later on, we will need the following asymptotics of $I$ as $c \ra \infty$. 

\begin{prop}[Tails of $I$]\label{asymptotics_I_large_c}\ \\[-5mm]
\begin{itemize}
\item[(a)] $\displaystyle \lim_{\theta \ra \infty} \theta^d I(\theta) = \tfrac{1}{d \, B(\alpha - d + 1, d)}.$  
\item[(b)] $\displaystyle \lim_{\theta\da 0} \theta^{-1} (1-I(\theta)) = C_0,$
where the constant $C_0$ is given by \vspace{-3mm}
%\be{define_C_0} 
\[ C_0 = \tfrac{1}{B(\alpha-d+1,d)} \Big(\int_0^1 \alpha v^{\alpha-d}(1-v)^{d-1} \tB(v;\alpha-d,d) \diff v
+ B\big(2(\alpha-d), 2d-1\big)\Big) \, . \]
\end{itemize}
\end{prop}

\begin{proof} (a) As $\tilde{B}(v, \alpha-d,d) \leq 1$ and $v \mapsto \tilde{B}(v, \alpha-d,d)$ is nondecreasing we get,
for $0 \leq v \leq 1$,
\[\begin{aligned} \tfrac{1}{\phi_\theta(v)} &  = 1 - \tilde{B}(v, \alpha-d,d) + (1+\theta)^\alpha 
\big(\tfrac{\theta}{v} + 1\big)^{d-\alpha} \tilde{B}\big(\tfrac{v+\theta}{1+\theta}, \alpha-d,d\big) \\ 
%& \geq (1+\theta)^\alpha (\theta/v + 1)^{d - \alpha} \tilde{B}\Big(\frac{1+\theta}, \alpha-d,d\Big) \\
& \geq (1+\theta)^d \, v^{\alpha - d}\,  \big(\tfrac{1+\theta}{v+\theta}\big)^{\alpha - d} 
\tilde{B}\big(\tfrac{\theta}{1+\theta}, \alpha-d,d\big) \geq \tfrac{1}{2} (1+\theta)^d v^{\alpha - d} \,, \end{aligned} \]
where we chose $\theta$ large enough such that $\tilde{B}(\frac{\theta}{1+\theta}, \alpha-d,d) \geq \frac{1}{2}$.
Hence, %we deduce that 
for $\theta$ large enough,
\[ \begin{aligned} \theta^d I(\theta) & = \tfrac{1}{B(\alpha - d + 1, d)} \int_0^1 v^{\alpha - d}(1-v)^{d-1} \theta^d \phi_\theta(v) 
\, \diff v \\ & \leq \tfrac{2}{B(\alpha - d + 1, d)}  \int_0^1 (1-v)^{d-1}  \, \diff v 
= \tfrac{2}{d B(\alpha - d + 1, d)} \, . \end{aligned}\]
Therefore, since $\theta^d \phi_\theta(v) \ra v^{d-\alpha}$ pointwise for every $v \in (0,1)$ as $\theta \ra \infty$, we can invoke
the dominated convergence theorem to complete the proof of the lemma.

(b) We can write
\[ 1-I(\theta) = \tfrac{1}{B(\alpha-d+1,d)} \int_0^1v^{\alpha-d}(1-v)^{d-1} \phi_\theta(v)(\phi_\theta^{-1}(v) - 1) \, . \]
Set $\tB(v) = \tB(v; \alpha-d,d)$. Then, since $\phi_\theta(v) \ra 1$ for every $v$ as $\theta \da 0$,  we can concentrate on
\[ \begin{aligned} \phi_\theta(v)^{-1} - 1 & = (1+\theta)^d v^{\alpha-d}\, \big(\tfrac{1+\theta}{v+\theta}\big)^{\alpha-d} \tB\big(\tfrac{\theta+v}{1+v}\big) - \tB(v) \\
& = \tB\big(\tfrac{\theta+v}{1+\theta}\big) \, \big( \tfrac{(1+\theta)^\alpha v^{\alpha-d}}{(v+\theta)^{\alpha-d}} -1 \big) + \tB\big(\tfrac{\theta+v}{1+\theta}\big) - \tB(v). 
\end{aligned} \]
The first summand can be bounded by $(1+\theta)^{\alpha} - 1 \leq 2\alpha \theta$, 
eventually for all $\theta$. For the second term, we have that
\[ \tB\big(\tfrac{\theta+v}{1+\theta}\big) - \tB(v) = \int_v^\frac{v+\theta}{1+\theta} u^{\alpha-d-1}(1-u)^{d-1} \diff u
\leq \theta(1-v)^d \max\{ v^{\alpha-d-1},1\} \, . \]
Combining the two estimates we obtain that $\theta^{-1}(1-I(\theta))$ is bounded,
%\[ \theta^{-1}(1-I(\theta)) \leq  \tfrac{1}{B(\alpha-d+1,d)}\Big( 2\alpha \int_0^1v^{\alpha-d}(1-v)^{d-1}\diff v +
%\int_0^1 v^{\alpha-d}(1-v)^{2d-1} \max\{ v^{\alpha-d-1},1\}\diff v \Big)\, , \]
%which is finite, 
so that by the dominated convergence theorem, we may take the limit of $\theta^{-1}(\phi_\theta^{-1}(v) -1 )$ as $\theta\da0$
under the integral. 
\end{proof}

\subsection{Ageing for the solution profile}\label{ageing_for_solution}

In this section, we prove Theorem~\ref{ageing_for_u} by combining the results about ageing for the maximizer $Z_t$ from 
the previous section with the localization results in~\cite{KLMS09}. We start with a preliminary calculation that will 
be used several times in the remainder.

\begin{lemma}\label{difference_phi} If $\Phi_t(x) = \Phi_t(y)$ for some $t>0$ and $x,y\in\Z^d$ such that $t\xi(x) > |x|$ and $t \xi(y) > |y|$, then for all $s >0$ such that $s\xi(x) > |x|$ and $s \xi(y) > |y|$, we have that
 \[ \Phi_s(x) - \Phi_s(y) = (\xi(x) - \xi(y)) \big( 1- \tfrac{t}{s} \big) \, . \]
\end{lemma}

\begin{proof}By the assumptions on $t,x,y$, we find that
\[ \Phi_t(x) - \Phi_t(y) = (\xi(x) - \xi(y)) - \tfrac{1}{t} 
\big( |x| \log \xi(x) - |y| \log \xi(y) - \eta(x) + \eta(y) \big) = 0 \, . \]
Rearranging, we can substitute into 
\[ \begin{aligned} \Phi_s(x) - \Phi_s(y) &= (\xi(x) - \xi(y)) - \tfrac{1}{s} 
\big( |x| \log \xi(x) - |y| \log \xi(y) - \eta(x) + \eta(y) \big) \\
& = (\xi(x) - \xi(y)) \, \big( 1- \tfrac{t}{s} \big) \, , \end{aligned} \]
which completes the proof.
\end{proof}

\begin{rem}\label{Z_of_jump_time} 
% Since $\Phi_t(x) > 0$ for all $|x| \leq 1$ and all $t\geq1$, we know that 
Let $Z_t^{\ssup 1}, Z_t^{\ssup 2},\ldots\in\Z^d$ be sites in $\Z^d$ producing the largest values of $\Phi_t$
in descending order (choosing the site with largest $\ell^1$-norm in case of a tie), and recall that 
$Z_t=Z_t^{\ssup 1}$. It is then easy to see that $t \xi(Z_t^{\ssup i}) > |Z_t^{\ssup i}|$ 
for $i = 1,2$ and all $t \geq 1$. 
Hence, if $\tau>1$ is a jump time of the process $(Z_t \colon t>0)$, then $\Phi_\tau(Z_\tau^{\ssup 1}) = 
\Phi_\tau(Z_\tau^{\ssup 2})$, so that we can apply Lemma~\ref{difference_phi} with $x = Z_\tau^{\ssup 1}$ and 
$y = Z_\tau^{\ssup 2}$ and the conclusion holds for all $s \geq \tau$.
\end{rem}

\begin{lemma}\label{xi_of_Z_t} Almost surely, the function $u \mapsto \xi(Z_u)$ is nondecreasing on $(1, \infty)$. \end{lemma}

\begin{proof} 
Let $\{ \tau_n\}$ be the successive jump times of the process $(Z_t \colon t \geq 1)$. 
By definition, $$\Phi_{\tau_{n+1}}(Z^{\ssup 1}_{\tau_{n+1}}) = \Phi_{\tau_{n+1}}(Z^{\ssup 2}_{\tau_{n+1}})$$ 
and by right-continuity of $t \mapsto Z_t^{\ssup 1}$, we have that $Z^{\ssup 2}_{\tau_{n+1}} = Z^{\ssup 1}_{\tau_n}$. 
Now, consider $\tau_{n+1}< t< \tau_{n+2}$ such that $Z_t^{\ssup i} = Z_{\tau_{n+1}}^{\ssup i}$ for $i =1,2$, 
%using that by~\cite[Lemma 3.2]{KLMS09} $t \xi(Z_t^{\ssup i}) > |Z_t^{\ssup i}|$ for $i=1,2$ we can deduce 
then by Lemma~\ref{difference_phi} and Remark~\ref{Z_of_jump_time} we know that
\begin{equation}\label{preeq}
 \bal \Phi_t(Z_t^{\ssup 1}) - \Phi_t(Z_t^{\ssup 2}) & = \Phi_t(Z_{\tau_{n+1}}^{\ssup 1}) - \Phi_t(Z_{\tau_{n+1}}^{\ssup 2})
= (\xi ( Z_{\tau_{n+1}}^{\ssup 1}) - \xi ( Z_{\tau_{n+1}}^{\ssup 2}) ) ( 1 - \tfrac{\tau_{n+1}}{t}) \\
& =(\xi ( Z_{\tau_{n+1}}^{\ssup 1}) - \xi ( Z_{\tau_{n}}^{\ssup 1}) ) ( 1 - \tfrac{\tau_{n+1}}{t})  \, . \eal  
\end{equation}
As $t < \tau_{n+2}$, and $t \mapsto \Phi_t(Z_t^{\ssup 1}) - \Phi_t(Z_t^{\ssup 2})$ is not constant, the left hand side of  
\eqref{preeq} is strictly positive, which implies that $\xi ( Z_{\tau_{n+1}}) - \xi ( Z_{\tau_{n}}) > 0$, thus completing the proof.
\end{proof}

As an immediate consequence of this lemma, we get that $(Z_t \colon t>1)$ never returns to the same point 
in $\Z^d$. We now prove the first part of Theorem~\ref{ageing_for_u}.

\begin{lemma}\label{weak_ageing} For any sufficiently small $\eps >0$, 
\[ \lim_{t \ra \infty} \Prob \big\{ \sup_{z \in \R^d} \big| {v(t,z)} - {v(t+\theta t,z)} \big| < \eps \big\}
= \lim_{t \ra \infty} \Prob\{ Z_t= Z_{t+\theta t} \} = I (\theta) \, . \]
\end{lemma}

\begin{proof} Suppose $0< \eps < \frac{1}{2}$ and
let us throughout this proof argue on the event 
\[ A_t = \big\{  \, v(t,Z_t) > 1- \tfrac{\eps}{2}, \, v(t+\theta t,Z_{t+\theta t}) > 1- \tfrac{\eps}{2} \big\} \, . \] 
Now, if $z \neq Z_t$, then 
\[ u(t,z) \leq \sum_{x \neq Z_t} u(t,x) = U(t) - u(t, Z_t)< \tfrac{\eps}{2} \,U(t) \, , \]
and similarly if $z \neq Z_{t+\theta t}$, then $u(t+\theta t,z) \leq \frac{\eps}{2} \, U(t+\theta t)$.
%\[ u(t+\theta t,z) \leq \sum_{x \neq Z_{t+\theta t}} u(t+\theta t,x) 
%= U(t+\theta t) - u(t+\theta t, Z_{t+\theta t})< \tfrac{\eps}{2} \, U(t+\theta t) .\]
In particular, if $z \neq Z_t$ and $z \neq Z_{t+\theta t}$, then
$| v(t,z) - v(t+\theta t,z)| < \eps$.
%\be{bound}  \big| v(t,z) - v(t+\theta t,z) \big| < \eps \, . \ee
Now, if $Z_t = Z_{t+\theta t}$, then by assumption $A_t$
we have $| v(t,z) - v(t+\theta t,z)| < \eps$ for any $z \in \Z^d$. 
Conversely, suppose that $Z_t \neq Z_{t+\theta t}$. From above we then get 
$u(t+\theta t,Z_{t}) < \frac{\eps}{2} U(t+\theta t)$
and since we argue on the event $A_t$, we find that
$v(t,Z_t) - v(t+\theta t, Z_t) > 1 - \eps > \eps$,
so that 
\[ \sup_{z \in \Z^d} \big| v(t,z) - v(t+\theta t,z) \big| 
\geq \big|v(t,Z_t) - v(t+\theta t, Z_t)\big| > \eps \, . \]
To complete the proof, it remains to notice that
since $v(t, Z_t)$ converges weakly to one, we have that  $\Prob(A_t) \ra 1$ as $t \ra \infty$.
\end{proof}

Before we can prove the remaining part of Theorem~\ref{ageing_for_u}, we need to collect the following 
fact about the maximizers $Z^{\ssup 1}$ and $Z^{\ssup 2}$.

\begin{lemma}\label{separation_first_second} Let $\la_t = (\log t)^{-\beta}$ for some $\beta > 1 + \frac{1}{\alpha-d}$.
If $t_1\leq t_2$ are sufficiently large, satisfy $Z^{\ssup 1}_{t_1} = Z^{\ssup 1}_{t_2}$ and
\be{diff_first_second} \Phi_t(Z_t^{\ssup 1}) - \Phi_t(Z_t^{\ssup 2}) \geq  \tfrac{1}{2} \, a_t \la_t  \, , \\[2mm]\ee
holds for $t = t_1$ and $t= t_2$, then~(\ref{diff_first_second}) holds for all $t \in [t_1,t_2]$.
\end{lemma}

\begin{proof} First, we additionally assume that
$Z_t^{\ssup 2} = Z_{t_1}^{\ssup 2}$ for all $t \in [t_1,t_2)$.
By Lemma~\ref{xi_of_Z_t} we have\vspace{-1mm} that $Z^{\ssup 1}_{t} = Z^{\ssup 1}_{t_1}$ for all $t \in [t_1,t_2]$. 
Using also the continuity of $t \mapsto \Phi_t(Z_t^{\ssup i})$, $i = 1,2$, we get
\[\begin{aligned} \Phi_t(Z_t^{\ssup 1}) & - \Phi_t(Z_t^{\ssup 2}) 
 = \Phi_t(Z_{t_1}^{\ssup 1}) - \Phi_t (Z_{t_1}^{\ssup 2}) \\
& = \xi(Z_{t_1}^{\ssup 1}) - \xi(Z_{t_1}^{\ssup 2}) - \tfrac{1}{t} \big( |Z_{t_1}^{\ssup 1}|\log \xi(Z_{t_1}^{\ssup 1}) 
- |Z_{t_1}^{\ssup 2}|\log \xi(Z_{t_1}^{\ssup 2}) - \eta(Z_{t_1}^{\ssup 1})  + \eta(Z_{t_1}^{\ssup 2}) \big) \\
& = A - \tfrac{1}{t} B \quad \mbox{for all } t \in [t_1,t_2] \, , \end{aligned} \]
for some constants $A, B \in \R$ depending only on $t_1$. 
%Now, by assumption
%\[ \Phi_t(Z_t^{\ssup 1}) - \Phi_t(Z_t^{\ssup 2})  \geq \tfrac{1}{2} \, a_t \la_t %=: \gamma(t) 
%\, , \]
%for $t = t_1,t_2$. 
Now, defining
\[ f(t) = A - \tfrac{1}{t} B - \tfrac{1}{2} \, a_t \la_t %\gamma(t) 
\, , \]
we get that $f(t_1) \geq 0$ and $f(t_2) \geq 0$ by our assumption. 
%Now, calculating the derivative of $\gamma(t)$ gives	
%\[ \gamma'(t) = \tfrac{1}{2}\, \frac{t^{q-1}}{(\log t)^{q + \beta}} ( q - \frac{q + \beta}{\log t}) \, , \]
%so that $\gamma'$ is either strictly increasing or decreasing for all $t \geq t_1$ if $t_1$ is sufficiently large. It follows that for all $t \geq t_1$, $f'(t) = B \frac{1}{t^2} - \gamma'(t)$ has at most one zero $t'$. 
Moreover, 
\[ f'(t) %= B \frac{1}{t^2} - \gamma'(t) 
= \frac{1}{t^2}\Big( B - \tfrac{1}{2} \, \frac{t^{q+1}}{(\log t)^{q + \beta}} \big( q - \tfrac{q + \beta}{\log t}\big)\Big), \]
which is negative for $t$ larger than some threshold depending on~$t_1$. 
Also, if $t_1$ is large enough, the function $t \mapsto \frac{t^{q+1}}{(\log t)^{q + \beta}} \big( q - \frac{q + \beta}{\log t}\big)$ is strictly increasing 
for $t \geq t_1$, hence $f'$ has at most one zero for $t \geq t_1$.
Therefore, if $f'$ has a zero $t' \geq t_1$, then $f'$ is negative for all $t> t'$, implying that $f$ does not have 
a minimum at $t' \in (t_1,t_2)$. If $f'$ does not have a zero for $t \geq t_1$, then it follows that $f'(t) < 0$ for all $t \geq t_1$. 
In either case,  $f(t_1) \geq 0$ and $f(t_2) \geq 0$ imply that $f(t) \geq 0$ for all $t \in [t_1,t_2]$, in other words~(\ref{diff_first_second}) holds for all $t \in [t_1,t_2]$.

Now we drop the extra assumption on $Z^{\ssup 2}_t$ and define the jump times
\[ \tau^- = \sup\big\{ t < t_1 \, : \, Z_t^{\ssup 1} \neq Z_{t_1}^{\ssup 1} \big\} \quad \mbox{and}\quad
\tau^+ = \inf\big\{ t > t_2 \, : \, Z_t^{\ssup 1} \neq Z_{t_1}^{\ssup 1} \big\} \, . \]
%$n \in \N$ such that $\tau_n \leq t_1 < t_2 < \tau_{n+1}$.
Furthermore, define a sequence $s^{\ssup i}$ by setting $s^{\ssup 0} = \tau^-$ and for $i \geq 1$ setting
\[ s^{\ssup i} = \inf\{ s > s^{\ssup{i-1}} \, : \, \Phi_s(Z_s^{\ssup 2}) = \Phi_s(Z_s^{\ssup 3})  \}  \, . \]
Then, there exists $N \geq 1$ such that $s^{\ssup N} < \tau^+ < s^{\ssup {N+1}}$, where $N \geq 1$ since,
by Lemma~\ref{xi_of_Z_t}, 
$$Z_{\tau^-}^{\ssup 2} = \lim_{t \ua \tau^{-}} Z_{t}^{\ssup 1}\neq Z_{\tau^+}^{\ssup 1}.$$ 
Using that $\Phi_{s^{\ssup i}}(Z_{s^{\ssup i}}^{\ssup 2}) = 
\Phi_{s^{\ssup i}}(Z_{s^{\ssup i}}^{\ssup 3})$ and Proposition 3.4 in~\cite{KLMS09},
\[ \Phi_{s^{\ssup i}}(Z_{s^{\ssup i}}^{\ssup 1}) - \Phi_{s^{\ssup i}}(Z_{s^{\ssup i}}^{\ssup 2}) = \Phi_{s^{\ssup i}}(Z_{s^{\ssup i}}^{\ssup 1}) - \Phi_{s^{\ssup i}}(Z_{s^{\ssup i}}^{\ssup 3}) \geq \, a_{s^{\ssup i}} \la_{s^{\ssup i}} \, . \]
Therefore,~(\ref{diff_first_second}) holds for $t = s^{\ssup i}$, $i =1, \ldots, N$. Hence, the additional assumption that we made in the first part of the proof holds for each of the intervals $[t_1,s^{\ssup 1})$, $[s^{\ssup i}, s^{\ssup{ i+1}})$, for $i = 1, \ldots, N-1$ and 
$[s^{\ssup N},t_2)$. Thus, we can deduce that~(\ref{diff_first_second}) holds for all $t$ in the union of these intervals, 
%i.e.\[ \Phi_t(Z_t^{\ssup 1}) - \Phi_t(Z_t^{\ssup 2}) \geq  \tfrac{1}{2} \, a_t \la_t \quad \mbox{for all } t \in [t_1,s^{\ssup 1}] \cup %\bigcup_{i=1}^{N-1} [s^{\ssup i}, s^{\ssup{ i+1}}]  \cup [s^{(N)},t_2] \, , \]
which completes the proof.
\end{proof}

Finally, we can now show the stronger form of ageing for the profile $v$ and thereby 
complete the proof of Theorem~\ref{ageing_for_u}.

\begin{proof}[Proof of Theorem~\ref{ageing_for_u}] 
By Proposition~\ref{ageing_for_Z}, it suffices to show that
\[ \lim_{t \ra \infty} \Prob \Big\{ \sup_{\heap{z \in \R^d}{s \in [t, t+\theta t ]}} 
\big| {v(t,z)} - {v(s,z)} \big| < \eps \Big\}
= \lim_{t \ra \infty} \Prob\{ Z^{\ssup 1}_t= Z^{\ssup 1}_{t+\theta t} \} \,. \]
First of all, note that by Lemma~\ref{xi_of_Z_t} we know that $Z^{\ssup 1}_t = Z^{\ssup 1}_{t+\theta t}$ if and only
if $Z^{\ssup 1}_t = Z^{\ssup 1}_s$ for all $s \in [t,t+\theta t]$. We will work on the event
\[ A_t = \big\{  \Phi_{t}(Z_{t}^{\ssup 1}) - \Phi_{t}(Z_{t}^{\ssup 2}) \geq a_{t} \la_{t}/ 2 \big\} \cap
\big\{  \Phi_{t+\theta t}(Z_{t+\theta t}^{\ssup 1}) - \Phi_{t+\theta t}(Z_{t+\theta t}^{\ssup 2}) 
\geq a_{t+\theta t} \la_{t+\theta t}/ 2 \big\} \, . \]
Recall from Proposition 5.3 in~\cite{KLMS09} that if $\Phi_t(Z_t^{\ssup 1})$ and $\Phi_t(Z_t^{\ssup 2})$ are sufficiently far apart, then the profile is localized in $Z_t^{\ssup 1}$. More precisely, almost surely, 
\[ \lim_{t \ra \infty} \sum_{z \in \Z^d \setminus \{ Z_t^{\ssup 1}\}} v(t,z) \,
\1\{ \Phi_t(Z_t^{\ssup 1}) - \Phi_t(Z_t^{\ssup 2}) \geq a_t \la_t / 2 \}= 0\, .  \]
In particular, for given $\eps<\frac12$, we can assume that $t$ is sufficiently large, so that for all $s \geq t$, 
\be{total_mass_if_separated}  \sum_{z \in \Z^d \setminus \{ Z_s^{\ssup 1}\}} v(s,z) \,
\1\{ \Phi_s(Z_s^{\ssup 1}) - \Phi_s(Z_s^{\ssup 2}) \geq a_s \la_s / 2\} < \tfrac\eps2\, .\ee
%since this is an event whose probability tends to $1$ as $t \ra \infty$.
Now, if $Z_t^{\ssup 1} \neq Z_{t+\theta t}^{\ssup 1}$, then on $A_t$, we know by~(\ref{total_mass_if_separated}) that
$v(t+\theta t, Z_t^{\ssup 1}) \leq \frac\eps2$.
%\[ u(t+\theta t, Z_t^{\ssup 1}) \leq \sum_{z \in \Z^d\setminus\{Z_{t+\theta t}^{\ssup 1}\}} u(t+\theta t,z) < \tfrac{1}{2} \, \eps \, U(t+\theta t) %\, , \]
Combining this with the fact that $v(t, Z_t^{\ssup 1}) > 1 - \frac\eps2$, we have that
\[ \sup_{\heap{z \in \Z^d}{s \in [t,t+\theta t]}} \big| v(t,z) - v(s,z) \big| 
\geq \big| {v(t,Z_t^{\ssup 1})} - v(t+\theta t,Z_t^{\ssup 1}) \big|
> 1- \eps > \eps \, . \]
Conversely, assume that $Z_t^{\ssup 1} = Z_{t+\theta t}^{\ssup 1}$, then by Lemma~\ref{xi_of_Z_t}, $Z_t^{\ssup 1}= Z_s^{\ssup 1}$ for all \mbox{$s \in [t,t+\theta t]$}. Now, on the event $A_t$ we know by Lemma~\ref{separation_first_second}  that for all $s \in [t,t+\theta t]$,
\be{difference_first_second} \Phi_s(Z_s^{\ssup 1}) - \Phi_s(Z_s^{\ssup 2}) \geq  a_s \la_s / 2 \, . \ee
This implies by~(\ref{total_mass_if_separated}) that 
$$\sum_{z \in \Z^d \setminus \{ Z_s^{\ssup 1}\}} v(s,z) < \eps/2 \qquad \mbox{ for all $s \in [t,t+\theta t]$.}$$
As in the proof of Lemma~\ref{weak_ageing}, this yields that
\[ \sup_{\heap{z \in \R^d}{ s \in [t, t+\theta t]}} \big| {v(t,z)} - {v(s,z)} \big| < \eps \, . \]
Hence, to complete the proof, it remains to notice that by~\cite[Lemma 6.2]{KLMS09} the pair
$(\Phi_t(Z_t^{\ssup 1})/{a_t},{\Phi_t(Z_t^{\ssup 2})}/{a_t})$ converges weakly to a limit random variable 
with a density, from which we conclude that $\Prob(A_t) \ra 1$ as $t \ra \infty$. % might need better explanation
\end{proof}

\section{Ageing: an almost-sure limit theorem}\label{se:almost_sure_ageing}

In this section, we prove Theorem~\ref{asymptotics_R_X}. As in the previous section, we first concentrate on an analogous 
theorem for the maximizer of the variational problem $\Phi_t$. In particular, in Section~\ref{sect_moderate_deviations}, we 
extend Proposition~\ref{ageing_for_Z} to a moderate deviations principle. This estimate allows us to prove the equivalent 
of the almost sure ageing Theorem~\ref{asymptotics_R_X} in the setting of the variational problem in Section~\ref{asymptotics_tau}. 
Finally, in Section~\ref{sect_asymptotics_sigma}, we transfer this result to the maximizer of $v$.

\subsection{Moderate deviations}\label{sect_moderate_deviations}

Recall from Proposition~\ref{asymptotics_I_large_c} that 
\[ \lim_{t \ra \infty} \Prob\{Z_t = Z_{t+\theta t}\} = I(\theta) \sim \frac{1}{dB(\alpha-d+1,d)} \, \theta^{-d} \, , \]
where the latter asymptotic equivalence holds for $\theta$ tending to infinity. 
We now show that we obtain the same asymptotic for $\Prob\{ Z_t = Z_{t+\theta t} \}$ 
if we allow $\theta$ to grow slowly with $t$. 

\begin{prop}[Moderate deviations]\label{moderate_deviations} For any positive function $\theta_t$ such that \mbox{$\theta_t \ra \infty$} and 
$\theta_t \leq (\log t)^\delta$ for some $\delta >0$, we have that 
\[ \Prob\{ Z_t= Z_{t(1+\theta_t)} \} = \big( \tfrac{1}{dB(\alpha-d+1,d)}+o(1) \big)\,\theta_t^{-d}\,.\]
%for some explicit constant $C>0$.
\end{prop}

Unlike in the proof of Proposition~\ref{ageing_for_Z}, we cannot directly
use the point process techniques, as the weak convergence only applies to compact sets, whereas here we deal with sets that 
increase slowly with~$t$ to a set that has infinite mass under the intensity measure $\nu$. We start by expressing $\Phi_t(z)$ in terms of $\xi(z)$ and $|z|$, while carefully controlling the errors.

\begin{lemma}\label{exact_errors_for_Phi}
There exist $C_1,C_2>0$ and $t_0>0$ such that, for all $z\in\Z^d, t>t_0$ with $t \xi(z) > |z|$, %we have
\[ \frac{\xi(z)}{a_t} - q \frac{|z|}{r_t}\big(1+2\tfrac{\log (N_t + q g_t)}{\log t}\big)  
\leq \frac{\Phi_t(z)}{a_t} \leq \frac{\xi(z)}{a_t} - q \frac{|z|}{r_t}\big(1-C_1\tfrac{\log \log t}{\log t}\big) + C_2 \frac{1}{\log t} \, , \]
where the lower bound holds uniformly for all functions $N_t,g_t$ such that $\Phi_t(z) \leq a_t N_t, |z|\leq r_t g_t$ and 
$N_t,g_t \ra \infty$ as $t \ra \infty$. Similarly, for 
$\theta \geq 0$ and $z\in\Z^d$ such that $(1+\theta)t \xi(z) > |z|$, we have
\[ \frac{\xi(z)}{a_t} - \frac{q}{1+\theta}\frac{|z|}{r_t}
\big(1+2 \tfrac{\log(N_t + q g_t)}{\log t}\big) \leq \frac{\Phi_{t+\theta t}(z)}{a_t} \leq 
\frac{\xi(z)}{a_t} - \frac{q}{1+\theta}\frac{|z|}{r_t}\big(1-C_1 \tfrac{\log \log t}{\log t}\big) + C_2 \frac{1}{\log t} \, , \]
again with the restriction that for the lower bound we assume that
$\Phi_t(z) \leq a_t N_t$ and $|z|\leq r_t g_t$. %Moreover, for all $z$ such that $\Phi_t(z) \leq a_t N_t$ and $|z|\leq r_t g_t$
%\[ \frac{\Phi_t(z)}{a_t} + \frac{qc}{1+c}\frac{|z|}{r_t}\Big(1 - C_1 \frac{\log\log t}{\log t}\Big) - \frac{C_2}{\log t} \leq \frac{\Phi_{t+\theta t}(z)}{a_t} \leq \frac{\Phi_t(z)}{a_t} + \frac{q c}{1+c} \frac{|z|}{r_t}\Big(1 + \frac{\log (N_t + q g_t)}{\log t}\Big) \, . \]
\end{lemma}

\begin{proof} Using that $r_t = \frac{t}{\log t} a_t$, we have, for $t \xi(z) > |z|$, that
$$\begin{aligned} \frac{\Phi_t(z)}{a_t} & = \frac{\xi(z)}{a_t} - \frac{1}{ta_t} ( |z| \log \xi(z) - \eta(z) ) 
 = \frac{\xi(z)}{a_t} - q \frac{|z|}{r_t} + {\rm error}(t,z),  \end{aligned}$$
where
$${\rm error}(t,z)=
q\frac{|z|}{r_t} \frac{\log \log t}{\log t} - \frac{|z|}{r_t \log t} \log \frac{\xi(z)}{a_t} + \frac{\eta(z)}{r_t \log t} \, .$$
It thus suffices to find suitable upper and lower bounds for the last two terms. For the upper bound, we use that $\eta(z) \leq |z|\log d$ and also that $x \log x\geq -e^{-1}$ for any $x>0$, to get
\[ -\frac{|z|}{r_t \log t} \log \frac{\xi(z)}{a_t} +\frac{\eta(z)}{a_t t} \leq -\frac{|z|}{r_t \log t} \log \frac{|z|}{r_t \log t} + \frac{|z|}{r_t}\frac{\log d}{\log t} \leq \frac{1}{\log t} e^{-1} + \frac{|z|}{r_t}\frac{\log \log t + \log d}{\log t} \, , \]
so that the upper bound holds for $C_1\ge\frac1q\,(1+\log d)+q$ and $C_2\ge e^{-1}$.
For the lower bound we\vspace{-1mm} note that either $\xi(z)/a_t < (1+\frac{g_t}{\log t})^2$, 
or we can use $\log x \leq x^{1/2}$, for all $x >0$, to estimate
\[\begin{aligned} \frac{\Phi_t(z)}{a_t} + q\frac{|z|}{r_t} & \geq \frac{\xi(z)}{a_t} - \frac{|z|}{r_t\log t} \log \frac{\xi(z)}{a_t}
 \geq \frac{\xi(z)}{a_t} - \frac{|z|}{r_t \log t} \Big(\frac{\xi(z)}{a_t}\Big)^{1/2}\\ & \geq \frac{\xi(z)}{a_t} - \frac{g_t}{ \log t} \Big(\frac{\xi(z)}{a_t}\Big)^{1/2} \geq\Big(\frac{\xi(z)}{a_t}\Big)^{1/2} \, . \end{aligned}\]
Therefore, we have
${\xi(z)}/{a_t} \leq \max\{ (1+\tfrac{g_t}{\log t})^2, (\tfrac{\Phi_t(z)}{a_t} + q\tfrac{|z|}{r_t})^2\}
\leq (N_t + q g_t)^2$.
Hence, we can conclude that
\[{\rm error}(t,z) \geq   - q \frac{|z|}{r_t}\, \frac{2 \log (N_t+q g_t)}{\log t}\, .\]
For the bound on $\Phi_{t+\theta t}(z)$ it suffices to note that
\[\bal \frac{\Phi_{t + \theta t}(z)}{a_t} % & = \frac{\xi(z)}{a_t} - \frac{1}{(1+\theta)t a_t}(|z|\log \xi(z) - \eta(z) ) \\
%& = \frac{\xi(z)}{a_t} - \frac{q}{1+\theta}\frac{|z|}{r_t} + \frac{1}{1+\theta}\Big( q \frac{|z|}{r_t}\frac{\log \log t}{\log t}
%- \frac{|z|}{r_t \log t} \log \frac{\xi(z)}{a_t} + \frac{\eta(z)}{r_t \log t} \Big) \\
& = \frac{\xi(z)}{a_t} - \frac{q}{1+\theta}\frac{|z|}{r_t} + \frac{1}{1+\theta} \, \mbox{error}(t,z) \, , \eal \]
where error$(t,z)$ is precisely the same error term as in the first part of the lemma.
%For the final estimate, note that
%\[ \frac{\Phi_{t+\theta t}(z)}{a_t} = \frac{\Phi_t(z)}{a_t} + \frac{qc}{1+c}  \Big( \frac{|z|}{r_t \log t} \log \frac{\xi(z)}{a_t} - \frac{\eta(z)}{r_t \log t} - \frac{|z|}{r_t}\frac{\log \log t}{\log t} \Big) \, . \]
%Thus the claim follows with the same error estimates as before, using in particular that $\frac{c}{1+c} \leq~1$.
\end{proof}

In analogy to the proof of Proposition~\ref{ageing_for_Z}, we will have to restrict $(Z_t/r_t, \Phi_t(Z_t)/a_t)$  to large boxes in $\R^d \times \R$. The first step is therefore to estimate the probability that $(Z_t/r_t, \Phi_t(Z_t)/a_t)$ lies outside a large box.

\begin{lemma}\label{restricting_maximizer}There exist constants $C, C'>0$ such that for all $t>0$ large enough, uniformly for all $N\geq 1$, 
\vspace{-2mm}
% and $0<\eta <1$,
\begin{itemize} 
\item[(a)] $\displaystyle\Prob\big\{ \tfrac{|Z_t|}{r_t} \geq N \big\} \leq C \, N^{d-\alpha}$,
\item[(b)] $\displaystyle\Prob\big\{ \tfrac{\Phi_t(Z_t)}{a_t}  \geq N \big\} \leq C \, N^{d-\alpha}$,
\item[(c)] for any positive function $\eta_t\le 1$ such that $\eta_ta_t \ra \infty$ we have
\[ \Prob\big\{ \tfrac{\Phi_t(Z_t)}{a_t} \leq \eta_t \big\} \leq C e^{-C'\eta_t^{d-\alpha}} \, . \]
\end{itemize}
\end{lemma}

\begin{proof} (a) Using Lemma~\ref{exact_errors_for_Phi}, we can estimate 
\[\begin{aligned}  \Prob\{ |Z_t| \geq N r_t\} & \leq \Prob\{ \exists z\in\Z^d \mbox{ with }|z|\geq N r_t, \tfrac{\Phi_t(z)}{a_t} \geq 0 \} \\
& \leq \sum_{\heap{z\in\Z^d}{|z|\geq N r_t}} \Prob\Big\{ \tfrac{\xi(z)}{a_t} \geq q\tfrac{|z|}{r_t}\big(1-C_1 \tfrac{\log \log t}{\log t}\big) - C_2 \tfrac{1}{\log t} \Big\} \\
& = (1+o(1))\sum_{\heap{z\in\Z^d}{|z|\geq N r_t}} a_t^{-\alpha} \big(q\tfrac{|z|}{r_t}\big)^{-\alpha} = (1+o(1))q^{-\alpha}r_t^{\alpha-d}\sum_{\heap{z\in\Z^d}{|z|\geq N r_t}} |z|^{-\alpha} \, , \end{aligned}\]
where we used that $r_t^d=a_t^\alpha$ and $o(1)$ tends to $0$ as $t \ra \infty$ uniformly in $N\geq 1$. We obtain
the required bound by noting that the sum is bounded by a constant multiple of $(Nr_t)^{d-\alpha}$.
%Approximating the sum by an integral we get the upper bound
%\[ \begin{aligned}  & (1+o(1))q^{-\alpha}r_t^{\alpha-N} \int_{|z| \geq N r_t} (|z| -1)^{-\alpha}\diff z \\
%& = (1+o(1)) \, \tfrac{2^d}{(d-1)!q^{\alpha}} r_t^{\alpha-d} \int_{r \geq r_t N} r^{d-1}(r-1)^{-\alpha} \diff r \leq
%(1+o(1)) \,\tfrac{2^{d+\alpha}}{(d-1)!q^\alpha} N^{d-\alpha} 	,  \end{aligned} \]
%assuming that $r_t \geq 2$, so that the first claim follows.

(b) For the second estimate, we use again Lemma~\ref{exact_errors_for_Phi} to obtain
\[ \begin{aligned} \Prob\{ \Phi_t(Z_t) & \geq N a_t \} \leq \sum_{z\in\Z^d} \Prob\{ \Phi_t(z) \geq N a_t\} \\
& \leq \sum_{z \in \Z^d} \Prob\Big\{ \tfrac{\xi(z)}{a_t} \geq N + q \tfrac{|z|}{r_t} \Big( 1- C_1\tfrac{\log \log t}{\log t}\Big) - C_2\tfrac{1}{\log t} \Big\} \\
& \leq (1+o(1)) \sum_{z \in \Z^d} a_t^{-\alpha}\big(N + q\tfrac{|z|}{r_t}\big)^{-\alpha} \, . \end{aligned} \vspace{-2mm}\]
Similarly as before, observe that the sum is bounded by a constant multiple of
$\int_0^\infty r^{d-1} (N + qr)^{-\alpha}$, which itself is bounded by a constant multiple of
$N^{d-\alpha}$.

(c) For the last bound, note first that by Lemma~\ref{exact_errors_for_Phi}, that if $t\xi(z) > |z|$ and \mbox{$|z|/r_t < g_t:=\log t$}  and $\Phi_t(z)/a_t < 1$, then there exists $C>0$ such that
\[ \tfrac{\xi(z)}{a_t} - q\tfrac{|z|}{r_t}\big(1+C \tfrac{\log\log t}{\log t}\big) < \tfrac{\Phi_t(z)}{a_t} \, . \]
Hence, we can estimate
\[\begin{aligned} \Prob\big\{ \tfrac{\Phi_t(Z_t)}{a_t} \leq \eta_t  \big\} & \leq 
\Prob\big\{ \tfrac{\Phi_t(z)}{a_t} \leq  \eta_t \mbox{ for all } z \mbox{ with } t\xi(z) > |z| \mbox{ and } |z|< r_t(\log t) \big\} \\
& \leq \prod_{\heap{z\in\Z^d}{|z|\leq r_t g_t}} \Prob\Big\{ t\xi(z) \leq |z| \mbox{ or }  \, \tfrac{\xi(z)}{a_t} \leq \eta_t + q\tfrac{|z|}{r_t}\big(1+C\tfrac{\log \log t}{\log t}\big) \Big\}
 \, . \end{aligned} \]
Now, if $t\xi(z) \leq |z|$ and $t$ is large enough, the second inequality must hold as well. Hence we obtain
\[\begin{aligned} \Prob\big\{ \tfrac{\Phi_t(Z_t)}{a_t}\leq \eta_t \big\} & \leq \prod_{\heap{z\in\Z^d}{|z|\leq r_t g_t}} \Prob\Big\{ \tfrac{\xi(z)}{a_t} \leq \eta_t + q\tfrac{|z|}{r_t}\big(1+C\tfrac{\log \log t}{\log t}\big) \Big\} \\
& = \exp\Big\{ \sum_{\heap{z\in\Z^d}{|z|\leq r_t g_t}} \log \Big(1 -  a_t^{-\alpha} \big(\eta_t + q\tfrac{|z|}{r_t}\big(1+ C \tfrac{\log \log t}{\log t}\big)\big)^{-\alpha}\Big) \Big\} \\
& \leq \exp\Big\{ - (1+o(1))\sum_{\heap{z\in\Z^d}{|z|\leq r_t g_t}} a_t^{-\alpha} \big(\eta_t + q\tfrac{|z|}{r_t}\Big)^{-\alpha}\Big\} \, ,
\end{aligned}  \vspace{-1mm}\]
using that $a_t \eta_t \ra \infty$ and  $\log(1-x) \leq -x$ for $x<1$. The sum can be bounded from below by a constant multiple of
$$r_t^{-d}\int_0^{r_t g_t-\frac{1}{2}}  r^{d-1}\big(\eta_t + q\tfrac{r+\tfrac{1}{2}}{r_t}\big)^{-\alpha}\diff r
= (1+o(1))\, \int_0^{g_t-\frac{1}{2}r_t^{-1}}  r^{d-1}(\eta_t + qr)^{-\alpha} \diff r,$$
and the latter integral can be seen to be bounded from below by a constant multiple of~$\eta_t^{d-\alpha}$.
\end{proof}

\begin{proof}[Proof of Proposition~\ref{moderate_deviations}] 
The \emph{main idea} is again to restrict $(Z_t/r_t, \Phi_t(Z_t)/a_t)$ to large boxes to be able to control the error when approximating $\Phi_t$. To set up the notation, we introduce functions $\eta_t = (\log t)^{-\beta'}$, $N_t = (\log t)^{\beta}$,
$g_t = (\log t)^\gamma$ for some parameters $\beta, \beta', \gamma > 0$, which we will choose later on depending on the function $\theta_t$
%. In particular, the idea is to restrict the maximizer 
such that
\[ \Prob\big\{ Z_t = Z_{t(1+\theta_t)} \big\} = 
\Prob\big\{ Z_t = Z_{t(1+\theta_t)} ,\, |Z_t| \leq r_t g_t,\, \tfrac{\Phi_t(Z_t)}{a_t} \in [\eta_t,N_t] \big\} + o(\theta_t^{-d}) \, . \]
Once these growing boxes are defined, we can find by Lemma~\ref{exact_errors_for_Phi} a constant $C > 0$ such that the function $\delta_t = C \frac{\log \log t}{\log t}$ satisfies
\[ \frac{\xi(z)}{a_t} - q\frac{|z|}{r_t}(1+\delta_t) \leq \frac{\Phi_t(z)}{a_t} \leq \frac{\xi(z)}{a_t} - q\frac{|z|}{r_t}(1-\delta_t) + \delta_t \, , \]
where the upper bound holds for all $z\in\Z^d$ and the lower bound for all $z\in\Z^d$ such that $|z| \leq r_t g_t$ and $\Phi_t(z) \leq a_t N_t$.

\emph{Upper bound.}
%By Lemma~\ref{restricting_maximizer}, there exists constant $C_1,C_2 >0$ such for all $t>0$.
%\[\begin{aligned} |\Prob\{ Z_t = Z_{t+\theta_t} \} & - \Prob\{ Z_t = Z_{t+\theta_t}; \eta'_t a_t \leq \Phi_t(Z_t) \leq \eta_t a_t; |Z_t| \leq g_t r_t \} | \\
%& \leq \Prob\{ \Phi_t(Z_t) \leq \eta'_t a_t \} + \Prob\{\Phi_t(Z_t) \geq \eta_t a_t \} + \Prob\{ |Z_t| \geq g_t r_t \}
%\\ & \leq C_1 ( e^{C_2 (\eta'_t)^{d-\alpha}} + \eta_t^{d-\alpha} + g_t^{d-\alpha} )
%  \, , \end{aligned} \]
%where $\eta'=(\log t)^{\gamma'},\eta=(\log t)^{\gamma}$ and $g_t=(\log t)^\beta$ for some constants $\beta,\gamma,\gamma' > 0$ which by our assumption on the growth of $\theta_t$ we can choose large enough such that
%\[ \theta_t^d (e^{C_2 (\eta'_t)^{d-\alpha}} + \eta_t^{d-\alpha} + g_t^{d-\alpha} ) \ra 0 \, \quad \mbox{as } t \ra \infty \, . \]
%This implies that when we try to calculate $\theta_t^d \Prob\{ Z_t = Z_{t+\theta t} \}$ we can concentrate on the second probability where $(\Phi_t(Z_t)/a_t,Z_t/r_t)$ are restricted to a box growing slowly in time.
%Note first that by Lemma~\ref{exact_errors_for_Phi}, there exists a constant $C>0$ so that if we denote by $\delta_t = C\frac{\log \log t}{\log t}$, then for all $z$ such that $t \xi(z) > |z|$ we have that
%\[ \frac{\Phi_t(z)}{a_t} < \frac{\xi(z)}{a_t} - q\frac{|z|}{r_t}(1-\delta_t) + \delta_t \, . \]
%Thus, since $t\xi(Z_t) > |Z_t|$ we can estimate for $N_t = (\log t)^\gamma$ and $\eta_t = (\log t)^{-\gamma'}$
We use a slight variation on the general idea, and consider
\be{upper_bound_can_restrict}\begin{aligned} \Prob\big\{ Z_t =& Z_{t(1+\theta_t)} \big\}  \leq \Prob \big\{ Z_t = Z_{t(1+\theta_t)},\, 
\eta_t \leq \tfrac{\xi(Z_t)}{a_t} - q\tfrac{|Z_t|}{r_t}(1-\delta_t) + \delta_t < N_t \big\} \\ 
& + \Prob\big\{ \Phi_t(Z_t) < \eta_t a_t \big\} + \sum_{z\in\Z^d}\Prob\big\{\tfrac{\xi(z)}{a_t} - q\tfrac{|z|}{r_t}(1-\delta_t) + \delta_t 
\ge N_t \big\} \, . \end{aligned} \ee
By Lemma~\ref{restricting_maximizer}(c) and the proof of (b), we have that
\[ \Prob\{ \Phi_t(Z_t) < \eta_t a_t \} + \sum_{z\in\Z^d}\Prob\Big\{\tfrac{\xi(z)}{a_t} - q\tfrac{|z|}{r_t}(1-\delta_t) + \delta_t \ge N_t \Big\} \leq C_1\big(e^{-C_2 \eta_t^{d-\alpha}} + N_t^{d-\alpha}\big) \, , \]
so that this error term is of order $o(\theta_t^{-d})$ if $\beta>0$ is large enough.

Now, we can unravel the definition of $Z_t$ being the maximizer of $\Phi_t$ (in particular we know $t \xi(Z_t) > |Z_t|$ and $\Phi_{t}(Z_t)$ is positive) and write 
\be{expression_in_growing_box} \begin{aligned} 
\Prob& \big\{ Z_t = Z_{t+\theta_t},\, \eta'_t a_t \leq \Phi_t(Z_t) \leq \eta_t a_t,\, |Z_t| \leq g_t r_t \big\} \\
& = \int_{\eta_t}^{N_t} \sum_{z \in\Z^d} \Prob\left\{\ba{cc} \Phi_t(\oz) \leq \Phi_t(z) & \mbox{for } \oz \mbox{ with } t \xi(\oz) > |\oz|; \\ \Phi_{t(1+\theta_t)}(\oz) \leq \Phi_{t(1+\theta_t)}(z) & \mbox{for } \oz \mbox{ with } t(1+\theta_t) \xi(\oz) > |\oz|; \\
t \xi(z) > |z|  & \mbox{for }\frac{\xi(z)}{a_t} - q\frac{|z|}{r_t}(1-\delta_t) + \delta_t \in \diff y \ea\right\} 
\, . \end{aligned}\ee
%We continue by finding an upper bound for the latter probability. 
Let $z$ be such that $|z| < g_t r_t$, and $ \frac{\xi(z)}{a_t} - q\frac{|z|}{r_t}(1-\delta_t) + \delta_t = y < N_t$. For any $\oz$ 
with $|\oz| < g_t r_t$ and
\[ \ba{cl} \Phi_t(\oz) \leq \Phi_t(z) & \mbox{if } t \xi(\oz) > |\oz|, \\
\Phi_{t(1+\theta_t)}(\oz) \leq \Phi_{t(1+\theta_t)}(z) & \mbox{if } t(1+\theta_t) \xi(\oz)> |\oz|, \ea \]
we can deduce from Lemma~\ref{exact_errors_for_Phi} that
\[ \ba{cl} \frac{\xi(\oz)}{a_t} - q\frac{|\oz|}{r_t}(1+\delta_t)\leq y 
& \mbox{if } t \xi(\oz) > |\oz|, \\
\frac{\xi(\oz)}{a_t} - \frac{q}{1+\theta_t}\frac{|\oz|}{r_t}(1+\delta_t) \leq y + \frac{q\theta_t}{1+\theta_t} \frac{|z|}{r_t}(1- \delta_t) 
& \mbox{if } t(1+\theta_t) \xi(\oz)> |\oz|.  \ea \]
Recalling that $r_t \log t = t a_t$ it is easy to see that the inequalities on the left hold automatically for 
sufficiently large~$t$, if the conditions on the right are violated.
Therefore, using the independence of the $\xi(z)$, we get an upper bound on the expression in~(\ref{expression_in_growing_box}), 
\be{discrete_products} \begin{aligned} \int_{\eta_t}^{N_t}  \sum_{z\in\Z^d}& \Prob\big\{ \tfrac{\xi(z)}{a_t}-q\tfrac{|z|}{r_t}(1-\delta_t) + \delta_t  \in \diff y \big\} 
\prod_{\heap{\oz\in\Z^d}{|\oz|< |z|}} \Prob\big\{ \tfrac{\xi(\oz)}{a_t} \leq y + q\tfrac{|\oz|}{r_t}(1+\delta_t)\big\}  \\
&  \times \prod_{\heap{\oz\in\Z^d}{|z|<|\oz|< r_t g_t}} \Prob\big\{ \tfrac{\xi(\oz)}{a_t} - \tfrac{q}{1+\theta_t}\tfrac{|\oz|}{r_t}(1+\delta_t) \leq y + \tfrac{q\theta_t}{1+\theta_t} \tfrac{|z|}{r_t}(1- \delta_t) \big\} \, . \end{aligned}\ee
We now require that $\beta' < 1$, so that $\delta_t\eta_t^{-1} \ra 0$. 
In the following steps, we treat each of the products in the above expression separately. 
First of all, as $\xi(0)$ is Pareto-distributed,
\[\begin{aligned} \frac{1}{\diff y}\, \Prob\big\{ \tfrac{\xi(z)}{a_t}-q\tfrac{|z|}{r_t}(1-\delta_t) + \delta_t  \in \diff y \big\} 
& = \alpha \, a_t^{-\alpha} \big(y + q\tfrac{|z|}{r_t}(1-\delta_t) - \delta_t \big)^{-(\alpha+1)}  \\
& \leq (1-\delta_t \eta_t^{-1})^{-(\alpha+1)} \alpha\, a_t^{-\alpha} \big(y + q\tfrac{|z|}{r_t} \big)^{-(\alpha+1)} \, . \end{aligned}\]
For the second expression in~(\ref{discrete_products}), we find that for all $y > \eta_t$, we know that $a_t y > a_t \eta_t > 1$, 
assuming that $t$ is large enough.
In particular, we can use the approximation $\log (1-x) < -x$ for $x<1$ to obtain uniformly for all $y > \eta_t$ and all $z$,
\[\begin{aligned} \prod_{\heap{\oz\in\Z^d}{|\oz|< |z|}} \Prob\big\{ &\tfrac{\xi(\oz)}{a_t} \leq y + q\tfrac{|\oz|}{r_t}(1+\delta_t)\big\} 
 \leq \exp \Big\{ \sum_{\heap{\oz\in\Z^d}{|\oz| < |z|}} \log \big( 1- a_t^{-\alpha} \big(y + q\tfrac{|\oz|}{r_t}(1+\delta_t)\big)^{-\alpha}\big) \Big\}  \\	
& \leq \exp \Big\{ - \sum_{\heap{\oz\in\Z^d}{|\oz| < |z|}} a_t^{-\alpha} \big(y + q\tfrac{|\oz|}{r_t}(1+\delta_t)\big)^{-\alpha} \Big\} \\
& \leq \exp \Big\{ - 	(1+\delta_t)^{-\alpha} \int_{|\oz|<|z|-\frac{1}{2}} r_t^{-d} \big(y + q\tfrac{|\oz|+\frac{1}{2}}{r_t}\big)^{-\alpha} \diff \oz \Big\} \\
& \leq 
%\exp \Big\{ - (1+\delta_t)^{-\alpha} \Big((1+q\eta_t^{-1}r_t^{-1})^{-\alpha} \int_{|\ox|<\frac{|z|}{r_t}} (y +q|\ox|)^{-\alpha}  \diff \ox - %r_t^{-d}\eta_t^{-\alpha}\Big)\Big\} \\ & = 
(1+o(1)) \exp \Big\{ -(1+o(1)) \int_{|\ox|<\frac{|z|}{r_t}} (y +q|\ox|)^{-\alpha}  \diff \ox  \Big\}
\, ,\end{aligned} \]
where our assumptions on $\eta_t$ guarantee that all the error terms are of order $o(1)$. Finally, we consider the last product in~(\ref{discrete_products}), and a similar calculation to above shows that uniformly in $y \geq \eta_t$ and for all $z \in \Z^d$,
\[ \begin{aligned} \prod_{\heap{\oz\in\Z^d}{|z|<|\oz|< r_t g_t}} & \Prob\big\{ \tfrac{\xi(\oz)}{a_t} - \tfrac{q}{1+\theta_t}\tfrac{|\oz|}{r_t}(1+\delta_t) \leq y + \tfrac{q\theta_t}{1+\theta_t} \tfrac{|z|}{r_t}(1- \delta_t) \big\} \\
%& \leq \exp\Big\{ -(1+\delta_t)^{-\alpha}\sum_{|z| < |\oz| < r_t g_t} r_t^{-d} \big(   y + \tfrac{q\theta_t}{1+\theta_t} \tfrac{|z|}{r_t} + %\tfrac{q}{1+\theta_t}\tfrac{|\oz|}{r_t}\big)^{-\alpha} \Big\} \\
& \leq (1+o(1)) \exp \Big\{ -(1+o(1)) \int_{\tfrac{|z|}{r_t}\leq |\ox| \leq g_t} \big(y + \tfrac{q\theta_t}{1+\theta_t} \tfrac{|z|}{r_t} + \tfrac{q}{1+\theta_t} |\ox| \big)^{-\alpha} \diff \ox \Big\} 
\, . \end{aligned} \]
Combining these estimates to bound~(\ref{discrete_products}) and thus~(\ref{expression_in_growing_box}), we obtain
\[ \begin{aligned} \Prob\big\{ & Z_t = Z_{t+\theta_t},\, \eta'_t a_t \leq \tfrac{\xi(Z_t)}{a_t} - q\tfrac{|Z_t|}{r_t}(1-\delta_t) + \delta_t \leq \eta_t a_t,\, |Z_t| \leq g_t r_t \big\} \\
& \leq (\alpha+o(1)) \int_{\eta_t}^{N_t} \sum_{z \in \Z^d} r_t^{-d} 
\exp\Big\{-(1+o(1))\int_{|\ox|<\frac{|z|}{r_t}} (y +q|\ox|)^{-\alpha}  \diff \ox\Big\}\\
&  \qquad\times \exp \Big\{ -(1+o(1)) \int\limits_{\frac{|z|}{r_t}\leq |\ox| \leq g_t} \!\!\big(y + \tfrac{q\theta_t}{1+\theta_t} \tfrac{|z|}{r_t} + \tfrac{q}{1+\theta_t} |\ox| \big)^{-\alpha} \diff \ox \Big\} \big(y + q\tfrac{|z|}{r_t} \big)^{-(\alpha+1)} \diff y \\
& \leq (1+o(1)) \int_{\eta_t}^{N_t} \int_{x \in\R^d}
\exp\Big\{-(1+o(1))\int_{|\ox|<|x|} (y +q|\ox|)^{-\alpha}  \diff \ox\Big\}\\
&  \qquad \times \exp \Big\{ -(1+o(1)) \int_{|x|\leq |\ox| \leq g_t} \big(y + \tfrac{q\theta_t}{1+\theta_t} |x| + \tfrac{q}{1+\theta_t} |\ox| \big)^{-\alpha} \diff \ox \Big\} \frac{\alpha \, \diff x \, \diff y}{(y + q|x| )^{\alpha+1}} \, , 
\end{aligned} \]
where, as before, the approximation of the sum by an integral works because $\eta_t a_t \ra \infty$.	Note also that, uniformly in $x$ and $y$,
\[ \int_{|\ox|\geq g_t} \big(y + \tfrac{q\theta_t}{1+\theta_t}|x| + \tfrac{q}{1+\theta_t}|\ox|\big)^{-\alpha} \diff \ox
\leq (1+\theta_t)^{\alpha} q^{-\alpha}\int_{|\ox| \geq g_t} |\ox|^{-\alpha} \leq C' \theta_t^\alpha g_t^{d-\alpha} \, , \]
where $C'>0$ is some universal constant. Choosing $\gamma>0$ large enough ensures that this term tends to $0$. Hence, together with~(\ref{upper_bound_can_restrict}) we have shown that 
\[\begin{aligned} & \Prob\{ Z_t = Z_{t(1+\theta_t)} \}  \\
& \leq (1+o(1)) \int_{y >0} \int_{x \in\R^d} 
\exp\Big\{-(1+o(1)) \int_{|\ox|<|x|} (y +q|\ox|)^{-\alpha}  \diff \ox \Big\} \\
& \hspace{1cm} \exp \Big\{ -(1+o(1)) \int_{|x|\leq |\ox| } \big(y + \tfrac{q\theta_t}{1+\theta_t} |x| + \tfrac{q}{1+\theta_t} |\ox| \big)^{-\alpha} \diff \ox \Big\} \frac{\alpha \, \diff x \, \diff y}{(y + q|x| )^{\alpha+1}}
+ o(\theta_t^{-d}) \, .\end{aligned} \]

\emph{Lower bound.} Before we simplify the expression for the upper bound, we derive a similar expression for the lower bound. %let $g_t = (\log t)^{\gamma}, \eta_t = (\log t)^{-\beta'}$ and $N_t = (\log t)^{\beta}$  for $\gamma,\beta',\beta > 0$ to be determined later, and consider the lower bound
As in the upper bound, we follow the main idea and restrict our attention to large boxes and estimate
\be{lower_bound_1} \begin{aligned} \Prob\{ & Z_t = Z_{t(1+\theta_t)} \}  \geq \sum_{\heap{z\in\Z^d}{|z| \leq r_t g_t}} %\int_{\eta_t}^{N_t}
\Prob\big\{ Z_t = z = Z_{t+\theta t},\, \tfrac{\xi(z)}{a_t} - 2q\tfrac{|z|}{r_t} \leq N_t \big\} \\
& = \sum_{\heap{z\in\Z^d}{|z| \leq r_t g_t}}  \, 
\Prob\left\{\ba{cc} \Phi_t(\oz) \leq \Phi_t(z) & \mbox{for } \oz \mbox{ with } t \xi(\oz) > |\oz|; \\ \Phi_{t(1+\theta_t)}(\oz) \leq \Phi_{t(1+\theta_t)}(z) & \mbox{for } \oz \mbox{ with } t(1+\theta_t) \xi(\oz) > |\oz|; \\
t \xi(z) > |z| ; &\tfrac{\xi(z)}{a_t} - 2q\tfrac{|z|}{r_t} \leq N_t \ea\right\} . \end{aligned} \ee
%where $\delta_t = C\frac{\log \log t}{\log t}$, where $C$ is a constant which arises from Lemma~\ref{exact_errors_for_Phi} in the following way. If $z$ is such that $|z|\leq g_t r_t$ and $\frac{\xi(z)}{a_t} - q\frac{|z|}{r_t} < N_t$, then
The proof of Lemma~\ref{exact_errors_for_Phi} shows that if $z$ is such that $|z|\leq g_t r_t$ and $\frac{\xi(z)}{a_t} - 2q\frac{|z|}{r_t} \leq N_t$, then we\vspace{-1mm} can find $C>0$ such that with $\delta_t = C \frac{\log \log t}{\log t}$ we 
have that 
\[ \ba{c} \tfrac{\xi(z)}{a_t} - q\tfrac{|z|}{r_t}(1+\delta_t)    \leq  \tfrac{\Phi_t(z)}{a_t} \leq \tfrac{\xi(z)}{a_t} - q\tfrac{|z|}{r_t}(1-\delta_t) + \delta_t \mbox{ and } \\[1mm]
\tfrac{\xi(z)}{a_t} - \tfrac{q}{1+\theta}\tfrac{|z|}{r_t}(1+\delta_t)    \leq  \tfrac{\Phi_{t+\theta t}(z)}{a_t} \leq \tfrac{\xi(z)}{a_t} - \tfrac{q}{1+\theta}\tfrac{|z|}{r_t}(1-\delta_t) + \delta_t. \ea \]
Therefore, we can approximate~\eqref{lower_bound_1} further by
%we have that
%\[ \frac{\Phi_t(z)}{a_t} \geq \frac{\xi(z)}{a_t} - q \frac{|z|}{r_t}(1+\delta_t) \quad\mbox{ and }\quad
%\frac{\Phi_{t(1+\theta_t)}(z)}{a_t} \geq \frac{\xi(z)}{a_t} - \frac{q}{1+\theta_t} \frac{|z|}{r_t}(1+\delta_t) \, . \]
%Also, we require that $\delta_t$ is chosen such that for all $\oz$ such that $t\xi(\oz)>|\oz|$, 
%\[ \frac{\Phi_t(\oz)}{a_t} \leq \frac{\xi(\oz)}{a_t} - q \frac{|\oz|}{r_t}(1-\delta_t)+\delta_t \quad\mbox{ and }\quad
%\frac{\Phi_{t(1+\theta_t)}(\oz)}{a_t} \leq \frac{\xi(\oz)}{a_t} - \frac{q}{1+\theta_t} \frac{|\oz|}{r_t}(1-\delta_t) +\delta_t \, . \]
%In particular, we can obtain a lower bound on~(\ref{lower_bound_1}) 
\be{lower_bound_2}\begin{aligned} & \Prob\{Z_t=Z_{t(1+\theta_t)}\}  \\
& \geq 
\sum_{\heap{z\in\Z^d}{|z| \leq r_t g_t}} \int_{\eta_t}^{N_t} \, 
\Prob\left\{\ba{cc} \frac{\xi(\oz)}{a_t} - q \frac{|\oz|}{r_t}(1-\delta_t)+\delta_t \leq y %\frac{\xi(z)}{a_t} - q \frac{|z|}{r_t}(1+\delta_t) 
& \mbox{for } \oz\neq z; \\ \frac{\xi(\oz)}{a_t} - \frac{q}{1+\theta_t} \frac{|\oz|}{r_t}(1-\delta_t) +\delta_t \leq y + \frac{q\theta_t}{1+\theta_t}(1+\delta_t) %\frac{\xi(z)}{a_t} - \frac{q}{1+\theta_t} \frac{|z|}{r_t}(1+\delta_t) 
& \mbox{for } \oz \neq z; \\
\frac{\xi(z)}{a_t} - q\frac{|z|}{r_t}(1+\delta_t)  \in \diff y & \ea\right\}
\end{aligned} \ee
We now show that, depending on whether $|z| \leq |\oz|$ or $|z| > |\oz|$ one of the two conditions in the bracket
above is superfluous. Indeed, if $|\oz| \leq |z|$ and the first condition holds we can deduce that
\[ \frac{\xi(\oz)}{a_t} - \frac{q}{1+\theta_t} \frac{|\oz|}{r_t}(1-\delta_t) +\delta_t \leq 
y +\frac{q\theta_t}{1+\theta_t} \frac{|\oz|}{r_t}(1-\delta_t) \leq y + \frac{q\theta_t}{1+\theta_t}\frac{|z|}{r_t}(1+\delta_t) \, . \]
Conversely, if $|\oz| > |z|$ and we assume the second condition it follows that
\[ \frac{\xi(\oz)}{a_t} - q \frac{|\oz|}{r_t}(1-\delta_t)+\delta_t \leq y - \frac{q\theta_t}{1+\theta_t} \frac{|\oz|}{r_t} (1-\delta_t) + \frac{q\theta_t}{1+\theta_t}\frac{|z|}{r_t}(1-\delta_t) \leq y\, . \]
Hence, we have found a lower bound which can be expressed using the independence of the $\xi$ as 
%\be{lower_bound_3}
\begin{align} & \Prob\{Z_t=Z_{t(1+\theta_t)}\}  \notag\\
& \geq 
\sum_{\heap{z\in\Z^d}{|z| \leq r_t g_t}} \int_{\eta_t}^{N_t} \, \Prob\big\{ \tfrac{\xi(z)}{a_t} - q \tfrac{|z|}{r_t} (1+\delta)\in\diff y\big\}
\prod_{\heap{\oz\in\Z^d}{|\oz|< |z|}} \Prob\big\{ \tfrac{\xi(\oz)}{a_t} - q \tfrac{|\oz|}{r_t}(1-\delta_t)+\delta_t \leq y \big\} \notag\\
& \hspace{2cm} \times \prod_{\heap{\oz\in\Z^d}{|\oz|>|z|}} \Prob\Big\{ \tfrac{\xi(\oz)}{a_t} - \tfrac{q}{1+\theta_t} \tfrac{|\oz|}{r_t}(1-\delta_t) +\delta_t \leq y + \tfrac{q\theta_t}{1+\theta_t}\tfrac{|z|}{r_t}(1-\delta_t) \big\} \, .  \label{lower_bound_3}\end{align}
We use that $\log(1-x) \geq -x (1+x)$ for $0<x<1/2$ to see that
\[ \begin{aligned} \prod_{\heap{\oz\in\Z^d}{|\oz|< |z|}}  \Prob\big\{ & \tfrac{\xi(\oz)}{a_t} - q \tfrac{|\oz|}{r_t}(1-\delta_t)+\delta_t \leq y \big\} \\[-2mm]
& = \exp \Big\{ \sum_{\heap{\oz\in\Z^d}{|\oz|<|z|}} \log \big( 1 - a_t^{-\alpha} \big(y + q \tfrac{|\oz|}{r_t}(1-\delta_t)-\delta_t\big)^{-\alpha}\big) \Big\} \\
%& \geq \exp \Big\{ -\sum_{|\oz|<|z|} (1+\eta_t^{-\alpha}a_t^{-\alpha}) \Big(y + q \frac{|\oz|}{r_t}(1-\delta_t)-\delta_t\Big)^{-\alpha} \Big\} \\
%& \geq \exp \Big\{ -(1+\eta_t^{-\alpha}a_t^{-\alpha})(1-\delta_t\eta_t^{-1})^{-\alpha}\sum_{|\oz|<|z|}  \Big(y + q \frac{|\oz|}{r_t}\Big)^{-\alpha} \Big\} \\
& \geq (1+o(1))\exp \Big\{ -(1+o(1)) \int_{|\ox|<\frac{|z|}{r_t}} (y + q|\ox|)^{-\alpha} \diff \ox \Big\} \, , \end{aligned} \]
where $o(1)$ tends to $0$ uniformly in $y\geq \eta_t$ and all $x \in \Z^d$. Similarly as in the upper bound, we can deal with the other products in~(\ref{lower_bound_3}) and approximate the sums by integrals to obtain
\[ \begin{aligned}  \Prob&\{Z_t=Z_{t(1+\theta_t)}\}  \\
& \geq  (1+o(1))\int_{|x| \leq g_t} \int_{\eta_t}^{N_t} 
\exp\Big\{-(1+o(1)) \int_{|\ox|<|x|} (y +q|\ox|)^{-\alpha}  \diff \ox \Big\} \\
& \hspace{1cm} \times\exp \Big\{ -(1+o(1)) \int_{|x|\leq |\ox|} \big(y + \tfrac{q\theta_t}{1+\theta_t} |x| + \tfrac{q}{1+\theta_t} |\ox| \big)^{-\alpha} \diff \ox \Big\} \frac{\alpha\, \diff x \, \diff y}{(y + q|x| )^{\alpha+1}}
 \, ,\end{aligned} \]
which is almost the same expression as for the upper bound. In order to control the difference, we first estimate
\[\begin{aligned}  \int_{|x| \geq g_t} \int_{\eta_t}^{N_t} & \exp\Big\{-(1+o(1)) \int_{|\ox|<|x|} (y +q|\ox|)^{-\alpha}  \diff \ox \Big\} \\
& \hspace{1cm} \times\exp \Big\{ -(1+o(1)) \int_{|x|\leq |\ox|} \big(y + \tfrac{q\theta_t}{1+\theta_t} |x| + \tfrac{q}{1+\theta_t} |\ox| \big)^{-\alpha} \diff \ox \Big\} \frac{\alpha \, \diff x \, \diff y}{(y + q|x| )^{\alpha+1}} \\
& \leq \int_{|x| \geq g_t} \int_{\eta_t}^{N_t} \exp\Big\{-(1+o(1)) \int (y +q|\ox|)^{-\alpha}  \diff \ox \Big\}
\frac{\alpha \, \diff x \, \diff y}{(y + q|x| )^{\alpha+1}} \\
& = \tfrac{\alpha 2^d}{(d-1)!} \int_{\eta_t}^{N_t} \int_{r\geq g_t} e^{-(1+o(1)) \vartheta y^{d-\alpha}} \frac{r^{d-1}\diff r \, \diff y}{(y + qr)^{\alpha+1}} \, ,
\end{aligned} \]
where we used the same simplification as in Proposition~\ref{I} and $\vartheta = \frac{2^d B(\alpha-d,d)}{q^d (d-1)!}$. 
Similarly, %with $\vartheta=\frac{2^d B(\alpha-d,d)}{q^d (d-1)!}$ 
we get an upper bound of\vspace{-2mm}
\[ \begin{aligned} & \leq \tfrac{\alpha 2^d}{(d-1)!q^d} \int_0^\infty e^{-(1+o(1))\vartheta y^{d-\alpha}} y^{d-\alpha-1} \int_0^{N_t/(qg_t)} v^{\alpha-d}(1-v)^{d-1} \diff v \, \diff y \\
& \leq (1+o(1))\tfrac{\alpha}{B(\alpha-d,d)(\alpha-d)} \big(\tfrac{N_t}{q g_t}\big)^{\alpha-d+1} \, . \end{aligned} \]
Making $\beta>0$ larger depending on $\theta_t$, and then choosing $\gamma>0$ large depending on $\beta$ and $\theta_t$, 
we can ensure that this term is of order $o(\theta_t^{-d})$. Similar calculations yield
\[\begin{aligned}  \int_{x \in \R^d} 
\int_0^{\eta_t} & \exp\Big\{-(1+o(1)) \int_{|\ox|<|x|} (y +q|\ox|)^{-\alpha}  \diff \ox \Big\} \\
& \hspace{.5cm} \times\exp \Big\{ -(1+o(1)) \int_{|x|\leq |\ox|} \big(y + \tfrac{q\theta_t}{1+\theta_t} |x| + \tfrac{q}{1+\theta_t} |\ox| \big)^{-\alpha} \diff \ox \Big\} \frac{\alpha \, \diff x \, \diff y}{(y + q|x| )^{\alpha+1}} \\
%& \leq \int_{x\in\R^d} 
%\int_0^{\eta_t} e^{-(1+o(1))\int_{\ox \in \R^d} (y + q|\ox|)^{-\alpha}} \frac{\alpha \, \diff x \, \diff y}{(y + q|x| )^{\alpha+1}} \\
%& = \tfrac{\alpha 2^d}{(d-1)!q^d} \int_0^{\eta_t}  e^{-(1+o(1)) \vartheta y^{d-\alpha}} y^{d-\alpha-1} \int_0^1 v^{\alpha-d}(1-v)^{d-1}\\
& \leq \tfrac{\alpha 2^d B(\alpha-d+1,d)}{(d-1)!q^d} \, e^{-(1+o(1))\vartheta\eta_t^{d-\alpha}} \eta_t^{d-\alpha} \, ,    \end{aligned} \]
which is of order $o(\theta_t^{-d})$, and
\[\begin{aligned}  \int_{x \in \R^d} & \int_{N_t}^\infty \exp\Big\{-(1+o(1)) \int_{|\ox|<|x|} (y +q|\ox|)^{-\alpha}  \diff \ox \Big\} \\
&  \times\exp \Big\{ -(1+o(1)) \int_{|x|\leq |\ox|} \big(y + \tfrac{q\theta_t}{1+\theta_t} |x| + \tfrac{q}{1+\theta_t} |\ox| \big)^{-\alpha} \diff \ox \Big\} \frac{\alpha \, \diff x \, \diff y}{(y + q|x| )^{\alpha+1}} 
%& \leq \int_{x\in\R^d} \int_{N_t}^\infty %e^{-(1+o(1))\int_{\ox \in \R^d} (y + q|\ox|)^{-\alpha}} 
%\frac{\alpha \, \diff x \, \diff y}{(y + q|x| )^{\alpha+1}} \\
%& \leq \frac{\alpha 2^d}{(d-1)!q^d}  \int_{N_t}^\infty \int_0^1 v^{\alpha-d}(1-v)^{d-1} y^{\alpha-d-1} \diff v \, \diff y
 \leq C N_t^{\alpha-d} \, , \end{aligned} \]
for some constant $C>0$, which by choice of $\beta>0$ is also of order $o(\theta_t^{-d})$.
\pagebreak[3]

\emph{Final step.} Combining the upper and lower bound we have shown that
\[\begin{aligned} & \Prob\{ Z_t = Z_{t(1+\theta_t)} \}  \\
& = (1+o(1)) \int_{y >0} \int_{x \in\R^d} 
\exp\Big\{-(1+o(1)) \int_{|\ox|<|x|} (y +q|\ox|)^{-\alpha}  \diff \ox \Big\} \\
& \hspace{1cm} \exp \Big\{ -(1+o(1)) \int_{|x|\leq |\ox| } \big(y + \tfrac{q\theta_t}{1+\theta_t} |x| + \tfrac{q}{1+\theta_t} |\ox| \big)^{-\alpha} \diff \ox \Big\} \frac{\alpha \, \diff x \, \diff y}{(y + q|x| )^{\alpha+1}}
+ o(\theta_t^{-d}) \, .\end{aligned} \]
Simplifying the integrals as in Proposition~\ref{I}, we obtain that
$\Prob\{ Z_t = Z_{t(1+\theta_t)} \} = (1+$\vspace{-1mm} $o(1))I(\theta_t)+o(\theta_t^{-d})$,
and an appeal to Proposition~\ref{asymptotics_I_large_c} completes the proof.
\end{proof}
\smallskip

\begin{rem}\label{uniform_in_c} In fact, the proof of Proposition~\ref{moderate_deviations} even shows a slightly stronger statement. Namely, let $\gamma > 0$ and suppose $\ell_t$ is a function such that $\ell_t \ra \infty$ as $t\ra\infty$. Then for any $\eps > 0$, there exists $T>0$ such that for all $t\geq T$ and all $\ell_t \leq \theta \leq (\log t)^\gamma$, we have that
\[ (1-\eps) \tfrac{1}{d \, B(\alpha-d-1)} \, \theta^{-d} \leq \Prob\{ Z_t = Z_{t+\theta t} \} \leq (1+\eps) \tfrac{1}{d \, B(\alpha-d-1)} 
\, \theta^{-d} \, .\]
\end{rem}

As indicated in Section~\ref{se.guide} the previous proposition suffices to prove the upper bound in Theorem~\ref{asymptotics_R_X}. 
For the lower bound we also need to control the decay of correlations.

\begin{lemma}\label{asymptotic_independence}%\label{correlations_c_large} 
Let $\theta_t$ be a positive, nondecreasing function such that $\theta_t\ra\infty$ as $t\ra\infty$ and for some $\delta >0$, %$(\log t)^{\eps}\leq 
$\theta_t \leq (\log t)^{\delta}$ for all $t>0$. Then, for any $t >0$ and $s \geq (1+\theta_t)t$,
\[ \Prob\{ Z_t = Z_{t(1+\theta_t)} \neq Z_s = Z_{s(1+\theta_s)} \} \leq (1+o(1))\frac{1}{d^2 B(\alpha-d+1,d)^2}\theta_t^{-d}\theta_s^{-d}  \, , \]
where $o(1)$ is an error term that vanishes as $t\ra\infty$.
\end{lemma}

\begin{proof} %In the first step, fix $g_t$ such that \comment{add condition}
%\[\begin{aligned} \Prob\{ Z_t = Z_{t(1+\theta_t)} & \neq Z_s = Z_{s(1+\theta_s)} \} \\
%& = \sum_{z_1\in\Z^d} \sum_{z_2 \in\Z^d\setminus\{z_1\}}  \Prob\{ Z_t = Z_{t(1+\theta_t)} = z_1 ; Z_s = Z_{s(1+\theta_s)}=z_2 \}  \\
%& = \sum_{k=1}^5 \sum_{(z_1,z_2) \in \calR_k} \Prob\{ Z_t = Z_{t(1+\theta_t)} = z_1 ; Z_s = Z_{s(1+\theta_s)}=z_2 \} 
%& \sum_{|z_1| \leq g_t r_t} \sum_{g_t r_t \leq |z_2| \leq g_s r_s} + \sum_{|z_1| > g_t r_t} \sum_{z_2 \neq z_1}
%+ \sum_{|z_1| \leq g_t r_t} \sum_{
%\, , \end{aligned} \]
%where we split up the sums over $\{(z_1,z_2) \, : \, z_1 \neq z_2 \}$ in the following way
%\[\ba{l} \calR_1 = \{ (z_1,z_2) \, : \, |z_1|<g_t r_t; g_t r_t < |z_2| < g_s r_s \} \\
%\calR_2 = \{ (z_1,z_2) \, : \, z_1 \neq z_2; |z_1| \geq g_t r_t ; |z_2| < g_s r_s \} \\
%\calR_3 = \{ (z_1,z_2) \, : \, z_1 \neq z_2; |z_1| \geq g_t r_t ; |z_2| \geq g_s r_s \} \\
%\calR_4 = \{ (z_1,z_2) \, : \, |z_1| < g_t r_t ; |z_2| \leq g_t r_t \} \\
%\calR_5 = \{ (z_1,z_2) \, : \, |z_1| < g_t r_t ; |z_2| \geq g_s r_s \} \, .  \ea \]
%We will show that the sum over $\calR_1$ gives the main contribution, while the sums over the other regions will give a contribution which is of order at most $o(\theta_t^{-d}\theta_s^{-d})$.
%
%\emph{Contributions from $\calR_2$.} By Lemma~\cite[Lemma 3.2]{KLMS09}, we know that $\Phi_t(Z_t)>0$. We can use Lemma~\ref{exact_errors_for_Phi}
We use a similar notation as in the proof of Proposition~\ref{moderate_deviations}. In particular, we will choose functions $g_t,\eta_t, N_t$ depending on $\theta_t$. Also, let $\delta_t = C\frac{\log \log t}{\log t}$, where $C$ is the constant implied in the error bounds in Lemma~\ref{exact_errors_for_Phi}. A lengthy routine calculation similar to Lemma~\ref{restricting_maximizer}
shows that
\be{correlations_first_error}\begin{aligned} \Prob\{ Z_t = Z_{t(1+\theta_t)} & \neq Z_s = Z_{s(1+\theta_s)} \} \\
& = \Prob\left\{\ba{c} Z_t = Z_{t(1+\theta_t)}  \neq Z_s = Z_{s(1+\theta_s)}; \\[1mm]
\frac{\xi(Z_t)}{a_t} - q\frac{|Z_t|}{r_t}(1-\delta_t) + \delta_t \in [\eta_t,N_t];\\[1mm] \frac{\xi(Z_s)}{a_s} - q\frac{|Z_s|}{r_s}(1-\delta_s) + \delta_s \in [\eta_s,N_s] \ea  \right\} + {\rm error}(s,t) \, , \end{aligned} \ee
where, for some constants $C_1,C_2>0$,
\[ \mbox{error}(t,s) \leq C_1 ( e^{-C_2 \eta_t^{d-\alpha}} + N_t^{d-\alpha})( e^{-C_2 \eta_s^{d-\alpha}} + \theta_s^{-d} + N_s^{d-\alpha}) + C_1 \theta_t^{-d}(e^{-C_2 \eta_s^{d-\alpha}} + N_s^{d-\alpha})\, . \]
Taking $N_t = \theta_t^{q+3/2}$ and $\eta_t = \theta_t^{-\beta'}$ for
$\beta'>0$ ensures that the error is of order $o(\theta_t^{-d}\theta_s^{-d})$. 
We can therefore focus the probability on the right hand side of~(\ref{correlations_first_error}). 
Using Lemma~\ref{exact_errors_for_Phi}, we find the following upper bound
\[ \begin{aligned} &\Prob  \left\{\ba{c} Z_t = Z_{t(1+\theta_t)}  \neq Z_s = Z_{s(1+\theta_s)}\, ; \\
\frac{\xi(Z_t)}{a_t} - q\frac{|Z_t|}{r_t}(1-\delta_t) + \delta_t \in [\eta_t,N_t]\, ;\\ \frac{\xi(Z_s)}{a_s} - q\frac{|Z_s|}{r_s}(1-\delta_s) + \delta_s \in [\eta_s,N_s] \ea  \right\} \\
& \leq \sum_{z_1\in \Z^d}\sum_{z_2\in\Z^d\setminus\{z_1\}} \Prob\left\{\ba{c} \Phi_{t(1+\theta_t)}(\oz) \leq \Phi_{t(1+\theta_t)}(z_1) \quad \forall |\oz| \leq r_t g_t \mbox{ with } \oz \neq z_1,z_2\, ;\\
\Phi_{s(1+\theta_s)}(\oz) \leq \Phi_{s(1+\theta_s)}(z_2) \quad \forall r_t g_t < |\oz| \leq r_s g_s\mbox{ with } \oz \neq z_1,z_2\, ; \\
\frac{\xi(z_1)}{a_t} - q\frac{|z_1|}{r_t}(1-\delta_t) + \delta_t \in [\eta_t,N_t]\, ;\\ \frac{\xi(z_2)}{a_s} - q\frac{|z_2|}{r_s}(1-\delta_s) + \delta_s \in [\eta_s,N_s] \ea  \right\} \, , \end{aligned} \]
which, taking $g_t = \theta_t^{q+3/2}$ and using the independence, we can estimate as
\[ \begin{aligned}
& \leq \sum_{z_1\in \Z^d}\sum_{z_2\in\Z^d\setminus\{z_1\}} \int_{\eta_t}^{N_t}  \int_{\eta_s}^{N_s} 
\prod_{\heap{|\oz| < g_t r_t}{\oz\neq z_1,z_2}} \Prob\Big\{ \tfrac{\xi(\oz)}{a_t} - q\tfrac{|\oz|}{r_t}(1+\delta_t) \leq y_1 + \tfrac{q\theta_t}{1+\theta_t}\tfrac{|z_1|}{r_t}(1-\delta_t) \Big \} \\
& \hspace{1cm} \times\prod_{\heap{g_t r_t < |\oz| < g_s r_s}{\oz \neq z_1,z_2}} \Prob\Big\{ \tfrac{\xi(\oz)}{a_s} - q\tfrac{|\oz|}{r_s}(1+\delta_s) \leq y_2 + \tfrac{q\theta_s}{1+\theta_s}\tfrac{|z_2|}{r_s}(1-\delta_s) \Big \} \\ 
&\hspace{1cm}\times \Prob\Big\{ \tfrac{\xi(z_1)}{a_t} - q\tfrac{|z_1|}{r_t}(1-\delta_t) + \delta_t \in \diff y_1 \Big\} 
\Prob\Big\{\tfrac{\xi(z_2)}{a_s} - q\tfrac{|z_2|}{r_s}(1-\delta_s) + \delta_s \in \diff y_2 \Big\}. 
\end{aligned} \]
As before, we can work out the probabilities, and approximate the sums by integrals to finally obtain $(1+o(1))$ times
\be{correlations_factorize} \begin{aligned} & \int \int_{\eta_t}^{N_t} 
\exp \Big\{ -(1+o(1)) \int_{|\ox|< g_t} \Big(y_1 + \tfrac{q}{1+\theta_t}|\ox| +\tfrac{q \theta_t}{1+\theta_t}|x_1|\big)^{-\alpha} \diff \ox \Big\} \frac{\alpha \, \diff x_1\, \diff y_1}{(y_1+q|x_1|)^{\alpha +1}} \\
&\times \int \int_{\eta_s}^{N_s} \exp\Big\{ -(1+o(1)) \hspace{-0.9cm}\int\limits_{g_tr_t/r_s <|\ox|< g_s} \hspace{-0.5cm}\big(y_2 + \tfrac{q}{1+\theta_s}|\ox| +\tfrac{q \theta_s}{1+\theta_s}|x_2|\big)^{-\alpha} \diff \ox \Big\} 
\frac{\alpha \, \diff x_2 \, \diff y_2}{(y_2+q|x_2|)^{\alpha +1}} \, . \end{aligned}\ee
In the remainder of the proof, we have to show that the first term is of order $\theta_t^{-d}$, whereas the second is of order $\theta_s^{-d}$. 
The integral in the first factor equals in polar coordinates
\[\begin{aligned} \tfrac{2^d}{(d-1)!} & \int\limits_{0<r< g_t} \Big(y_1 + \tfrac{q}{1+\theta_t}r +\tfrac{q \theta_t}{1+\theta_t}|x_1|\big)^{-\alpha} r^{d-1} \diff r \\
& \geq \tfrac{2^d}{(d-1)!}(1+\theta_t)^d \int\limits_{0<r< \frac{g_t}{1+\theta_t}} (y_1 + qr + q |x_1|\big)^{-\alpha} r^{d-1}\diff r \\
& = \tfrac{2^d}{(d-1)!}(1+\theta_t)^d \Big\{ q^{-\alpha}(y_1+q|x_1|)^{d-\alpha} B(\alpha-d,d) - \!\!\!\!\int\limits_{r> \frac{g_t}{(1+\theta_t)}}\hspace{-.2cm} (y_1 + qr + q |x_1|\big)^{-\alpha} r^{d-1}\diff r \Big\}  
\, . \end{aligned} \]
The subtracted integral is bounded from above by $q^{-\alpha} g_t^{d-\alpha}(1+\theta_t)^{\alpha-d}$ and therefore, by our assumptions, 
together with the $(1+\theta_t)^d$ factor tends to zero. 
Hence we can conclude that, with $\vartheta$ as before, the first factor in~(\ref{correlations_factorize}) is bounded from above by 
\[\begin{aligned}	 (1&+o(1)) \, \int \int_{\eta_t}^{N_t} e^{-(1+o(1)) \vartheta (1+\theta_t)^d(y_1+q|x_1|)^{d-\alpha}} \frac{\alpha\, \diff y_1 \, \diff x_1}{(y_1+q|x_1|)^{\alpha+1}} \\
&\le  (1+o(1)) \,  \tfrac{2^d}{(d-1)!} \int_0^\infty \int_0^\infty e^{-(1+o(1)) \vartheta (1+\theta_t)^d(y+qr)^{d-\alpha}} \frac{\alpha \, r^{d-1}\diff y_1 \, \diff r}{(y_1+qr)^{\alpha+1}} \\
& \leq  (1+o(1)) \, \tfrac{\alpha 2^d}{(d-1)!q^d} \int_{0}^\infty \int_0^1 e^{-(1+o(1))\vartheta(1+\theta_t)^d y^{d-\alpha}v^{\alpha-d}} y^{\alpha-d-1}v^{\alpha-d}(1-v)^{d-1} \diff v \, \diff y \\
& = (1+o(1)) \,(1+\theta_t)^{-d}\tfrac{\alpha}{(\alpha-d)B(\alpha-d,d)}\int_0^1 (1-v)^{d-1} \diff v = (1+o(1)) \,\theta_t^{-d} \tfrac{1}{d B(\alpha-d+1,d)} \, . \end{aligned}  \]
For the second factor in~(\ref{correlations_factorize}), we almost get the same expression, and it suffices to consider the following term and, 
using similar arguments as above, we can estimate uniformly in $y_2 \geq \eta_s$,
\[\begin{aligned}  \int_{|\ox|<\frac{g_tr_t}{r_s}} & (y_2 + \tfrac{q}{1+\theta_s} |\ox| + q|x_2|)^{-\alpha}\diff \ox \\
%& = (1+\theta_s)^d \tfrac{2^d}{(d-1)!} \int_{r< \frac{g_t r_t}{r_s(1+\theta_s)}} (y_2 + q r + q|x_2|)^{-\alpha}r^{d-1} \diff r \\
%& = (1+\theta_s)^d \tfrac{2^d}{q^d (d-1)!} (y_2 + q|x_2|)^{d-\alpha} \int_{\frac{y_2+q|x_2|}{y_2+q|x_2|+q\frac{g_tr_t}{r_s(1+\theta_s)}}}^1 %u^{\alpha-d-1}(1-u)^{d-1} \diff u \\
& \leq (1+\theta_s)^d \tfrac{2^d}{q^d (d-1)!} (y_2 + q|x_2|)^{d-\alpha} \int_{1-\frac{q g_t r_t}{r_s(1+\theta_s)\eta_s}}^1 u^{\alpha-d-1}(1-u)^{d-1} \diff u
\, . \end{aligned} \]
%Now, we claim that the latter integral converges to $0$. Indeed, 
Using that $s/t \geq (1+\theta_t)$ and recalling that $\eta_t = \theta_t^{-\beta'}$, where we can assume $0<\beta'<1$ and $g_t = \theta_t^{q+3/2}$, we obtain
\[ \frac{g_t r_t}{r_s(1+\theta_s)\eta_s} \leq \frac{g_t (\log t + \log (1+\theta_t))^{q+1} }{(\log t)^{q+1} \theta_t^{q+2-\beta'}} 
\leq (1+o(1))\theta_t^{\beta'-\frac{1}{2}} \, , \]
so that, by choosing $\beta'<\frac12$, this term tends to $0$. Now, we can simplify the second factor in~(\ref{correlations_factorize}) in the same way as the first one to show that it is of the required form.
%\[\begin{aligned}  \int_{|\ox|<\frac{g_tr_t}{r_s}} & (y_2 + \frac{q}{1+\theta_s} |\ox| + \frac{q\theta_s}{1+\theta_s}|x_2|)^{-\alpha}\diff \ox 
%\leq \eta_s^{-\alpha} \Big(\frac{g_t r_t}{r_s}\Big)^d \leq \Big(\frac{g_t r_t}{r_{t (1+\theta_t)^{(1+\eps)}} \eta_{t(1+\theta_t)^{(1+\eps)}}^{\alpha/d}}\Big)^d  \\
%& \leq \Big( \frac{g_t (\log t + (1+\eps)\log(1+\theta_t))^{\frac{\alpha}{d}\beta' + q + 1} }{(1+\theta_t)^{(1+\eps)(q+1)} (\log t)^{q+1}} \Big)^{d}\, ,\end{aligned} \]
%which by our assumptions on $g_t$ and $\eta_t$ tends to zero, thus completing the proof.
\end{proof}

\subsection{Almost sure asymptotics for the maximizer of~$\Phi_t$}\label{asymptotics_tau}

In analogy with the residual lifetime function $R$ for the process $X_t$, we can also define the residual lifetime function $R^V$ for the maximizer $Z_t$ of the variational problem, by setting
\[ R^V(t) = \sup \{ s \geq 0 \, : \, Z_t = Z_{t + s} \} \, . \]
Using the moderate deviation principle, Proposition~\ref{moderate_deviations}, developed in the previous section together with the Borel-Cantelli lemma, we aim to prove the following analogue of Theorem~\ref{asymptotics_R_X}.

\begin{prop}\label{limsup_asymptotics}
For any nondecreasing function $h:(0,\infty)\ra(0,\infty)$ %$\varphi\colon[1,\infty)\to(0,\infty)$ 
we have, almost surely, 
$$\limsup_{t \ra\infty} \frac{R^V(t)}{t h(t)} = 
\left\{ \begin{array}{ll} 0 & \mbox{ if } \displaystyle \int_1^\infty \frac{\d t}{t h(t)^d} <\infty,\\[3mm]
\infty & \mbox{ if } \displaystyle \int_1^\infty \frac{\d t}{t h(t)^d} =\infty.\\
\end{array}\right.$$\end{prop}

%\begin{lemma}\label{upper_bound_tau_n} For any nondecreasing function $h: [0,\infty)\ra(0,\infty)$, 
%\[ \limsup_{t \ra\infty} \frac{R^V(t)}{th(t)} = 0 \quad \mbox{ if }
%\int_1^\infty \frac{\diff t}{t h(t)^d} < \infty \, . \] 
%\end{lemma}

%We will need the following lemma adapted from~\cite[Lemma 1]{Be03}.

%\begin{lemma}\label{berlinkov} Fix $\rho > 0$, $d \geq 1$. Suppose $\varphi\colon[1,\infty)\to(0,\infty)$ is nondecreasing and
%let $h(t)=\phi(t)/t$. Then,
%\[ \int_1^{\infty} \frac{1}{th(t)^d} \,\diff t < \infty \mbox{ if and only if } \sum_{k=1}^\infty h(\rho e^k)^{-d} < \infty \, . \]
%\end{lemma}

%\begin{proof} Note that $\int_1^{\infty} \frac{1}{th(t)^d} \, \diff t < \infty$ if and only if 
%$\int_1^{\infty} \frac{1}{h(\rho e^t)^d} \,\diff t < \infty$.  Let $t \geq 0$, then 
%\[ h(\rho e^{s+t}) = \frac{1}{\rho e^{s+t}} \varphi(\rho e^{s+t}) \geq \frac{1}{\rho e^{s+t}} \varphi(\rho e^{s})
%= e^{-t} h(\rho e^s) \, . \]
%Now, suppose that $\int_1^{\infty} \frac{1}{h(\rho e^t)^d} \,\diff t = \infty$, then 
%\[ \sum_{k=1}^\infty h(\rho e^k)^{-d} \geq e^{-d} \sum_{k=1}^\infty \sup\{ h(\rho e^{x})^{-d} \, : \, k \leq x \leq k+1\} = \infty \, %. \]
%Conversely, if $\sum_{k=1}^\infty h(\rho e^k)^{-d} = \infty$, then
%\[ \int_1^\infty \frac{1}{h(\rho e^t)^d} \, \diff t \geq \sum_{k=1}^\infty \inf\{ h(\rho e^x)^{-d} \, : \, k \leq x \leq k+1 \} 
%\geq e^{-d} \sum_{k=1}^\infty h(\rho e^{k+1})^{-d} = \infty \, , \]
%which completes the proof of the lemma.
%\end{proof}

\begin{proof}[Proof of the first part of Proposition~\ref{limsup_asymptotics}] %[Proof of Theorem~\ref{asymptotics_R_X}, upper bound on limit superior.]
Consider $h \colon(0,\infty)\ra(0,\infty)$ such that
$\int_1^\infty \frac{\diff t}{th(t)^d}<$\vspace{-1mm}
$\infty$ %and define $h(t) = \varphi(t)/t$. This implies $\int_1^\infty \frac{1}{th(t)^d} < \infty$,
which  %by Lemma~\ref{berlinkov} 
is equivalent to $\int_{t>1} h(\frac{1}{3}e^t)^{-d} \diff t < \infty$, so that
\vspace{-2mm} 
\be{summable_criterion} \sum_{n=1}^\infty h(\tfrac{1}{3} e^n)^{-d} < \infty . \ee
It is not hard to see that $h(t)\to\infty$ and that we can assume, without loss of generality, that
$h(t) \le (\log t)^{\gamma}$ for some $\gamma>1$, replacing~$h(t)$
by $\tilde h(t) =h(t) \wedge (\log t)^{\gamma}$ if necessary. 
%Note that this implies $\lim_{t\to\infty}h(t)=\infty$. 
\pagebreak[3]

%Indeed, we can define $\tilde{h}(t) = \min \{ (\log t)^{\frac{1}{2}(1+\gamma)}, h(t) \}$. 
%Then, $t \mapsto t \tilde{h(t)}$ is increasing, since for$s<t$,
%\[ s \tilde{h}(s) \leq s \min \{ (\log s)^{\frac{1}{2}(1+\gamma)}, h(s) \} \leq  \min \{ t(\log t)^{\frac{1}{2}(1+\gamma)}, t h(t) 
%\} = t \tilde{h(t)} \, . \]
%Also, 
%\[ \bal \int_{t > 1} \frac{\diff t}{t\tilde{h}(t)^d} & = \int_{t > 1} \max \Big\{ \frac{1}{t h(t)^d} , \frac{1}{t (\log t)^{\frac{1}{2}(1+\gamma)}} %\Big\} \diff t\\
%& \leq \int_{t > 1}\frac{1}{t h(t)^d} \diff t + \int_{t > 1} \frac{1}{t (\log t)^{\frac{1}{2}(1+\gamma)}} < \infty 
%\, . \eal \]
%Therefore, the assumptions of the theorem apply to $\tilde{h}$ and if we can prove the theorem in this case, we can deduce that since $\tilde{h} \leq %h$, 
%\[ \limsup_{t \ra \infty} \frac{R(t)}{t h(t)} \leq \limsup_{t \ra \infty} \frac{R(t)}{t \tilde{h}(t)} = 0 \, , \]
%which justifies our assumption $\lim_{t\ra\infty} h(t) (\log t)^{-\gamma} = 0$.
Fix $\eps > 0$ and an increasing sequence $t_n \ra \infty$.  It suffices to show that almost surely, 
\[ \limsup_{n \ra \infty} \frac{R^V(t_n)}{t_n h(t_n)} \leq \eps \, . \]

To this end, we now show that for all but finitely many~$n$, 
\begin{equation}\label{claim}
\tfrac{R^V(t_n)}{t_n} > \eps h(t_n) \quad\mbox{ implies }\quad
\tfrac{R^V(t)}{t} > \tfrac{1}{4}\eps h(t_n) \mbox{ for all $t \in [t_n,3t_n]$.}
\end{equation}
%Since, by Lemma~\ref{xi_of_Z_t}, $Z_t$ never returns to a point that it has visited once, we know that if 
By definition, ${R^V(t_n)} > \eps\,{t_n} h(t_n)$ implies that $Z_t$ does not jump during the interval 
$[t_n,t_n(1+\eps h(t_n))]$. As $R^V$ is affine with slope $-1$ on this interval
\[ \frac{R^V(t)}{t} = \frac{R^V(t_n) + t_n - t }{t} >  \frac{ (1+\eps h(t_n)) t_n - t}{t} 
\geq \frac{\eps h(t_n)}{4} \quad \mbox{for $t \in [t_n,t_n\frac{(1+\eps 
h(t_n))}{(1+\frac{1}{4}\eps h(t_n))}]$}\]
Recall that $h(t)\to\infty$, and hence we have,  for all but finitely many~$n$, that
$(1+\eps h(t_n))\geq 3 {(1+\frac{1}{4}\eps h(t_n))}$, completing the proof of~\eqref{claim}.

Now, define $k(n) = \inf \{ k \, : \, e^k \geq t_n\}$, so that in particular $t_n \leq e^{k(n)} < 3 t_n$. 
Then, by~\eqref{claim} and monotonicity of $\phi$, we can deduce that for $n$ large enough
\[ \tfrac{R^V(t_n)}{t_n} \geq \eps h(t_n) \quad\mbox{ implies }\quad \tfrac{R^V(e^{k(n)})}{e^{k(n)}} \geq \tfrac{\eps}{4} h(t_n) \geq \tfrac{\eps}{12} h(\tfrac{1}{3}e^{k(n)}) \, .  \]
This shows in particular that 
\[ \Prob \big\{ \tfrac{R^V(t_n)}{t_n} \geq \eps h(t_n)  \mbox{ infinitely often} \big\} \leq 
\Prob \big\{ \tfrac{R^V(e^n)}{e^n} \geq \tfrac{\eps}{12} h(\tfrac{1}{3}e^n) \mbox{ infinitely often} \big\} \, . \]
By Proposition~\ref{moderate_deviations} we can deduce that exists a constant
$\tilde{C}$ such that for all $n$ large enough 
\[ \Prob\big\{ \tfrac{R^V(e^n)}{e^n} \geq \tfrac{\eps}{12} h(\tfrac{1}{3}e^n) \big\}
\leq  \tilde{C} h(\tfrac{1}{3} e^n)^{-d} \, . \] 
By~(\ref{summable_criterion}) these probabilities are summable, so that Borel-Cantelli 
completes the proof.
\end{proof}

For the second part of Proposition~\ref{limsup_asymptotics}, we need to prove a lower bound on the limit superior, so our strategy is to use the fine control over the decay of correlations that we developed in the previous section and combine it with the Kochen-Stone lemma.

\begin{proof}[Proof of second part of Proposition~\ref{limsup_asymptotics}] 
Let $h \colon (0,\infty)\ra(0,\infty)$ be %by $h(t) = \varphi(t)/t$ 
such that $\int_1^\infty \frac{\diff t}{th(t)^d} = \infty$.
%We can also assume that $h(t) \leq (\log t)^2 $ for all $t\geq 0$. Namely, assume that $h(t_n) > (\log t_n)^2$ for some sequence $t_n\ra\infty$, and %take $\tilde{h}(t) = \min \{ (\log t)^2, h(t) \}$. 
%Then, 
%\[ \int_{t > 1} \frac{\diff t}{\tilde{h}(t)^d} \geq \int_{t>1} \frac{\diff t}{h(t)^d} = \infty, . \]
%Clearly $\tilde{h}(t) \leq (\log)^{2}$ and so if we can prove the theorem under this extra assumption, 
%we can deduce that there exists a sequence $t_n \ra \infty$ such that
%\[ \lim_{n\ra\infty}\frac{R(t_n)}{t_n\tilde{h}(t_n)} = \infty \, . \]
%Now, if there exists a subsequence $t_{n_k}$ such that $h(t_{n_k}) > (\log t_{n_k})^2$, then $\tilde{h(t_{n_k})} = \log (t_{n_k})^2$ and it follows %that\[ \lim_{k\ra\infty} \frac{R(t_{n_k})}{t_{n_k} (\log (t_{n_k}))^2} = \infty \,  ,\]
%which contradicts the first part of the theorem, as $\sum_{k\geq 1} \frac{1}{n^2}$ is summable.
%Hence, we know that eventually for all $n$ large enough, $\tilde{h(t_n)} = h(t_n)$ and we can deduce that
%\[ \lim_{n\ra\infty} \frac{R(t_n)}{t_n h(t_n)} =  \lim_{n\ra\infty} \frac{R(t_n)}{t_n \tilde{h}(t_n)} = \infty \, , \]
%so that the theorem also holds for $h$.
Then, we can deduce that %By Lemma~\ref{berlinkov} we have
\vspace{-2mm} 
\begin{equation}\label{summ}
\sum_{n=1}^\infty h(e^n)^{-d} = \infty.
\end{equation}
Without loss of generality, we can assume that $h(t) \ra\infty$ and also additionally 
that $h(t) \leq (\log t)^{2/d}$ for all $t$. Indeed, if necessary, 
we may replace $h(t)$\vspace{-2pt} by $\tilde{h}(t)= h(t) \wedge (\log t)^{2/d}$. 
For fixed $\kappa>0$, define the event $E_n = \{ \frac{R(e^n)}{e^n} \geq \kappa h(e^n)\}$.
By Proposition~\ref{moderate_deviations}, 
\[ \Prob (E_n) = \tfrac{1}{d \, B(\alpha-d+1,d)} (1+o(1)) \kappa^{-d} h(e^n)^{-d} \, ,  \]
so that by~\eqref{summ}  we have $\sum_{n=1}^\infty \Prob(E_n) = \infty$. 
By the Kochen-Stone lemma, see for instance~\cite{FG97}, we then have that
\be{kochen-stone} \Prob \{ E_n \mbox{ infinitely often } \} \geq \limsup_{k\ra\infty} \frac{\big(\sum_{n=1}^k \Prob(E_n) \big)^2}{\sum_{n=1}^k \sum_{m=1}^k \Prob(E_m \cap E_n)} \, . \ee
Fix $\eps > 0$.  By Proposition~\ref{moderate_deviations} and Remark~\ref{uniform_in_c} we can deduce that we can choose $N$ large enough such that for all $t \geq N$ and all $(\log t)^{\frac{1}{2d}}\wedge h(t) \leq \theta \leq (\log t)^{6}$, we have that
\be{mod_deviations} (1-\eps)\tfrac{1}{d \, B(\alpha-d+1,d)} \,\theta^{-d} \leq  \Prob\{ Z_t = Z_{t+\theta t} \} \leq (1+\eps)\tfrac{1}{d \, B(\alpha-d+1,d)} \, \theta^{-d} \, . \ee
Also, by Lemma~\ref{asymptotic_independence}, we know that we can assume $N$ is large enough such that such that for all $n\geq N$
and $m \geq n + \log(1+\kappa h(e^n))$, we have that
\be{asymptotic_independence_estimate} \bal \Prob\{ Z_{e^n} = Z_{e^n(1+\kappa h(e^n))} & \neq Z_{e^m} = Z_{e^m(1+\kappa h(e^m))} \} \\ & \leq (1+\eps)\big(\tfrac{1}{dB(\alpha-d+1,1)}\big)^2 \kappa^{-2d}h(e^n)^{-d}h(e^m)^{-d} \\
& \leq \tfrac{1+\eps}{1-\eps}\, \Prob(E_n)\Prob(E_m) \, .\eal \ee
Note that by Lemma~\ref{xi_of_Z_t}, we know that $Z_t$ never returns to the same point, therefore we have
\[\bal & \Prob( E_n \cap E_m) \\&= \Prob\big\{ Z_{e^n} = Z_{e^m(1+\kappa h(e^n))}\big\} + 
\Prob\big\{Z_{e^n} = Z_{e^n(1+\kappa h(e^n))} \neq Z_{e^m} = Z_{e^m(1+\kappa h(e^m))} \big\}. \eal\]
In particular, notice that the second probability is zero if $n \leq m \leq n +\log(1+\kappa h(e^n))$. 
Hence,  we can estimate for $n > N$ and for $k$ large enough,
using~(\ref{mod_deviations}) and~(\ref{asymptotic_independence_estimate}),
\[\bal \sum_{m=n}^k & \Prob(E_n \cap E_m) \\
& \leq \sum_{m=n}^{n + 2 \log n} \Prob\{ Z_{e^n} = Z_{e^m(1+\kappa h(e^n))} \} + \sum_{m=n+2 \log n}^k \hspace{-0.5cm} \Prob\{ Z_{e^m n^{-2}} = Z_{e^m(1+\kappa h(e^m))} \} \\ 
& \quad + \tfrac{1+\eps}{1-\eps} \sum_{m = n + \log(1+\kappa h(e^n))}^k \Prob(E_n) \Prob(E_m) \\
& \leq \tilde{C}\,\Prob(E_n)  \sum_{m = n}^k e^{d(n-m)} + \tilde{C}\, n^{-2d} \sum_{m=n}^k \Prob(E_m) + \tfrac{1+\eps}{1-\eps} \sum_{m = n}^k \Prob(E_n) \Prob(E_m) \, ,
\eal\]
where $\tilde{C}$ is some suitable constant. 
Finally, in order to bound the right hand side of~(\ref{kochen-stone}), we can estimate for $k > N$,
\[\bal \sum_{n=1}^k  \sum_{m=1}^k \Prob(E_n \cap E_m)  %\\ 
 &\leq 2N \sum_{n=1}^k \Prob(E_n) + \sum_{n=N}^k\sum_{m=N}^k \Prob(E_n \cap E_m) \\
& \leq 2\sum_{n=1}^k \Big(N + \sum_{m=1}^k m^{-2d} + \sum_{m =n}^k \tilde{C} e^{d(n-m)}\Big)  \Prob(E_n) 
\\ & \hspace{2cm} 
+ 2 \tfrac{1+\eps}{1-\eps} \sum_{n=N}^k \sum_{m=n}^k \Prob(E_n)\Prob(E_m) \\
& \leq C' \sum_{n=1}^k \Prob(E_n) + \tfrac{1+\eps}{1-\eps} \sum_{n=1}^k\sum_{k=1}^k \Prob(E_n) \Prob( E_m) 
\, , \eal \]
where $C'> 0$.
Therefore, we can conclude from~(\ref{kochen-stone}) that 
$\Prob\{ E_n \mbox{ infinitely often } \} \geq \tfrac{1-\eps}{1+\eps}$, and
since $\eps > 0$ and $\kappa>0$ were arbitrary, 
the second statement of Proposition~\ref{limsup_asymptotics} follows.
\end{proof}

\subsection{Almost sure asymptotics for the maximizer of the solution profile}\label{sect_asymptotics_sigma}

In this section, we prove Theorem~\ref{asymptotics_R_X}. Thus, we have to transfer the almost sure ageing result of Proposition~\ref{limsup_asymptotics}, which was formulated on the level of the variational problem, to the residual lifetime function of the maximizer $X_t$ of the profile $v$. The underlying idea is that most of the time $X_t$ and the maximizer of the variational problem $Z_t$ agree and we only have to control the length of the intervals when they can disagree. The latter scenario corresponds to those times during which the processes relocate to another point. Therefore, our strategy is to look at the jump times and show that both processes jump at almost the same times.

%Around each jump of the maximizer $Z^{\ssup 1}$ there is a certain interval during which $u$ is localized in more than one point. Thus our aim is to show that the length of this interval divided by the time of the respective jump of $Z^{\ssup 1}$ tends to zero. 

%Recall that $X_t = \argmax\{ u(t,z) \, : \, z \in \Z^d\}$ and define {\color{red} (Define starting conditions)}
%\[\ba{ccl} b_n^- & = & \inf\{ t > b_{n-1}^+ \, : \, \Phi_t(Z_t^{\ssup 1}) - \Phi_t(Z_t^{\ssup 1}) \geq \frac{1}{2} a_t \la_t \}\, , \\
% b_n^+ & = & \inf\{ t > b_n^- \, : \, \Phi_t(Z_t^{\ssup 1}) - \Phi_t(Z_t^{\ssup 1}) \leq \frac{1}{2} a_t \la_t \}\, . \ea \]
%Furthermore, define $I_n = [b_n^-,b_n^+]$. 

The period when the maximizers relocates correspond exactly to those times when $Z_t^{\ssup 1}$ and $Z_t^{\ssup 2}$ produce a comparable value of $\Phi$. With this in mind, define for $\la_t = (\log t)^{-\beta}$ with $\beta > 1  + \frac{1}{\alpha - d}$, the set of exceptional transition times
\be{definition_of_calE} \calE = \calE(\beta) = \big\{ t > t_0 \, : \, \Phi_t(Z_t^{\ssup 1}) - \Phi_t (Z_t^{\ssup 2}) \leq \tfrac{1}{2} a_t \la_t 
\big\} \, , \ee
where $t_0$ is chosen sufficiently large and, to avoid trivialities, such that $t_0 \neq \inf \calE$. 
By~\cite[Lemma 3.4]{KLMS09} we can choose $t_0$ large enough such that for all $t > t_0$, 
\be{separation_first_third} \Phi_t(Z_t^{\ssup 1}) - \Phi_t(Z_t^{\ssup 3}) > a_t \la_t \, . \ee

\begin{lemma}\label{properties_jumps_of_Z} The process $(Z_t^{\ssup 1} \colon t \geq t_0)$ jumps only at times contained in the set~$\calE$. Moreover, each connected component of $\calE$ contains exactly one such jump time.
\end{lemma}

\begin{proof}
The first part of the statement is trivial, since at each jump time $\tau \geq t_0$ of $Z_t^{\ssup 1}$ we have that $\Phi_\tau(Z_\tau^{\ssup 1}) = \Phi_\tau(Z_\tau^{\ssup 2})$ so that $\tau \in \calE$.
For the second statement, let $[b^-,b^+]$ be a connected component of $\calE$, then 
\[ \Phi_t(Z_t^{\ssup 1}) - \Phi_t(Z_t^{\ssup 2}) = \tfrac{1}{2} a_t \la_t \, , \]
for $t = b^-,b^+$ (here we use that $b^- \geq \inf \calE \neq t_0$). Now, since $t \mapsto \Phi_t(Z_t^{\ssup 1}) - \Phi_t(Z_t^{\ssup 2})$ is never constant, if $Z_{b^-}^{\ssup 1} = Z_{b^+}^{\ssup 1}$ then by Lemma~\ref{separation_first_second} there is $t \in (b^-,b^+)$ such that $t \notin \calE$ contradicting the connectedness of $[b^-,b^+]$. Thus, we can conclude that $Z_t^{\ssup 1}$ jumps at least once in $[b^-,b^+]$. Finally, the fact that, by Lemma~\ref{xi_of_Z_t}, $Z_t^{\ssup 1}$ never returns to the same point combined with~(\ref{separation_first_third}) guarantees that $Z_t^{\ssup 1}$ only jumps once in $[b^-,b^+]$ (namely from $Z_{b^-}^{\ssup 1}$ to $Z_{b^-}^{\ssup 2}$).
\end{proof}

Denote by $(\tau_n)$ the jump times of the maximizer process $(Z^{\ssup 1}_t \colon t \geq t_0)$ in increasing order. In the
next lemma we have collected some of their basic properties.

\begin{lemma}\label{relative_lower_bound_tau_n} 
\begin{itemize} 
\item[(i)] Fix $\beta > 1 + \frac{1}{\alpha -d}$, then, almost surely, for all but finitely many $n$, 
\[ (\xi(Z_{\tau_n}^{\ssup 1}) - \xi(Z_{\tau_n}^{\ssup 2})) 
\, \big(\tfrac{\tau_{n+1} - \tau_n}{\tau_n}\big) \geq a_{\tau_n} (\log \tau_n)^{-\beta} \, . \]
\item[(ii)] Fix $\gamma > 1 + \frac{2}{\alpha - d}$, then, almost surely, for all but finitely many $n$,
\[ \frac{\tau_{n+1} - \tau_n}{\tau_n} \geq (\log \tau_n)^{-\gamma} \, . \]
\item[(iii)]\label{lower_bound_jump_xi} Fix $\delta > 1 + \frac{1}{\alpha-d} + \frac{1}{d}$, then, almost surely, for all but finitely many $n$, 
%(1 + \frac{1}{\alpha-d} + \frac{1}{d}$
\[ \xi(Z_{\tau_n}^{\ssup 1}) - \xi(Z_{\tau_n}^{\ssup 2}) \geq a_{\tau_n} (\log \tau_n)^{-\delta} \, . \]
\end{itemize}
\end{lemma}

%\begin{corollary}\label{lower_bound_jump_xi} For any $\delta > (1 + \frac{1}{\alpha-d})(1+\frac{2}{d})$ \comment{adjust parameters. Where is this needed?}
%\[ \xi(Z_{\tau_n}^{\ssup 1}) - \xi(Z_{\tau_n}^{\ssup 2}) \geq a_{\tau_n} (\log \tau_n)^{-\delta} \, . \]
%\end{corollary}
%\begin{proof}This follows from Lemma~\ref{relative_lower_bound_tau_n}(i) combined with Lemma~\ref{upper_bound_tau_n}.
%\end{proof}

\begin{proof}
(i) By Lemma~\ref{difference_phi} and Remark~\ref{Z_of_jump_time} we find that
\be{diff_Phi_tau_n} \begin{aligned} \Phi_{\tau_{n+1}} (Z_{\tau_n}^{\ssup 1}) - \Phi_{\tau_{n+1}} (Z_{\tau_n}^{\ssup 2}) & =  (\xi(Z_{\tau_n}^{\ssup 1}) - \xi(Z_{\tau_n}^{\ssup 2})) \big(\tfrac{\tau_{n+1} - \tau_n}{\tau_{n+1}}\big) \\
& \leq (\xi(Z_{\tau_n}^{\ssup 1}) - \xi(Z_{\tau_n}^{\ssup 2})) \big(\tfrac{\tau_{n+1} - \tau_n}{\tau_{n}}\big) \, . \end{aligned} \ee
Now, we can estimate the difference on the left-hand side from below by using that $Z_{\tau_n}^{\ssup 2}$ cannot produce more than the third largest value of $\Phi$ at time $\tau_{n+1}$. 
Indeed, Lemma~\ref{xi_of_Z_t} ensures that $Z^{\ssup 1}$ never visits the same point again, so that
$Z_{\tau_n}^{\ssup 2} = Z_{\tau_{n-1}}^{\ssup 1} \neq Z_{\tau_{n+1}}^{\ssup i}$ for $i = 1,2$ since $Z_{\tau_{n+1}}^{\ssup 2} = Z_{\tau_n}^{\ssup 1}$. Hence, using~\cite[Proposition 3.4]{KLMS09} for the second inequality, 
\[ \bal \Phi_{\tau_{n+1}} (Z_{\tau_n}^{\ssup 2}) & \leq \Phi_{\tau_{n+1}}(Z_{\tau_{n+1}}^{\ssup 3})
\leq \Phi_{\tau_{n+1}}(Z_{\tau_{n+1}}^{\ssup 1}) - a_{\tau_{n+1}} (\log \tau_{n+1})^{-\beta} 
\\ & \leq  \Phi_{\tau_{n+1}}(Z_{\tau_n}^{\ssup 1})- a_{\tau_n} (\log \tau_n)^{-\beta} \, , \eal \]
where in the last step we again used that $Z_{\tau_n}^{\ssup 1} = Z_{\tau_{n+1}}^{\ssup 2}$ and that
$t \mapsto a_t (\log t)^{-\beta}$ is increasing for all sufficiently large $t$.
Substituting this inequality into~(\ref{diff_Phi_tau_n}) completes the proof of part~(i).

(ii) By the first part, we need to get an upper bound on $\xi(Z_{\tau_n}^{\ssup 1})$. Therefore, our first claim is that for any $\delta > \frac{1}{\alpha - d}$, and all $t$ sufficiently large
\be{upper_bound_on_xi_of_Z} \xi (Z_t^{\ssup 1}) \leq a_t (\log t)^{\delta} \, . \ee
Indeed, by~\cite[Lemma 3.5]{HMS08}, for $\eps = \frac{1}{3}( \delta - \frac{1}{\alpha - d})$ we have for all sufficiently large~$r$, 
\[ \max_{|z| \leq r} \xi(z) \leq r^{\frac{d}{\alpha}} (\log r)^{\frac{1}{\alpha} + \eps} \, . \]
Moreover, by~\cite[Lemma 3.2]{KLMS09}, for $\eps ' = \frac{\alpha}{d}(  \frac{1}{3}( \delta - \frac{1}{\alpha - d}))$, we find that
$|Z_t^{\ssup 1}| < r_t (\log t)^{\frac{1}{\alpha - d} + \eps'}$ and therefore 
%using that $r_t^{\frac{d}{\alpha}} = a_t$ and that $\log r_t = (q+1)\log t(1+ o(1))$, 
\[ \begin{aligned} \xi(Z_t^{\ssup 1}) & \leq a_t (\log t)^{(\frac{1}{\alpha - d} + \eps')\frac{d}{\alpha}}  ( \log( r_t (\log t)^{\frac{1}{\alpha - d} + \eps'} ))^{\frac{1}{\alpha}+\eps} \\
& = a_t (q+1)^{\frac{1}{\alpha} + \eps} (\log t)^{ \frac{1}{\alpha - d} + \eps' \frac{d}{\alpha} + \eps } (1+o(1))
< a_t (\log t)^{\delta} \, , \end{aligned}\]
eventually for all $t$ sufficiently large.
\pagebreak[3]

Now, if we combine part (i) for $\beta = \frac{1}{2}(\gamma + 1) > 1 + \frac{1}{\alpha - d}$ 
with~(\ref{upper_bound_on_xi_of_Z}) for $\delta=\frac{1}{2}(\gamma - 1)$, we get
\[ \frac{\tau_{n+1} - \tau_n}{\tau_n} \geq \frac{a_{\tau_n}(\log \tau_n)^{-\beta}}{ \xi(Z_{\tau_n}^{\ssup 1})  }
\geq \frac{a_{\tau_n}(\log \tau_n)^{-\beta}}{a_{\tau_n} (\log \tau_n)^{\delta}} = (\log \tau_n)^{- (\delta + \beta)}
= (\log \tau_n)^{- \gamma} \, , \]
which completes the proof of the lemma.

(iii) Note that for any $\delta' > \frac{1}{d}$, Proposition~\ref{limsup_asymptotics}, shows that for all but finitely many $n$, 
\[ \frac{\tau_{n+1}-\tau_n}{\tau_n} = \frac{R^V(\tau_n)}{\tau_n}
\leq (\log \tau_n)^{\delta'} \, . \]
This observation together with part (i), immediately implies the statement of part (iii).
\end{proof}

A similar statement to Lemma~\ref{properties_jumps_of_Z} also holds for the process $X_t = \argmax\{ u(t,z) \, : \, z \in \Z^d\}$.  Fix $0<\eps <\frac{1}{3}$, then by~\cite[Proposition 5.3]{KLMS09} we can assume additionally that $t_0$ in the definition~(\ref{definition_of_calE}) of $\calE$ is chosen large enough such that for all $t > t_0$
\be{one_point} \Big[ U(t)^{-1} \sum_{\heap{z \in \Z^d}{z\not=Z_t^{(1)}}} u(t,z) \Big] \1\{ \Phi_t(Z_t^{\ssup 1}) - \Phi_t(Z_t^{\ssup 2}) \geq \tfrac{1}{2} a_t \la_t \} < \eps \, .\vspace{-3mm} \ee
Furthermore, by the `two cities theorem'~\cite[Theorem 1.1]{KLMS09}, we may assume that for all $t \geq t_0$, 
\be{two_point} \frac{u(t,Z_t^{\ssup 1}) + u(t,Z_t^{\ssup 2})}{U(t)} > 1 - \eps \, . \ee

\begin{lemma}\label{properties_jumps_of_X} The process $(X_t \colon t \geq t_0)$ only jumps at times contained in $\calE$ and each connected component of $\calE$ contains exactly one such jump time. Furthermore, it never returns to the same point in $\Z^d$.
\end{lemma}

\begin{proof} 
By~(\ref{one_point}), for any $t \in [t_0,\infty)\setminus \calE$, we have 
$X_t = Z_t$ so that, in particular, $X_t$ jumps only at times in $\calE$.
Now, let $[b^-,b^+]$ be a connected component of $\calE$. Note that the proof of Lemma~\ref{properties_jumps_of_Z} shows that for all $t \in [b^-,b^+]$, the set $\{ Z_t^{\ssup 1}, Z_t^{\ssup 2} \}$ consists of exactly two points, $z^{\ssup 1}:= Z_{b^+}^{\ssup 1}$ and $z^{\ssup 2}:=Z_{b^+}^{\ssup 2}= Z_{b^-}^{\ssup 1}$. Hence, by~(\ref{one_point}) we find that $X_{b^-} = z^{\ssup 2}$ and $X_{b^+} = z^{\ssup 1}$. Also, the two-point localization~(\ref{two_point}) implies that 
$\{ X_t \, : \, t \in [b^-,b^+]\} = \{ z^{\ssup 1},z^{\ssup 2} \}$. Hence, it remains to show that $(X_t \colon t>0)$ 
jumps only once (from $z^{\ssup 2}$ to $z^{\ssup 1}$) in the interval $[b^-,b^+]$.

Define the function
\[ g(t) = \frac{u(t,z^{\ssup 1})}{u(t,z^{\ssup 2})} \, . \]
Then, note that since $u$ solves the heat equation, for $z\in\{z^{\ssup 1}, z^{\ssup 2}\}$,
\[ \tfrac{\partial}{\partial t} u (t,z) = \Delta u(t,z) + \xi(z) \,u(t,z) = \sum_{y \sim z} (u(t,y) - u(t,z)) 
+ \xi(z) \,u(t,z) \, . \]
Furthermore, by~\cite[Lemmas 2.2, 3.2]{KLMS09}, we have $z^{\ssup 1} \not\sim z^{\ssup 2}$ so that using~(\ref{two_point}) we get
\[\begin{aligned} (-2d + \xi(z))\,u(t,z)  < \tfrac{\partial}{\partial t} u (t,z) &< 2d \eps U(t) - 2d u(t,z) + \xi(z)\, u(t,z) 
\\ & < 2d \tfrac{\eps}{1-\eps} (u(t,z^{\ssup 1}) + u(t,z^{\ssup 2})) + (\xi(z) - 2d)\, u(t,z) \, . \end{aligned} \]
Therefore, %if we calculate the derivative of $g$ we obtain
\[ \begin{aligned} g'(t) & = \frac{\tfrac{\partial}{\partial t} u(t,z^{\ssup 1}) u(t,z^{\ssup 2}) - u(t,z^{\ssup 1})\tfrac{\partial}{\partial t} u(t,z^{\ssup 2}) }{u(t,z^{\ssup 2})^2}
\\
& > \frac{1}{u(t,z^{\ssup 2})^2} \Big[ \big(\xi(z^{\ssup 1}) - \xi(z^{\ssup 2}) - 2d \tfrac{\eps}{1-\eps}\big) u(t,z^{\ssup 1}) u(t,z^{\ssup 2})  - 2d \tfrac{\eps}{1-\eps} u(t,z^{\ssup 1})^2 \Big] \\
& = g(t)  \Big(\xi(z^{\ssup 1}) - \xi(z^{\ssup 2}) - 2d \tfrac{\eps}{1-\eps}( 1+  g(t)) \Big). \end{aligned} \]
Now, since $z^{\ssup 1} = Z_{b^+}^{\ssup 1}$ and $z^{\ssup 2} = Z_{b^-}^{\ssup 1}$, Lemma~\ref{lower_bound_jump_xi} shows (again assuming that $t_0$ is large enough) that,
for any $\delta > 1+\frac{1}{\alpha-d}+\frac{1}{d}$, if $\tau$ is the jump time of $Z^{\ssup 1}$ in the interval $[b^-,b^+]$, then
\[ \xi(z^{\ssup 1}) - \xi(z^{\ssup 2}) \geq a_{\tau} (\log \tau)^{-\delta} \, . \]
Hence, we can deduce that if there exists $t'$ such that $g(t') = 1$, then $g'(t') > 0$. Using the continuity of $u$ 
we see that first there can be at most one such $t'$ and $g(t) < 1$ if $t < t'$ and $g(t) > 1$ if $t > t'$, and second
that there exists $t' \in [b^-,b^+]$ such that $g(t') = 1$. Therefore it has to be unique and $u(t,z^{\ssup 1}) < u(t,z^{\ssup 2})$ if $t < t'$ and $u(t,z^{\ssup 1}) > u(t,z^{\ssup 2})$ if $t > t'$. Thus, we can see that $X_t$ jumps exactly once in the interval $[b^-,b^+]$.
%To complete the proof, we just have to notice that by Lemma~\ref{xi_of_Z_t}, we know that $t \mapsto \xi(Z_t^{\ssup 1})$ is increasing on %$(t_0,\infty)$. Therefore since $Z^{\ssup 1}$ and $X^{\ssup 1}$ agree on $[t_0,\infty) \setminus \calE$ and since in any connected component %$[b^-,b^+]$ of $\calE$, 
%$X_t$ jumps once from $Z_{b^-}^{\ssup 1}$ to $Z_{b^+}^{\ssup 1}$, we conclude that $t \mapsto \xi(X_t)$ is increasing, so that %once $X_t$ has jumped, it never returns to a point that it has visited previously.
\end{proof}

In order to be able to deduce the asymptotics of the jump times of $(X_t \colon t>0)$ from those of $(Z_t\colon t>0)$, 
we find bounds for the length of a connected component of $\calE$.

%To formalize this idea, recall that we know that by~\cite[Proposition 5.3]{KLMS09} one-point localization occurs when $\phi(Z^{\ssup 1}_t) - \phi(Z^{\ssup 2}_t) \geq a_t \la_t / 2$. Thus, define 
%\be{defn_of_b_n} b_n = \inf\{ t > 0 \, : \, \Phi_{\tau_n + t} (Z_{\tau_n + t}^{\ssup 1}) - \Phi_{\tau_n + t} (Z_{\tau_n + t}^{\ssup 2}) = \tfrac{1}{2} \, a_{\tau_n + t} \, \la_{\tau_n + t}  \} \, , \ee
%where $\la_t = (\log t)^{-\beta}$ for some $\beta >(\frac{2}{d}(q+1) +1)(1+\frac{1}{\alpha -d })$.

\begin{lemma}\label{length_intermediate_interval} 
Suppose in the definition~(\ref{definition_of_calE}) we choose $\beta > 1+ \frac{q+2}{d} + \frac{1}{\alpha-d}$. %(\frac{2}{d}(q+2) +1)(1+\frac{1}{\alpha -d })$.
Then, for any $0 < \eps < \frac{1}{2} (\beta - (1+ \frac{q+2}{d} + \frac{1}{\alpha-d}))$, almost surely 
for any connected component $[b^-,b^+]$ of $\calE$ with $b^-$ large enough, we find that
\[ \frac{b^+ - b^-}{\tau} \leq (\log \tau)^{-\eps}\, , \]
where $\tau$ is the jump time of the process $(Z_t \colon t>0)$ in the interval $[b^-,b^+]$.
\end{lemma}

\begin{proof} We start by expressing the distances $b^+ - \tau$ and $\tau - b^-$ in terms of the potential values 
at the sites $Z_\tau^{\ssup 1}$ and $Z_\tau^{\ssup 2}$. 
As we have seen in the proof of Lemma~\ref{properties_jumps_of_Z}, 
$Z_{\tau}^{\ssup i} = Z_{b^+}^{\ssup i}$ for $i =1,2$. Hence, we obtain that
\[ \Phi_{b^+} (Z_{\tau}^{\ssup 1}) - \Phi_{b^+} (Z_{\tau}^{\ssup 2}) =
\Phi_{b^+} (Z_{b^+}^{\ssup 1}) - \Phi_{b^+} (Z_{b^+}^{\ssup 2})=  \tfrac{1}{2} a_{b^+} \la_{b^+} \, . \]
Moreover, by Lemma~\ref{difference_phi} we get that
\[ \Phi_{b^+} (Z_{\tau}^{\ssup 1}) - \Phi_{b^+} (Z_{\tau}^{\ssup 2})
= (\xi(Z_{\tau}^{\ssup 1}) - \xi(Z_{\tau}^{\ssup 2})) ( 1 - \tfrac{\tau}{b^+}) \, . \]
Combining the previous two displayed equations and rearranging yields
\be{difference_in_xi_b^+}   b^+ - \tau =  \frac{\frac{1}{2} b^+ a_{b^+} \la_{b^+}}{\xi (Z_{\tau}^{\ssup 1}) - \xi (Z_{\tau}^{\ssup 2})}   \, . \ee
Similarly, we know that $Z_{b^-}^{\ssup 1} = Z_\tau^{\ssup 2}$ and $Z_{b^-}^{\ssup 2} = Z_\tau^{\ssup 1}$ 
and deduce in the same way that
\be{difference_in_xi_b^-}  \tau - b^-   =  \frac{\frac{1}{2} b^- a_{b^-} \la_{b^-}}{\xi (Z_{\tau}^{\ssup 1}) - \xi (Z_{\tau}^{\ssup 2})}   \, . \ee
Define $\tau^+$ as the next jump of $Z_t^{\ssup 1}$ after $\tau$, then $b^+ \leq \tau^+$. We use~(\ref{difference_in_xi_b^+}) and~(\ref{difference_in_xi_b^-}) to get
\be{length_interval} \begin{aligned}
\frac{b^+ - b^-}{\tau} & = \frac{b^+ - \tau}{\tau} + \frac{\tau - b^-}{\tau}
= \frac{1}{2} \frac{1}{\xi (Z_{\tau}^{\ssup 1}) - \xi (Z_{\tau}^{\ssup 2})} \Big(   a_{b^+} \la_{b^+} \frac{b^+}{\tau}
+  a_{b^-} \la_{b^-} \frac{b^-}{\tau} \Big) \\
& \leq \frac{1}{2} \frac{1}{\xi (Z_{\tau}^{\ssup 1}) - \xi (Z_{\tau}^{\ssup 2})} \Big(   a_{b^+} \la_{b^+} \frac{\tau^+}{\tau}
+  a_{\tau} \la_{\tau}  \Big) \, , \end{aligned}
\ee
where we used in the last step that $\beta^- \leq \tau$ and that $t \mapsto a_t (\log t)^{-\beta} = \frac{t^q}{(\log t)^{q + \beta}}$ is increasing for all $t$ large enough.
Next, by the definition of $a_t$ and $\la_t$, we obtain that
\be{a_tau_n_b_n}  a_{b^+} \la_{b^+}  = \tfrac{ (b^+ )^q}{(\log b^+)^{q + \beta}}
\leq \tfrac{\tau^q}{(\log \tau)^{q+\beta}} \, \big( \tfrac{\ \tau^+}{\tau}\big)^q 
= a_{\tau} \la_\tau \big( \tfrac{\ \tau^+}{\tau}\big)^q   \, , \ee
where we used that $b^+ \leq \tau^+$ for the inequality. Using  Lemma~\ref{relative_lower_bound_tau_n}(i),
if $\tau$ is large enough, for $\beta'= 1 + \frac{1}{\alpha - d} + \frac{\eps}{2}$, we get
\[ \xi(Z_{\tau}^{\ssup 1}) - \xi(Z_{\tau}^{\ssup 2}) \geq  \frac{\tau}{ \tau^+ - \tau} \, a_{\tau} (\log \tau)^{-\beta'} \, . \]
Hence, substituting this estimate into~(\ref{length_interval}) together with the previous estimate~(\ref{a_tau_n_b_n}) yields
\[ \begin{aligned} \frac{b^+ - b^-}{\tau} & \leq 
%\frac{1}{2} \frac{1}{\xi (Z_{\tau}^{\ssup 1}) - \xi (Z_{\tau}^{\ssup 2})} \Big(   a_{b^+} \la_{b^+} \frac{\tau^+}{\tau}
%+  a_{\tau} \la_{\tau}  \Big) \\& \leq  
\frac{\tau^+ - \tau}{\tau} \,(\log \tau)^{\beta' - \beta} \Big( \big( \tfrac{\ \tau^+}{\tau}\big)^{q+1} + 1\Big)
\leq 2 \big( \tfrac{\ \tau^+}{\tau}\big)^{q+2} (\log \tau)^{\beta' - \beta}\, . \end{aligned} \]
It remains to bound the term $\tau^+/ \tau$. 
By Proposition~\ref{limsup_asymptotics}, for $\delta = \frac{1}{d} + \frac{\eps}{2(q+2)}$, we get
\[ \frac{\tau_+}{\tau} = 1 + \frac{\tau^+- \tau}{\tau} \leq (\log \tau)^{\delta} \, . \]
Finally, we have shown that if $b^-$ is large enough
$\frac{b^+ - b^-}{\tau} %\leq 2 \Big( \frac{\ \tau^+}{\tau}\Big)^{q+2} (\log \tau)^{\beta' - \beta}
\leq 2 (\log \tau)^{\beta' - \beta + (q+2) \delta} < (\log \tau)^{-\eps}$,
%where we used that by our choice of $\eps$
%\[ \beta - \beta' - \delta(q+2) = \beta - \Big(\frac{2}{d}(q+2) + 1\Big)\Big(1 + \frac{1}{\alpha - d}\Big) - \eps > \eps \,, \]
which completes the proof.
\end{proof}

We are now in the position to translate the results from Section~\ref{asymptotics_tau} from the setting of the variational problem to 
the setting of the residual lifetime function of the maximizer of the solution.

\begin{proof}[Proof of Theorem~\ref{asymptotics_R_X}]
Suppose $t \mapsto h(t)$ is a nondecreasing function such that 
$$\int_1^\infty \frac{\diff t}{t h(t)^d} < \infty.$$
Without loss of generality, we can assume that there exists $\gamma' > 0$ such that $h(t) \leq (\log t)^{\gamma'}$ for all $t>0$.
Also, let $\gamma > 1 + \frac{2}{\alpha - d}$. %Fix $\gamma > 1 + \frac{2}{\alpha - d}$ and $\gamma' > \frac{2(1+\alpha -d)}{d(\alpha - d)}$. 
Fix $\eps > 0$ and choose $\beta > 1 + \frac{q+2}{d} + \frac{1}{\alpha-d}$ %(\frac{2}{d}(q+2) +1)(1+\frac{1}{\alpha -d })$ 
large enough such that 
\[ \delta := \tfrac{1}{4} \big(\beta - \big(1 + \tfrac{q+2}{d} + \tfrac{1}{\alpha-d}\big)\big) >  \gamma' +\gamma\,. \]
Define $\calE = \calE(\beta)$ as in~(\ref{definition_of_calE}) and denote by $[b_n^-,b_n^+], n \geq 1$, the connected components of $\calE$.
By Lemmas~\ref{properties_jumps_of_Z} and~\ref{properties_jumps_of_X} each of the processes $(X_t \colon t \geq t_0)$ and $(Z_t \colon t \geq t_0)$ jumps only at times in $\calE$ and each interval $[b_n^-,b_n^+]$ contains exactly one jump time, which we denote by $\sigma_n$ for $X_t$ and $\tau_n$ for $Z_t$. 
%Since we are only interested in the asymptotics of the sequences of $(\sigma_n)_{n \geq 1}$ and $(\tau_n)_{n \geq 1}$, there is no loss of generality %by coupling the indices in this way.
By Lemma~\ref{relative_lower_bound_tau_n} and Proposition~\ref{limsup_asymptotics}, for all $n$ sufficiently large,
\be{bounds_tau_n} 2(\log \tau_n)^{-\gamma} \leq \frac{\tau_{n+1} - \tau_n}{\tau_n} \leq 
\tfrac{\eps}{3} h(\tfrac{1}{2} \tau_n) \leq \tfrac{1}{2}(\log \tau_n)^{\gamma'} \, . \ee
We now want to translate the upper bound to the jump times $(\sigma_n)$. For this purpose, we can invoke Lemma~\ref{length_intermediate_interval} to find that by our choice of $\beta$ and $\delta$ we have that for all $n$ sufficiently large
\be{difference_b_n} \frac{b_n^+ - b_n^-}{\tau_n} \leq (\log \tau_n)^{-\delta} \, .\ee
Now, we first use that $|\sigma_n - \tau_n| \leq b_n^+ - b_n^-$ and then the estimates~(\ref{bounds_tau_n}) and~(\ref{difference_b_n}) to obtain
\[ \begin{aligned} \frac{R(\sigma_n)}{\sigma_n h(\sigma_n)} = \frac{\sigma_{n+1} - \sigma_n}{\sigma_n h(\sigma_n)}
%\leq \frac{ \tau_{n+1}- \tau_n + b_{n+1}^+ - b_{n+1}^- + b_n^+ - b_n^- }{\tau_n(1-(\log \tau_n)^{-\delta}) h(\tau_n (1-(\log \tau_n)^{-\delta}))} \\& 
& \leq \Big( \frac{\tau_{n+1} - \tau_n}{\tau_n} + \frac{b_{n+1}^+ - b_{n+1}^-}{\tau_{n+1}} \frac{\tau_{n+1}}{\tau_n} + \frac{b_n^+ - b_n^-}{\tau_n}\Big) \\ & \hspace{1cm}\times\big((1-(\log \tau_n)^{-\delta}) h(\tau_n (1-(\log \tau_n)^{-\delta}))\big)^{-1}% \Big(1 - \frac{b_n^+ - b_n^-}{\tau_n} \Big)^{-1} 
\\
%& \leq  \Big( \frac{\tau_{n+1} - \tau_n}{\tau_n} + (\log \tau_{n+1})^{-\delta} \Big(1 + \frac{\tau_{n+1}- \tau_n}{\tau_n}\Big) + (\log \tau_n)^{-\delta} \Big) \Big(1 - (\log \tau_n)^{-\delta} \Big)^{-1} \\
& \leq  \big( \tfrac{\tau_{n+1} - \tau_n}{\tau_n} +  (\log \tau_{n+1})^{-\delta + \gamma'}  + (\log \tau_n)^{-\delta} \big) \,
\big( \tfrac{1}{2}h(\tfrac{1}{2}\tau_n )\big)^{-1} 
%\Big(1 - (\log \tau_n)^{-\delta} \Big)^{-1} 
\\
& \leq 3 \,\tfrac{\tau_{n+1} - \tau_n}{ h(\frac{1}{2} \tau_n) \tau_n} \leq \eps %h(\tfrac{1}{2} \tau_n) \leq \eps h(\frac{1}{2}\sigma_n (1-(\log \tau_n)^{-\delta})^{-1}) \leq \eps h(\sigma_n) 
\, , \end{aligned}\]
for all but finitely many $n$. %where we used in the last step that $\delta > \gamma'$ so that $(\log \tau_{n+1})^{-\delta + \gamma'} \ra 0$ as $n \ra \infty$. 
In particular, this shows that, almost surely,
\[ \limsup_{n \ra \infty} \frac{R(\sigma_n)}{\sigma_n h(\sigma_n)} = 0 \, . \]
However, since $R$ jumps only at the points $\sigma_n$ and decreases on $[\sigma_n,\sigma_{n+1})$, this immediately implies the first part of Theorem~\ref{asymptotics_R_X}, see also Figure~\ref{remaining_lifetime}.

For the \emph{second part} of the proof, suppose $t \mapsto h(t)$ is a nondecreasing function such that
\[ \int_1^\infty \frac{\diff t}{t h(t)^d} = \infty \, . \]
Fix $\kappa > 0$, then by Proposition~\ref{limsup_asymptotics}, we know that there exists a sequence $(t_n)_{n\ge 1}$ such that 
$R^V(t_n)\geq 3\kappa t_n h(2 t_n) $.
Define a subsequence of the jump times $(\tau_n)$ by choosing $n_k$ such that for some index $j$ we have that $t_j \in [\tau_{n_k},\tau_{n_k+1})$.
%those $\tau_n$ w
%setting 
%\[  n_k = \sup \{ \ell \, : \, t_k \geq \tau_\ell \} \, . \]
%By passing to a subsequence if necessary, we can assume that $(\sigma_{n_k})$ are a strictly increasing subsequence of jump times. 
In particular, since %$t_j \in [ \tau_{n_k}, \tau_{n_k +1})$, and 
$R^V$ is decreasing on the interval $[\tau_{n_k},\tau_{n_k+1})$, we can deduce that for $k$ large enough
\[\frac{\tau_{n_k +1}-\tau_{n_k}}{\tau_{n_k} h(2\tau_{n_k})} =  \frac{R^V(\tau_{n_k})}{\tau_{n_k} h(2\tau_{n_k})} \geq \frac{R^V(t_j)}{t_j h(2t_j)}  \geq 3\kappa \, , \]
Similarly as for the upper bound, we can estimate
\[ \begin{aligned} \frac{R(\sigma_{n_k})}{\sigma_{n_k}h(\sigma_{n_k})} 
& = \frac{\sigma_{{n_k}+1} - \sigma_{n_k}}{\sigma_{n_k}h(\sigma_{n_k})}
 \geq \frac{ \tau_{{n_k}+1}- \tau_{n_k} -(b_{{n_k}+1}^+ - b_{{n_k}+1}^-) - (b_{n_k}^+ - b_{n_k}^-) }{(\tau_{n_k} + ( b_{n_k}^+ - b_{n_k}^-))h(\tau_{n_k} + b_{n_k}^+ - b_{n_k}^-)} \\
& \geq  \Big(1 - \frac{b_{{n_k}+1}^+ - b_{{n_k}+1}^-}{\tau_{{n_k}+1}}\frac{\tau_{{n_k}+1}}{\tau_{n_k}}\frac{\tau_{n_k}}{\tau_{{n_k}+1}-\tau_{n_k}} - \frac{b_{n_k}^+ - b_{n_k}^-}{\tau_{n_k}}\frac{\tau_{n_k}}{\tau_{{n_k}+1}-\tau_{n_k}}\Big) \\
& \hspace{0.5cm} \times \frac{\tau_{{n_k}+1} - \tau_{n_k}}{\tau_{n_k}}\Big((1 + (\log\tau_{n_k})^{-\delta}) h(\tau_{n_k}(1 + (\log\tau_{n_k})^{-\delta}))\Big)^{-1} \\
%& \geq \frac{\tau_{{n_k}+1} - \tau_{n_k}}{\tau_{n_k}} \Big( 1 - (\log \tau_{{n_k}+1})^{- \delta} \Big( 1 + \frac{\tau_{{n_k}+1} - \tau_{n_k}}{\tau_{n_k}}\Big)
%(\log \tau_{n_k})^{\gamma} - (\log \tau_{n_k})^{- \delta + \gamma} \Big) ( 1 + (\log \tau_{n_k})^{-\delta})^{-1} \\
& \geq \frac{\tau_{{n_k}+1} - \tau_{n_k}}{\tau_{n_k}} ( 1 - (\log \tau_{{n_k}+1})^{\gamma + \gamma' - \delta} - (\log \tau_{n_k})^{\gamma - \delta}) (2h(2\tau_{n_k}))^{-1}\\
%( 1 +  (\log \tau_{n_k})^{-\delta})^{-1} 
& \geq \tfrac{1}{3}\frac{\tau_{{n_k}+1} - \tau_{n_k}}{\tau_{n_k}h(2\tau_{n_k})} \geq \kappa %h(2 \tau_{n_k}) \geq \kappa h( 2 \sigma_{n_k} ( 1 + (\log \tau_{n_k})^{-\delta})^{-1}) \geq \kappa h(\sigma_{n_k}) 
\, , \end{aligned}\]
eventually for all $k$ large enough. This implies that
$\displaystyle\limsup_{t \ra \infty} \tfrac{ R(t) }{t h(t)} = \limsup_{t \ra \infty} \tfrac{ R(t) }{\varphi(t)} \geq \kappa$,
thus\vspace{-1mm} completing the proof of Theorem~\ref{asymptotics_R_X}.
\end{proof}

\section{A functional scaling limit theorem}\label{sect_spatial_limit_theorem}

The aim of this section is to prove Theorem~\ref{spatial_limit_u}. As in previous sections, we start by dealing with 
the maximizer of the variational problem formulating a limit theorem for the process
\be{proc}
 \Big( \big( \tfrac{Z_{tT}}{r_T} , \tfrac{\Phi_{tT}(Z_{tT})}{a_T} \big) \, : \, t > 0 \Big). 
\ee 
Convergence will take place in the Polish space $D(0,\infty) := D((0,\infty),\R^{d+1})$ of all c\`adl\`ag processes defined on $(0,\infty)$ taking values in $\R^{d+1}$ equipped with the Skorokhod topology on compact subintervals. This means that $f_n\to f$ if, for every $0<a<b<\infty$
we can find a continuous and strictly increasing time-changes $\la_n\colon [a,b]\ra[a,b]$ such that
$$\sup_{t \in [a,b]} |\la_n(t) - t| \to 0 \quad \mbox{ and } \sup_{t \in [a,b]} |f(t) - f_n(\la_n(t))| \to 0 , $$
for more details see~\cite{Bil99}. The main part of this section is devoted to the proof of the following proposition 
stated in terms of the maximizer of the variational problem.

\begin{prop}\label{spatial_limit} %$\phantom{q}$
As $T\to\infty$
\[ \Big( \big(\tfrac{Z_{tT}}{r_T}, \tfrac{\Phi_{tT}(Z_{tT})}{a_T}\big)\, : \, t>0\Big) \weakconv  \Big(\big(Y^{\ssup 1}_t,Y^{\ssup 2}_t+q\big(1-\tfrac{1}{t}\big)|Y^{\ssup 1}_t|\big)\, : \, t>0 \Big) \, , \]
in the sense of weak convergence on $D(0,\infty)$.
\end{prop}

We will prove this result by first showing convergence of the finite-dimensional distributions in Section~\ref{sect_finite_diml} 
and then tightness in Section~\ref{sect_tightness}.  In Section~\ref{transfer_spatial_limit}, we transfer the results to the maximizer 
of the profile and the potential value at that site, hence showing Theorem~\ref{spatial_limit_u} and, by a slight variation, also
Proposition~\ref{classical}.

\subsection{Finite-dimensional distributions}\label{sect_finite_diml}

The next lemma shows that the finite-dimensional distributions of the process~\eqref{proc}
converge weakly to those of the limiting process defined in terms of $Y = (Y^{\ssup 1},Y^{\ssup 2})$. 
%Note since for any fixed $t$ the set $\{ Y(t) \neq Y(t-) \}$ has $\Prob$-measure $0$,
%the set $\calT_Y$, whose complement contains the fixed discontinuities of $Y$, 
%is equal to the interval $(0,\infty)$. Thus, we need to show convergence of the finite-dimensional distribution for any $0< t_1 < t_2 < \ldots < t_k < \infty$.
%Here, we use that for fixed $t$ the set $\{ Y(t) \neq Y(t-) \}$ has $\Prob$-measure $0$, so that the projection $\pi_t : D([\eps,M],\widehat H) \ra \widehat H$ is continuous except on a set of $\Prob$-measure $0$.

\begin{lemma}\label{finite_dimensional_distributions} Fix $0< t_1<\ldots<t_k <\infty$. Then as $T \ra \infty$,
\[ \bal \Big(\big(\tfrac{Z_{t_1T}}{r_T}, & \tfrac{\Phi_{t_1T}(Z_{t_1T})}{a_T}\big), \ldots,\big(\tfrac{Z_{t_kT}}{r_T}, \tfrac{\Phi_{t_kT}(Z_{t_kT})}{a_T}\big)\Big) \\ & \weakconv 
\big((Y^{\ssup 1}_{t_1},Y^{\ssup 2}_{t_1}+q(1-\tfrac{1}{t_1})|Y^{\ssup 1}_{t_1}|),\ldots, (Y^{\ssup 1}_{t_k},Y^{\ssup 2}_{t_k}+q(1-\tfrac{1}{t_k})|Y^{\ssup 1}_{t_k}|)\big) \, .\eal \]\end{lemma}

\begin{proof} First notice, 
by the continuous mapping theorem, see e.g.~\cite[Theorem~2.7]{Bil99}, we can equivalently show that for $Y_t = (Y_t^{\ssup 1},Y_t^{\ssup 2})$
\[ \bal \Big(\big(\tfrac{Z_{t_1T}}{r_T}, & \tfrac{\Phi_{t_1T}(Z_{t_2T})}{a_T} - q\big(1-\tfrac{1}{t_1}\big)\tfrac{|Z_{t_1T}|}{r_T}\big), \ldots,\big(\tfrac{Z_{t_kT}}{r_T}, \tfrac{\Phi_{t_kT}(Z_{t_kT})}{a_T} - q\big(1-\tfrac{1}{t_k}\big)\tfrac{|Z_{t_kT}|}{r_T}\big)\Big) \\ & \weakconv \big(Y_{t_1},\ldots, Y_{t_k}\big) \, .\eal\]
%((Y^{\ssup 1}_{t_1},Y^{\ssup 2}_{t_1}),\ldots, (Y^{\ssup 1}_{t_k},Y^{\ssup 2}_{t_k}+q(1-\frac{1}{t_k})|Y^{\ssup 1}_{t_k}|)) \, .\eal \]
Define
\[ \restrH = \big\{ (x,y) \in \R^d\times \R \, : \, y  > -q\big(1 - \tfrac{1}{t_k}\big) |x| \big\}\, . \]
and recall that, for large $T$, all components in the vectors above are in~$H^*$. Hence it suffices to show that, 
for any $A \subset (\restrH)^k$ with $\mathrm{Leb}_{k(d+1)}(\partial A) = 0$, 
we have, as $T \ra \infty$,
\be{fd} \Prob\Big\{ \big(\tfrac{Z_{t_iT}}{r_T},\tfrac{\Phi_{t_iT}(Z_{t_iT})}{a_T} - q\big(1-\tfrac{1}{t_i}\big)\tfrac{|Z_{t_iT}|}{r_T}\big)_{i=1}^k \in  A  \Big\} \ra \Prob\big\{ (Y_{t_i})_{i=1}^k \in A \big\} \, . \ee
The remainder of the proof is organised as follows: First, we show that in fact it suffices to show~(\ref{fd}) for $A$ intersected with large boxes. Second, we also show that it is enough to consider the maximizer of the variational problem on a large region. These steps let us express the probability in question in terms of the point process
$\Pi_T = \{ (z/r_t,\Phi_T(z)/a_T) \, : \, z \in \Z^d\}$ restricted to a relatively compact set, so that we can invoke the weak convergence of $\Pi_T \weakconv \Pi$ and recognize the resulting event in terms of the process $(Y_t \colon t>0)$. 

\emph{Step 1.} Define a large region
\[ B_N = \{ (x,y) \in \restrH \, : \, |x|\leq N,\tfrac{1}{N} - q|x| \leq y \leq N \}\, . \]
We claim that we only have to show that
\be{fd_restricted_to_box} \Prob\Big\{ \big(\tfrac{Z_{t_iT}}{r_T},\tfrac{\Phi_{t_iT}(Z_{t_iT})}{a_T} - q(1-\tfrac{1}{t_i})\tfrac{|Z_{t_iT}|}{r_T}\big)_{i=1}^k \in B_N^k \cap A \Big\}
\ra  \Prob\big\{ (Y_{t_i})_{i=1}^k \in B_N^k \cap A  \big\} \, ,  \ee
for all $N$ in order to deduce~(\ref{fd}). Indeed, 
%we can argue similarly as in Lemma~\ref{restriction_to_box},
%\[\begin{aligned} \Prob\Big\{ & \Big(\tfrac{Z_{t_iT}}{r_T},\tfrac{\Phi_{t_iT}(Z_{t_iT})}{a_T} - %q(1-\tfrac{1}{t_i})\tfrac{|Z_{t_iT}|}{r_T}\Big)_{i=1}^k \in  A \Big\}
%\\ & = \Prob\Big\{ \Big(\tfrac{Z_{t_iT}}{r_T},\tfrac{\Phi_{t_iT}(Z_{t_iT})}{a_T} - q(1-\tfrac{1}{t_i})\tfrac{|Z_{t_iT}|}{r_T}\Big)_{i=1}^k \in B_N^k %\cap A \Big\} \\
%& \quad + \Prob\Big\{ %\mbox{for some } i \, : \,
%\Big(\tfrac{Z_{t_iT}}{r_T},\tfrac{\Phi_{t_iT}(Z_{t_iT})}{a_T} - q(1-\tfrac{1}{t_i})\tfrac{|Z_{t_iT}|}{r_T}\Big)_{i=1}^k \in A \setminus B_N^k \Big\} %\, . \end{aligned} \]
%But, 
using that, by~\cite[Lemma 3.2]{KLMS09}, %and Lemma~\ref{xi_of_Z_t} respectively that 
$\Phi_t(Z_t)$ is an increasing function of $t$, for all $t$ large enough, % so that we can assume that these properties hold for all $t \geq \eps T$. 
we get
\begin{align} &\Prob \Big\{ 
\big(\tfrac{Z_{t_iT}}{r_T},\tfrac{\Phi_{t_iT}(Z_{t_iT})}{a_T} - q\big(1-\tfrac{1}{t_i}\big)\tfrac{|Z_{t_iT}|}{r_T}\big) \notin  B_N 
 \mbox{ for some } i \Big\} \notag \\
& \leq \sum_{i=1}^k \Prob\big\{ \tfrac{|Z_{t_iT}|}{r_T} > N \big\} + \Prob\big\{ \tfrac{\Phi_{t_iT}(Z_{t_iT})}{a_T} < \tfrac{1}{N} \big\} + \Prob\big\{ \tfrac{\Phi_{t_iT}(Z_{t_iT})}{a_T} + \tfrac{q}{t}\tfrac{|Z_{t_iT}|}{r_T}> N \big\} \notag\\
& \leq k \Big[\max_{i=1,\ldots,k}\Prob\big\{ \tfrac{|Z_{t_iT}|}{r_{t_iT}} > N \tfrac{r_T}{r_{t_kT}} \big\} + \Prob\big\{ \tfrac{\Phi_{t_1 T}(Z_{t_1 T})}{a_{t_1 T}} < \tfrac{1}{N}\tfrac{a_T}{a_{t_1 T}} \big\} \label{bound_outside_box}\\
& \hspace{1cm} + \Prob\big\{ \tfrac{\Phi_{t_kT}(Z_{t_kT})}{a_{t_kT}} > \tfrac{1}{2}N\tfrac{a_T}{a_{t_kT}} \big\}+ \max_{i=1,\ldots,k} 	\Prob\big\{ \tfrac{|Z_{t_iT}|}{a_{t_i T}} > \tfrac{N t_1}{2 q}\tfrac{r_T}{r_{t_kT}}\big\}\ \Big] \notag \\
& \leq C_1 \, k \, \Big[ (\tfrac{N}{t_k^{q+1}})^{d-\alpha} + e^{-C_2 (N t_1^q)^{\alpha-d}} + (\tfrac{N}{2t_k^q})^{d-\alpha} + 	(\tfrac{N t_1}{2qt_k^{q+1}})^{d-\alpha} \Big] \notag
 \,  , 
\end{align}
where $C_1,C_2>0$ are some constants and in the last step we used Lemma~\ref{restricting_maximizer} and 
the fact that ${a_T}/{a_{t_1 T}} \ra t_1^{-q}$ and ${r_T}/{r_{t_kT}} \ra t_k^{-(q+1)}$. 
%and also  $(Z_t/r_t,\Phi_t(Z_t)/a_t)$ converges to a random variable with density on $\R^d \times \R^+$ (see e.g.~\cite[Lemma 6.2]{KLMS09}), t
Hence, the terms in the last display tend to zero %uniformly in $T$, 
as $N \ra \infty$. Similarly, we can bound 
\begin{align} & \Prob\big\{ (Y_{t_i})_{i=1}^k \in A \setminus B_N^k \big\}  
\leq \sum_{i=1}^k  \Prob \big\{ Y_{t_i} \not\in  B_N \big\}\notag\\[-2mm] 
& \leq \sum_{i=1}^k \big[ \Prob \big\{ |Y_{t_i}^{\ssup 1}| > N \big\} + 
\Prob \big\{ - q(1-\tfrac{1}{t_k}) |Y_{t_i}^{\ssup 1}| \leq  Y_{t_i}^{\ssup 2} \leq \tfrac{1}{N} - q|Y_{t_i}^{\ssup 1}| \big\}  + \Prob\big\{Y_{t_i}^{\ssup 2} > N \big\} \big]\notag\\
& \leq k \big[ \Prob\big\{ |Y_{t_k}^{\ssup 1}| > N \big\} + \Prob \big\{ |Y_{t_1}^{\ssup 1}| \leq \tfrac{t_k}{qN} \big\} 
+ \Prob \big\{ Y_1^{\ssup 2} > N \big\} \big] \notag\, , \end{align}
where we used that $|Y_t^{\ssup 1}|$ is an increasing function in $t$ and $Y_1^{\ssup 2}\geq Y_t^{\ssup 2}$ for all $t$ by construction.
%, since $Y_1$ corresponds to the point of the Poisson point process $\Pi$ with the largest second component. 
By definition of $(Y_t \colon t>0)$ all the probabilities tend to zero, as $N \ra \infty$, and hence
if we can show~(\ref{fd_restricted_to_box}) we can also deduce~(\ref{fd}).

\emph{Step 2.} Denote, for $K > N$ by $Z_{tT}^{_{K,T}}$ the point
%with $(Z_{tT}^{_{K,T}}/r_T,Z_{tT}^{_{K,T}}/a_T)\in B_K$
satisfying
\[ \Phi_{tT}(Z_{tT}^{_{K,T}}) = \max\big\{ \Phi_{tT}(z) \, : \,  t\xi(z) \geq z \mbox{ and }\big(\tfrac{z}{r_T}, \tfrac{\Phi_T(z)}{a_T}\big) 
\in B_K \big\} \, , \]
where in case of a tie we take the one with the larger $\ell^1$ norm.
We claim that if $K$ is large, $Z_{tT}^{_{K,T}}$ agrees with high probability with the global maximizer $Z_{tT}$. Indeed, we find that
\begin{align} 
%& \Big| \Prob \Big\{  \big(\tfrac{Z_{t_iT}}{r_T},\tfrac{\Phi_{t_iT}(Z_{t_iT})}{a_T} - q(1-\tfrac{1}{t_i})\tfrac{|Z_{t_iT}|}{r_T}\big)_{i=1}^k \in %B_N^k \cap A \Big\} \notag\\
%& \hspace{2cm} - \Prob\Big\{ \Big(\tfrac{Z^{K,T}_{t_iT}}{r_T},\tfrac{\Phi_{t_iT}(Z_{t_iT}^{K,T})}{a_T} - %q(1-\tfrac{1}{t_i})\tfrac{|Z_{t_iT}^{K,T}|}{r_T}\Big)_{i=1}^k \in B_N^k \cap A \Big\} \Big| \notag\\ 
%& \leq 
\Prob\big\{ &\mbox{there exists } i \mbox{ with } Z_{t_iT}^{_{K,T}} \neq Z_{t_iT} \big\} 
\leq \sum_{i=1}^k  \Prob \big\{ \big( \tfrac{Z_{t_i T}}{r_T}, \tfrac{\Phi_{T}(Z_{t_iT})}{a_T} \big) \notin B_K \big\} \label{fd_restricted_maximizer}\\
& \leq k \max_{i=1,\ldots,k} \Big[ \Prob\big\{ \tfrac{|Z_{t_i T}|}{r_T} \geq K \big\} +  \Prob\big\{ \tfrac{\Phi_T(Z_T)}{r_T}>K\big\}
+  \Prob\big\{ \tfrac{\Phi_{t_iT}(Z_{t_iT})}{a_T} < \tfrac{1}{K} \big\} \Big] \notag
\, , \end{align}
where for the last term, we use that by Lemma~\ref{Phi_t(1+c)_in_Phi_t}, we can express
\[ \tfrac{\Phi_{t_iT}(Z_{t_i T})}{a_T} = \tfrac{\Phi_T(Z_{t_iT})}{a_T} + q(1-\tfrac{1}{t_i})\tfrac{|Z_{t_iT}|}{r_T} + \mbox{error}(T) \, , \]
where the error term tends to $0$. Hence, as in~(\ref{bound_outside_box}), we can use Lemma~\ref{restricting_maximizer} to
show that the expression~(\ref{fd_restricted_maximizer}) tends to zero if we first let $T$ and then $K \ra \infty$.

\emph{Step 3.} Using the point process we want to express the probability
\be{fd_double_restricted} \begin{aligned} \Prob\Big\{ & \big(\tfrac{Z^{{K,T}}_{t_iT}}{r_T},\tfrac{\Phi_{t_iT}(Z_{t_iT}^{K,T})}{a_T} - q(1-\tfrac{1}{t_i})\tfrac{|Z_{t_iT}^{K,T}|}{r_T}\big)_{i=1}^k \in B_N^k \cap A \Big\} \\
& = \int_{A\cap B_N^k} \Prob\Big\{ \tfrac{Z^{K,T}_{t_iT}}{r_T} \in \diff x_i,\tfrac{\Phi_{t_iT}(Z_{t_iT}^{K,T})}{a_T} - q(1-\tfrac{1}{t_i})\tfrac{|Z_{t_iT}^{K,T}|}{r_T} \in \diff y_i \mbox{ for all } i\Big\} \, , \end{aligned}
\ee
in the limit as $T\ra \infty$. 
First note that by Lemma~\ref{Phi_t(1+c)_in_Phi_t} we have that, for any $t \in [t_1,t_k]$,
\[ \tfrac{\Phi_{tT}(z)}{a_T} = \tfrac{\Phi_T(z)}{a_T} + q(1-\tfrac{1}{t}) \tfrac{|z|}{r_T} + \delta_{1-t}\big(T,\tfrac{z}{r_T},\tfrac{\Phi_T(z)}{a_T}\big) \, ,  \]
where the error $\delta_{1-t}$ goes to $0$ uniformly for all $z$ such that $(\frac{z}{r_T},\frac{\Phi_T(z)}{a_T}) \in B_K$ and also uniformly for all $t \in [t_1,t_k]$. 
Recall also that $\Pi_T$ converges weakly to $\Pi$ on $H^*$. 
Now, as the restriction to large boxes ensures that we are only dealing with the point process on relatively compact sets, we can in the limit as $T\ra\infty$ express the condition $$\tfrac{Z^{K,T}_{t_iT}}{r_T} = %\diff 
x_i,\tfrac{\Phi_{t_iT}(Z_{t_iT}^{K,T})}{a_T} - q(1-\tfrac{1}{t_i})\tfrac{|Z_{t_iT}^{K,T}|}{r_T} = y_i$$ 
by requiring that $\Pi$ has an atom in $(x_i,y_i)$ and all other points $(x,y)$ of $\Pi$ restricted to $B_K$ satisfy
$y + q(1-\tfrac{1}{t_i})|x| \leq y_i + q(1 - \tfrac{1}{t_i})|x_i|$.
Therefore, if we denote by $\calC_{t_i}(x_i,y_i)$ the open cone of all points $(x,y) \in \restrH$ satisfying
$y + q(1-\tfrac{1}{t_i})|x| > y_i + q(1-\tfrac{1}{t_i})|x_i|$,
we can express the probability in~(\ref{fd_double_restricted}) in the limit as
\[ \begin{aligned} &\lim_{T \ra \infty}  \Prob\Big\{  \big(\tfrac{Z^{K,T}_{t_iT}}{r_T},\tfrac{\Phi_{t_iT}(Z_{t_iT}^{K,T})}{a_T} - q(1-\tfrac{1}{t_i})\tfrac{|Z_{t_iT}^{K,T}|}{r_T}\big)_{i=1}^k \in B_N^k \cap A \Big\} \\
& = \int_{A\cap B_N^k} \Prob\Big\{ \left.\Pi\right|_{B_K} (\diff x_i \, \diff y_i) = 1 \mbox{ for } i=1,\ldots,k , \left.\Pi\right|_{B_K}\Big(\bigcup_{i=1}^k \calC_{t_i}(x_i,y_i)\Big) = 0 \Big\} \, . \end{aligned} \]

Now, we can remove the restriction of the point process to $B_K$, by letting $K \ra \infty$ and noting that the probability that for some $(x_i,y_i) \in A \cap B_N^k$ and some $i = 1,\ldots, k$ the point process $\Pi$ has a point in the set $\calC_{t_i}(x_i,y_i) \cap B_K^c$ can be bounded from above by the probability that $\Pi$ has a point in the set
\[ \big\{ (x,y) \in \R^{d+1}\, : \, y > \tfrac{1}{N} - q(1-\tfrac{1}{t_k})|x| \mbox{ and } ( y > K \mbox{ or } |x| > K ) \big\} \, . \]
But the intensity measure $\nu$ of $\Pi$ gives finite mass to this region, so that we can conclude that the probability of the latter event tends to zero as $K \ra \infty$. Hence, we can combine this observation with the estimate in~(\ref{fd_restricted_maximizer}) and letting first $T\ra \infty$ and then $K\ra\infty$, to deduce that 
\[ \begin{aligned} \lim_{T\ra \infty} \Prob\Big\{ & \big(\tfrac{Z_{t_iT}}{r_T},\tfrac{\Phi_{t_iT}(Z_{t_iT})}{a_T} - q(1-\tfrac{1}{t_i})\tfrac{|Z_{t_iT}|}{r_T}\big)_{i=1}^k \in B_N^k \cap A \Big\} \\
& = \int_{A\cap B_N^k} \Prob\Big\{ \Pi (\diff x_i \, \diff y_i) = 1 \mbox{ for } i=1,\ldots,k , \Pi \Big(\bigcup_{i=1}^k \calC_{t_i}(x_i,y_i)\Big) = 0 \Big\} \\
& = \Prob\{ (Y_{t_i})_{i=1}^k \in B_N^k \cap A  \} \, , \end{aligned}\]
where in the last step we used the definition of $Y$. For an illustration of the event under the integral, see also Figure~\ref{fig_finite_diml}. Thus we have completed the proof of the lemma.
\end{proof}

\begin{figure}[thbp]
\centering 
\includegraphics[width=6cm]{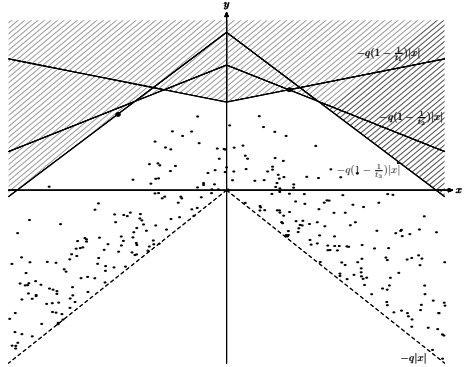} 
\caption[Calculation of finite-dimensional distributions.]{Calculation of finite-dimensional distributions at times $t_1 < 1 < t_2 < t_3$. The event that $Y_{t_i} = (x_i,y_i)$ translates to the condition that the point process $\Pi$ has an atom in each of the points $(x_i,y_i)$, but does not contain any points in the union of open cones with `slope' $-q(1-\frac{1}{t_i})$ whose boundaries touch the points $(x_i,y_i)$ (as indicated by the shaded region).
}\label{fig_finite_diml}
\end{figure}

\subsection{Tightness}\label{sect_tightness}

To prepare the tightness argument we prove two auxiliary lemmas. For fixed $0<a<b$ the first lemma gives us control on the probability that the maximizer  makes small jumps during the time interval $[aT,bT]$. The second shows that, with arbitrarily high probability, during this
time the maximizer stays within a box with  sidelength a multiple of~$r_T$. 

\begin{lemma}\label{small_jumps} Let $\tau_i$ denote the jump times of the process $(Z_t \colon t \geq a T)$ in increasing order. Then
\[ \liminf_{T\ra \infty} \Prob \big\{  \tau_{i+1}-\tau_i \geq \delta T \mbox{ for all jump times } \tau_i \in [a T, b T]  \big\} \geq p(\delta) \, , \]
where $p(\delta) \ra 1$ as $\delta \da 0$.
\end{lemma}

\begin{proof} 
Cover the interval $[a T, bT]$ by small subintervals of length $\delta T$ by setting $x_i = a T + i \delta T$ for
$i = 0,\ldots, N+1$, for $N = \left\lceil (b-a)/\delta \right\rceil$. We estimate 
\[ \bal \Prob & \big\{  \tau_{i+1}-\tau_i < \delta T \mbox{ for some jump times } \tau_i \in [a T, b T]  \big\}\\
& \leq \sum_{j=0}^{N-1} \Prob \big\{  \tau_{i+1}-\tau_i < \delta T \mbox{ for some jump times } \tau_i \in [x_j, x_{j+1}]  \big\} \\ 
& \leq \sum_{j=0}^{N-1} \Prob  \big\{ Z_t \mbox{ jumps more than once in the interval } [x_j,x_{j+2}] \big\} \, .
\eal \]
Hence, taking the limit $T\ra\infty$, we have that 
\[ \limsup_{T \ra \infty} \Prob  
\big\{ \tau_{i+1} - \tau_i < \delta T \mbox{ for some jump time } \tau_i \in [a T, b T]  \big\} \leq N  \,\tilde{p}(2\delta) \, , \]
where $$\tilde{p}(\delta) := \limsup_{T \ra \infty} \Prob \big\{ Z_t \mbox{ jumps more than once in the interval } [T,(1+\delta)T] \big\}.$$
Thus it remains to show that $\tilde{p}(\delta)/\delta \to 0$ as $\delta\ra 0$.
We use notation and ideas from Section~\ref{se:weak_ageing}, which tell us in particular that,
%we recall that $\Pi$ is a point process on $\widehat H$ with intensity 
%\[ \nu(\diff x\, \diff y) =  \frac{\alpha \,\diff x \,\diff y}{(y + q|x|)^{\alpha+1}} \, . \]
%Then, 
as $T\ra \infty$, if we fix $({Z_T}/{r_T},{\Phi_T(Z_T)}/{a_T}) =(x,y)$ then the probability that $(Z_t \, : \, t \geq T)$ jumps more than once in the interval $[T,(1+\delta)T]$ is bounded from above by the probability that the point process $\Pi$ has no points in the set $D_0(|x|,y)$ and at least two points in the set $D_\delta(|x|,y)\setminus D_0(|x|,y)$. To make this bound rigorous, one has restrict the process $(Z_t/r_t,\Phi_t(Z_t)/a_t)$ to large boxes, let $T \ra \infty$ and then the size of the boxes go to infinity and finally justify interchanging the limit. Details are very similar to Lemma~\ref{restriction_to_box} and Lemma~\ref{finite_dimensional_distributions} and are therefore omitted. Using this observation, we obtain the bound 
\begin{align}  \limsup_{T \ra \infty} &  \, \Prob \big\{ Z_t  \mbox{ jumps more than once in the interval } [T,(1+\delta)T] \big\} \notag\\
& \leq \int_{y \geq 0}\int_{x\in \R^d} \Prob \big\{ \Pi(\diff x \, \diff y) = 1, \, \Pi(D_0(|x|,y)) = 0, \,
\Pi(D_\delta(|x|,y)\setminus D_0(|x|,y)) \geq 2 \big\} \notag\\
%& = \int_{(x,y)} e^{-\nu(D_0(|x|,y))} \sum_{k\geq 2} \frac{f_\delta(x,y)^k}{k!} e^{-f_\delta(x,y)} \nu(\diff x \, \diff y)\, , 
& =  \int_{y \geq 0}\int_{x\in \R^d} e^{-\nu(D_0(|x|,y))} 
\big(1 - e^{-f_\delta(|x|,y)} - f_\delta(|x|,y)e^{-f_\delta(|x|,y)}\big) \, \nu(\diff x \, \diff y)\notag\\
& = \tfrac{2^d}{(d-1)!}\int_{0}^\infty\int_{0}^\infty  e^{-\nu(D_0(r,y))}  
\big(1 - e^{-f_\delta(r,y)} - f_\delta(r,y)e^{-f_\delta(r,y)}\big) \, \frac{\alpha r^{d-1} }{(y+qr)^{\alpha+1}}\, \diff r \diff y \, ,  
\label{integral}\end{align}
where $f_\delta(r,y) = \nu(D_\delta(r,y)) - \nu(D_0(r,y))$. 
It remains to be shown that the right hand side divided by $\delta$ converges to zero. 
As we would like to invoke the dominated convergence theorem, we show that this term is bounded by an integrable function.  
%Using that $1-e^{-x}\leq x$, for $x \geq 0$, 
We have
$$
\begin{aligned}  e^{-\nu(D_0(r,y))} & \tfrac{1}{\delta}\,\big(1 - e^{-f_\delta(r,y)} - f_\delta(r,y)e^{-f_\delta(r,y)}\big) 
\, \tfrac{\alpha r^{d-1} }{(y+qr)^{\alpha+1}} \notag\\
& \leq e^{-\nu(D_0(r,y))}  \tfrac{1}{\delta}\,\big(\nu(D_\delta(r,y)) - \nu(D_0(r,y))\big) \tfrac{\alpha r^{d-1}}{(y+qr)^{\alpha+1}}
 \, . \end{aligned}$$
Recall from~(\ref{nu_D_c}) that $\nu(D_\delta(r,y)) = \vartheta y^{d-\alpha} \phi_\delta(v)^{-1}$
with $y+qr = \frac{y}{v}$ and $\phi_\delta$ given by~(\ref{weight}), and $\vartheta= {2^d B(\alpha-d,d)}/{q^d(d-1)!}$.
Next, we estimate the part of the integrand that depends on~$\delta$, so for $\tilde{B}(x) := \tilde{B}(x,\alpha-d,d)$ we consider
\[\bal \tfrac{1}{\delta}(\phi_\delta(v)^{-1} - \phi_0(v)^{-1})  & = \tfrac{1}{\delta}\big( \tfrac{(1+\delta)^\alpha}{(\delta+v)^{\alpha-d}} v^{\alpha-d} \tB\big(\tfrac{v+\delta}{1+\delta}\big) - \tB(v) \big) \\
&  \leq \tfrac{1}{\delta}\big( (1+\delta)^\alpha - 1 \big) + \tfrac{1}{\delta} \big(\tB \big(\tfrac{v+\delta}{1+\delta}\big) - \tB(v) \big) \, . \eal \]
As the first term is $\leq 2 \alpha$ for all $\delta \leq \delta_0$ for some small $\delta_0$ (independent of $r$, $y$), we can concentrate on the second term. Now, we can use the definition of $\tB$ to write
\[ \bal \tfrac{1}{\delta} \big(\tB \big(\tfrac{v+\delta}{1+\delta}\big) - \tB(v) \big)
& = \tfrac{1}{\delta} \int_v^{\frac{v+\delta}{1+\delta}} u^{\alpha-d-1} (1-u)^{d-1} \diff u \leq \tfrac{1}{\delta} \int_v^{\frac{v+\delta}{1+\delta}} u^{\alpha-d-1}  \diff u  \\
& \leq \tfrac{1}{\delta}\big(\tfrac{v+\delta}{1+\delta}-v\big) \max\{ v^{\alpha-d-1}, 1\}  \leq \max\{v^{\alpha-d-1},1\} \, .\eal\] 
Combining the last three displays we obtain a majorant for the integrand in~\eqref{integral} divided by~$\delta$, which does not depend on~$0<\delta<\delta_0$. To show that this majorant is integrable we calculate
$$\begin{aligned}
\int_{0}^\infty\int_{0}^\infty & e^{-\nu(D_0(r,y))}  y^{d-\alpha}
\,\big( 2\alpha + \max\big\{(\tfrac{y}{y+qr})^{\alpha-d-1},1\big\}\big) \, \tfrac{\alpha \vartheta r^{d-1} }{(y+qr)^{\alpha+1}}\, \diff r \diff y \\
& \leq \tfrac{\alpha \vartheta}{q^d} \int_{0}^\infty \int_{0}^{1}  v^{\alpha-d}(1-v)^{d-1} y^{2(d-\alpha)-1} 3\alpha \max\{ v^{\alpha -d -1},1\}  e^{-\vartheta y^{d-\alpha}} \,  \diff v \,\diff y \\
& = \tfrac{3 \,\alpha^2 \vartheta}{q^d}  \int_0^1  v^{\alpha-d}(1-v)^{d-1}  \max\{ v^{\alpha -d -1},1\} \diff v \int_0^\infty  y^{2(d-\alpha)-1}  e^{-\vartheta y^{d-\alpha}} \diff y \\
%& \leq \frac{3\, 2^d \alpha^2}{q^d(d-1)!(\alpha-d)} \max\{B(\alpha-d+1,d),B(2(\alpha-d),d)\} \int_0^\infty w e^{-w}\diff w \\
& \leq \tfrac{3 \, \alpha^2}{q^d(\alpha-d)} \max\{B(\alpha-d+1,d),B(2(\alpha-d),d)\},
\end{aligned}$$
so that the proof is completed by applying the dominated convergence theorem.
\end{proof}

%The previous lemma gives us some control about the maximum number of jumps that $Z_t$ can make during an interval $[\eps T, M T]$.
%We will use this to show that the probability that the $\ell^1$-norm of the rescaled version $Z_t$ can never be too large.

\begin{lemma}\label{max_Z} For fixed $0<a<b$, we have that 
\[ \lim_{\kappa \ra \infty} \limsup_{T\ra\infty} \Prob\Big\{ \sup_{t \in [a T, b T]} \frac{|Z_t|}{r_T} \geq \kappa\Big\} = 0 \, . \]
\end{lemma}

\begin{proof} Fix a jump time $\tau$ of $Z_t$. By Lemma~\ref{xi_of_Z_t} we have $\xi(Z_\tau^{\ssup 1}) > \xi(Z_\tau^{\ssup 2})$. 
In particular, we have, using that $\chi(z) = x - \rho \log x$ is increasing on $x>\rho$, 
\[ \Phi_\tau(Z_\tau^{\ssup 1}) \geq \xi(Z_\tau^{\ssup 1}) - \tfrac{1}{\tau} |Z_\tau^{\ssup 1}|\log \xi(Z_\tau^{\ssup 1}) >
\xi(Z_\tau^{\ssup 2}) - \tfrac{1}{\tau} |Z_\tau^{\ssup 1}| \log \xi(Z_\tau^{\ssup 2}) \, . \]
Since $\Phi_\tau(Z_\tau^{\ssup 1}) = \Phi_\tau(Z_\tau^{\ssup 2})$, we thus obtain that
\[ \xi(Z_\tau^{\ssup 2}) - \tfrac{1}{\tau}|Z_\tau^{\ssup 2}| \log \xi(Z_\tau^{\ssup 2}) + \tfrac{1}{\tau}\eta(Z_\tau^{\ssup 2}) > \xi(Z_\tau^{\ssup 2}) - \tfrac{1}{\tau} |Z_\tau^{\ssup 1}| \log \xi(Z_\tau^{\ssup 2}) \, . \]
Hence using that $\eta(z) \leq |z|\log d$, we find that
\[ |Z_\tau^{\ssup 2}| < |Z_\tau^{\ssup 1}| \big( 1 - \tfrac{\log d}{\log \xi(Z_\tau^{\ssup 2})}\big)^{-1} < |Z_\tau^{\ssup 1}| \big(1 - \tfrac{\log d}{q \log \tau (1+o(1))}\big)^{-1} \, , \]
where we invoked~\cite[Lemma 3.2]{KLMS09} to deduce that eventually $\xi(Z_t^{\ssup 2}) > a_t (\log t)^{-1}$.
Hence, denoting by $N_T$ the number of jumps of $Z_t$ in the interval $[a T, b T]$, we have that for $T$ large enough
\be{max_in_random_jumps} \sup_{t \in [a T, b T]} |Z_t| \leq (1-\tfrac{2\log d}{q \log a T})^{-N_T} |Z_{bT}| \, . \ee
Fix $\eps > 0$. By Lemma~\ref{small_jumps}, we can choose $\delta > 0$ such that, 
\[ \liminf_{T \ra\infty} \Prob \big\{ \tau_{i+1} - \tau_i \geq \delta T
\mbox{ for all jump times } \tau_i \in [a T, bT] \big\} \geq 1 -\tfrac{\eps}{4} \, . \]
%where $\tau_i$ denote the jump times of $Z_t$ in $[a T, bT]$ in increasing order.
If all jump times $\tau_i$ in $[a T, b T]$ satisfy $\tau_{i+1} - \tau_i \geq \delta T$, 
then $N_T \leq \frac{b-a}{\delta}+1$ and hence
\[ \sup_{t \in [a T, b T]} |Z_t| \leq (1-\tfrac{2\log d}{q \log a T})^{-\frac{b-a}{\delta}-1} |Z_{bT}| \, . \]
Therefore, for any $\kappa > 1$, we can estimate that
\[\begin{aligned} \Prob\Big\{ & \sup_{t \in [a T, bT]} \tfrac{ |Z_t|}{r_T} \geq \kappa \Big\} \\
& \leq \Prob\Big\{ (1-\tfrac{2\log d}{q \log a T})^{-\frac{b-a}{\delta}-1} \tfrac{|Z_{bT}|}{r_T} \geq \kappa \Big\}
+ \Prob\big \{ \tau_{i+1}-\tau_i < \delta T \mbox{ for some } 1\le i\le N_T \big\} \\
& \leq \Prob\Big\{ \tfrac{|Z_{bT}|}{r_{bT}} \geq \kappa b^{-(q+1)}(1+o(1)) \Big\} + \tfrac{\eps}{2} \\
& \leq \big(1+\tfrac{\eps}{2}\big) \, \Prob \big\{ |Y^{\ssup 1}_1| \geq \kappa b^{-(q+1)} \big\} + \tfrac{ \eps}{2} 
\, , \end{aligned} \]
for all all $t$ sufficiently large, where we use that $Z_t/r_t \weakconv Y^{\ssup 1}_1$. Hence, by choosing $\kappa$ large enough,
the latter expression can be made smaller than $\eps$, which completes the proof.
\end{proof}

To prove tightness we use the following characterization (see,e.g.,~\cite[Thm. 13.2]{Bil99}). A family
$(P_T \colon T\ge 1)$ of probability measures on $D([a,b])$ is tight if and only if the following two conditions are satisfied:
\be{tightness} \ba{ll} \mbox{(i)} & \lim\limits_{\kappa \ra \infty} \limsup\limits_{T\ra \infty} P_T\{ x \,:\, \|x\| \geq \kappa \} = 0, %\mbox{ and } 
\\ \mbox{(ii)} & \mbox{for any }\epsilon > 0 \mbox{ we have } \lim\limits_{\delta \ra 0} \limsup\limits_{T \ra \infty} P_T\{ x \, : \, w'_x(\delta) \geq \epsilon \} = 0 \, . \ea \ee
Here, $\|x\|$ is the uniform norm, \ie $\|x\| = \sup_{t \in [a,b]} |x(t)|$,
and the modulus $w'_x(\delta)$ is defined as
\[ w'_x(\delta) = \inf_{\{t_i\}} \max_{1 \leq i \leq v} w_x[t_{i-1},t_i) \, , \]
where the infimum runs over all partitions $a = t_0 < t_1 < \cdots < t_v = b$ of $[a,b]$ satisfying $\min_{1\leq i\leq v}(t_i-t_{i-1}) > \delta$ and $w_x$ is the modulus of continuity defined for an interval $I\subset[a,b]$ as
\[ w_x(I) = \sup_{s,t \in I} |x(s) - x(t)| \, . \]

\begin{lemma}\label{Prob_T_tight} 
For any $0<a<b$, the family  $\{ \Prob_T\colon T\ge 1\}$ is a tight family of probability measures
on $D([a,b])$, where $\Prob_T$ is the law of 
\[ V_T = \Big(\big(\tfrac{Z_{tT}}{r_T},\tfrac{\Phi_{tT}(Z_{tT})}{a_T}\big) \, : \, t \in [a,b]\Big) \, , \]
under $\Prob$. \end{lemma}

\begin{proof} We have to check the two conditions in~(\ref{tightness}).

(i) First recall from~\cite[Lemma 3.2]{KLMS09} that eventually for all $t$,  the function $t \mapsto \Phi_t(Z_t)$ is increasing, so that we can assume throughout the proof that this property holds for all $t \geq a T$. 
Note that 
\[ \| V_T \| = \sup_{t \in[a,b]} \big\{ \big|\tfrac{Z_{tT}}{r_T}\big| + \big|\tfrac{\Phi_{tT}(Z_{tT})}{a_T}\big| \big\}
= \sup_{t \in [a, b]} \big\{ \tfrac{|Z_{tT}|}{r_T} \big\} + \tfrac{\Phi_{bT}(Z_{bT})}{a_T} \, .\]
Therefore, we find that for any $\kappa > 0$
\be{first_tightness} \bal \Prob\big\{ \| V_T\| \geq \kappa \big\} & \leq \Prob \Big\{\sup_{t\in[a,b]} \tfrac{|Z_{tT}|}{r_T} \geq \tfrac{\kappa}{2}\Big\} + \Prob\{ \tfrac{\Phi_{bT}(Z_{bT})}{a_T} \geq \tfrac{\kappa}{2} \big\}
%\\
%& \ra \Prob \{ (b)^{q+1}|Y^{\ssup 1}| +  (b)^q Y^{\ssup 2}\geq \kappa \} \quad \mbox{as } T\ra \infty
\, . \eal\ee
Now, by Lemma~\ref{max_Z} and the weak convergence of $\Phi_t(Z_t)/a_t \weakconv Y^{\ssup 2}_1$, we can deduce that
the above expressions tend to zero, if we first let $T\ra\infty$ and then $\kappa \ra\infty$.
%where we used that $\lim_{T\ra\infty} a_{Tb}/a_T =  b^q a_T$ and $\lim_{T\ra\infty} r_{Tb} /r_T = b^{q+1}$ and $$(Z_T/r_T,\Phi_T(Z_T)/a_T) \weakconv (Y_1^{\ssup 1},Y_1^{\ssup 2})$$ by~\cite[Lemma 6.1]{KLMS09}, where the latter is a random variable with density on $\R^d\times\R^+$. Thus, taking the limit as $\kappa \ra \infty$ shows that the probability on the right-hand side in~(\ref{first_tightness}) tends to $0$, so that condition (i) is satisfied.

(ii) Fix $\delta > 0$ and a partition $(t_i)_{i=0}^v$ of $[a,b]$ such that $\delta < t_{i+1}-t_i < 2 \delta$ and such that all the jump times of $(Z_{tT} \, : \, t \in[a,b])$ are some of the $t_i$. This is possible if all the jump times $\tau_i$ of $Z_t$ in $[a T, b T]$ satisfy $\tau_{i+1} - \tau_i \geq \delta T$, an event which by Lemma~\ref{small_jumps} has probability tending to $1$ if we first let $T \ra \infty$ and then $\delta \ra 0$. Thus, we can work on this event from now on.

First, using that $Z_{tT}$ does not jump in $[t_{i-1},t_i)$ and the fact that $\Phi_t(Z_t)$ is increasing and $t\mapsto \xi(Z_t)$ nondecreasing by Lemma~\ref{xi_of_Z_t}, we can estimate 
\[\bal w_{V_T}[t_{i-1},t_i) & = \sup_{s,t\in[t_{i-1},t_i)} \big| \tfrac{Z_{tT}}{r_T} - \tfrac{Z_{sT}}{r_T}\big|
+ \sup_{s,t \in [t_{i-1},t_i)} \big| \tfrac{\Phi_{tT}(Z_{tT})}{a_T} - \tfrac{\Phi_{sT}(Z_{sT})}{a_T} \big| \\
& = \tfrac{1}{a_T} \big(\Phi_{t_{i}T}(Z_{t_{i-1}T}) - \Phi_{t_{i-1}T}(Z_{t_{i-1}T})\big) \\
& = \tfrac{1}{a_T} \big( \tfrac{1}{t_{i-1}T} - \tfrac{1}{t_iT} \big)\, \big( |Z_{t_{i-1}}| \log \xi(Z_{t_{i-1}}) - \eta(Z_{t_{i-1}}) \big) \\
& \leq \tfrac{2 \delta}{ a^2} \sup_{s\in[a,b]}\{\tfrac{|Z_{sT}|}{r_T}\} \tfrac{\log \xi(Z_{bT})}{\log T} \,.  \eal \]
Now, recall that, by~(\ref{upper_bound_on_xi_of_Z}), % for any $\delta>0$, 
we can bound $\xi(Z_t) \leq a_t \log t$ eventually for all $t$ so that together with $\log a_T = (q + o(1))\log T$
we obtain
\[ w'_{V_T}(\delta) \leq \tfrac{2 \delta}{ a^2} \sup_{s\in[a,b]}\{\tfrac{|Z_{sT}|}{r_T}\} \tfrac{\log \xi(Z_{bT})}{\log T}
\leq \tfrac{2 q \delta}{ a^2} \sup_{s\in[a,b]}\{\tfrac{|Z_{sT}|}{r_T}\} (1+o(1))
%\frac{1}{\log T}\frac{2\delta}{a^2}\frac{|Z_{bT}|}{r_T}\log (\log bT)^\delta \leq \frac{|Z_{bT}|}{r_T} \frac{2\delta}{a^2} q(1+o(1)) 
\, . \]
Finally, we can use Lemma~\ref{max_Z} to deduce that
\[ \lim_{\delta\da 0}\limsup_{T \ra \infty} \Prob \big\{ w'_{V_T}(\delta) \geq \epsilon \big\} 
%\leq \Prob \Big\{ b^{q+1}\frac{2\delta q}{a^2} |Y^{\ssup 1}| \geq \eps \Big\} 
\leq \lim_{\delta\da0}\limsup_{T\ra\infty}\Prob\Big\{ \tfrac{2 q \delta}{ a^2} \sup_{s\in[a,b]}\{\tfrac{|Z_{sT}|}{r_T}\} (1+o(1)) \geq \epsilon \Big\} = 0
 \, , \] 
%where the latter probability tends to $0$ as $\delta \da 0$, 
so that also the second part of the criterion~(\ref{tightness}) is satisfied.
\end{proof}

\subsection{Functional limit theorem for the maximizer of the solution profile}\label{transfer_spatial_limit}

In this section, we prove Theorem~\ref{spatial_limit_u} by translating the functional limit theorem from the maximizer of the variational problem to the maximizer of the solution profile. We prove both parts (a) and (b) simultaneously. 
%Therefore, 
%we will extend our topology to the space of c\`adl\`ag functions $f:(0,\infty) \ra \R^d\times\R\times\R$.
The main argument is contained in the following lemma.
%Again, it suffices to consider the process on an interval $[\eps,M]$. Therefore, define
%\[ \tilde{X}_T = \big( \big(\tfrac{X_{tT}}{r_T}, \tfrac{1}{t T a_T} \log U(tT)\big)  \,:\, t \in[\eps,M] \big) \, .\]
%Denote by $\dist$ the Skorokhod metric on the space of c\`adl\`ag functions on $\tilde{H}$. 

\begin{lemma}\label{superappro}
As $T\ra\infty$, the difference process
$$\Big( \big(\tfrac{Z_{tT}}{r_T},\tfrac{\Phi_t(Z_{tT})}{a_T},\tfrac{\Phi_{tT}(Z_{tT})}{a_T} + \tfrac{q}{t}\tfrac{|Z_{tT}|}{r_T}\big)
- \big(\tfrac{X_{tT}}{r_T}, \tfrac{1}{ a_T} \tfrac{\log U(tT)}{tT}, \tfrac{\xi(X_{tT})}{a_T}\big) \colon t>0 \Big)$$
tends to zero in probability.
\end{lemma}

\begin{proof} Denoting the difference process above by $(D_T(t) \colon t>0$) it suffices to show
that, for  any $0<a<b$, there exist time-changes $\la_T\colon [a,b]\ra[a,b]$ such that as $T\uparrow\infty$, in probability, 
$$\sup_{t \in [a,b]} |\la_T(t) - t|  \Rightarrow 0 \quad \mbox{ and } 
\sup_{t \in [a,b]} \|D_T(\la_T(t))\| \Rightarrow 0.$$
Fix $0<a<b$. %Let $0<\gamma <\frac{b-a}{4}$ and n
Note that, by Proposition~\ref{ageing_for_Z}, 
\be{bdry}
\lim_{\gamma \da 0} \lim_{T\ra\infty}\Prob\big\{ Z_{a T} = Z_{a T(1+\gamma)} \mbox{ and } Z_{b T(1-\gamma)} = Z_{b T} \big\} =1.
\ee
so that we can henceforth assume that $0<\gamma <\frac{b-a}{4}$ is given such that the event above holds.
Let $(\sigma_i, i = 0,1,\ldots)$ be the jump times of $(X_t \colon t\ge aT)$ and $(\tau_i, i = 0,1,\ldots)$ 
be the jump times of $(Z_t \colon t\ge aT)$, both in increasing order. Recall from the discussion in 
Section~\ref{sect_asymptotics_sigma} that if $T$ is large enough then the jump times always occur in pairs 
which are close together, i.e.~for $\beta > 1 + \frac{1}{\alpha-d}$ each connected component of the set $\calE(\beta)$, defined in~(\ref{definition_of_calE}), contains exactly one jump time of each of the two processes.
In particular, by Lemma~\ref{length_intermediate_interval}, there exists $\delta > 0$ such that
\be{close_jumps} 
 \frac{|\sigma_i - \tau_i|}{\tau_i} \leq (\log \tau_i)^{-\delta} \leq (\log a T)^{-\delta} < \gamma \, ,\ee
so that under the event in~\eqref{bdry} there exists
$$N=\max\big\{i \colon \sigma_i\in[aT,bT]\big\}= \max\big\{i \colon \tau_i\in[aT,bT]\big\}.$$
%if we assume $Z_{a T} = Z_{a T(1+\gamma)}$ and $Z_{b T(1-\gamma)} = Z_{b T}$, we know that all pairs $\sigma_i, \tau_i \in (a T, b T)$.
%(if necessary set $\sigma_0$ or $\tau_0$ equal to $a T$, similarly set $\sigma_N$ or $\tau_N$ equal to $bT$). 
%Then, for all $i = 0,\ldots,N$, $\sigma = \sigma_i$ and $\tau = \tau_i$ satisfy the bound~(\ref{close_jumps}). 
%Recall, also from Section~\ref{sect_asymptotics_sigma} that 
%Now, we can set up a time change that relates $X_t$ and $Z_t$ as follows. 
Denote $s_i = \sigma_i /T$ and $t_i = \tau_i / T$ and define $\lambda=\lambda_T \colon [a,b] \ra \R$ 
such that $\la(a) = a, \la(b) = b$ and $\la(s_i) = t_i$ for all $i = 0,\ldots,N$, and 
linear between these points. Then
\be{time_change} \bal \sup_{t \in [a,b]} | \la(t) - t| 
& = \sup_{i = 0,\ldots,N} |\la(s_i) - s_i| = \sup_{i = 0,\ldots,N}\tfrac{1}{T}|\tau_i - \sigma_i|
\leq b \sup_{i = 0,\ldots,N}\tfrac{|\tau_i-\sigma_i|}{\tau_i} \\
& \leq b \sup_{i=0,\ldots,N} (\log \tau_i)^{-\delta}
\leq b (\log a T)^{-\delta} \, , \eal
\ee
which converges to $0$ when $T \ra \infty$, as required.

We now look at the individual components of the process~$D_T$.
For the \emph{first component}, we simply observe that the time-change is set up in such way that 
$X_{tT} = Z_{\la(t) T}$ for all $t \in [a,b]$. For the \emph{second component}, we split
\begin{align} \tfrac{1}{a_T} & \big| \tfrac{\log U(tT)}{tT} - \Phi_{\la(t)T}(Z_{\la(t)T}) \big| \notag
\\ & \leq \tfrac{1}{a_T}  \big| \tfrac{\log U(tT)}{tT} - \Phi_{tT}(Z_{tT}) \big|
+ \tfrac{1}{a_T} \big| \Phi_{tT}(Z_{tT}) - \Phi_{\la(t)T}(Z_{tT})\big| \label{estimate_second_component}\\
& \hspace{1cm} + \tfrac{1}{a_T} \big| \Phi_{\la(t)T}(Z_{tT}) - \Phi_{\la(t)T}(Z_{\la(t)T})\big| \notag \, , \end{align}
and look at the three terms separately.
%show that each of the expressions tends to $0$ as $T\ra \infty$ (uniformly for all $t \in [a,b]$).
For the first term, we use Propositions 4.2 and 4.4 from~\cite{KLMS09} to conclude that there exists $\delta' > 0$ and $C>0$ such that almost surely, for all sufficiently large~$t$,
\[ \Phi_t(Z_t)- 2d + o(1) \leq \frac{1}{t} \log U(t) \leq \Phi_t(Z_t) + Ct^{q-\delta'} \, , \]
so that the first term in~(\ref{estimate_second_component}) tends to $0$ uniformly for all $t \in[a,b]$. 
For the second term, we use the bound $\eta(z) \leq |z| \log d$, the bound~(\ref{time_change}) 
for the time-change, and that, by~(\ref{upper_bound_on_xi_of_Z}) combined with~\cite[Lemma 3.2]{KLMS09}, there exists a $\delta'>0$ such that 
\be{xiest}a_t(\log t)^{-\delta'} \leq \xi(Z_t) \leq a_t (\log t)^{\delta'}.\ee 
This gives, for $T$~large enough and all $t\in[a,b]$,
\[ \begin{aligned} \tfrac{1}{a_T}\big| \Phi_{tT}(Z_{tT}) & - \Phi_{\la(t)T}(Z_{tT}) |  = 
\tfrac{1}{a_T}\big|\tfrac{1}{tT}- \tfrac{1}{\la(t)T}\big| \, \big| |Z_{tT}| \log \xi(Z_{tT}) - \eta(Z_{tT})\big| \\	
& \leq \tfrac{1}{a^2} |\la(t) - t| \sup_{t \in [a,b]}\tfrac{|Z_{tT}|}{r_T \log T}  \max\{ |\log \xi(Z_{bT})|, 2d\} \\
%& \leq \tfrac{1}{a^2b}  (\log a T)^{-\delta} \frac{|Z_{bT}|}{r_{bT}} \frac{r_{bT}}{r_T \log T} (\log a_T + \log (\log b T)^{\delta'}) \\ 
& \leq (1+ o(1))\tfrac{qb}{a^2} (\log a T)^{-\delta}  \sup_{t \in [a,b]}\tfrac{|Z_{tT}|}{r_T}\, ,
\end{aligned} \]
and the right hand side tends to zero in probability by Lemma~\ref{max_Z}.
%as $\frac{Z_{bT}}{r_{bT}}$ tends by~\cite[Lemma 6.2]{KLMS09} to a random variable with density.
In order to deal with the last term in~(\ref{estimate_second_component}), note that if $t \in (s_i\vee t_i, s_{i+1} \wedge t_{i+1})$ for some $i=0,\ldots,N-1$, then $Z_{tT} = Z_{\la(t)T}$ so that the term vanishes.  Otherwise, if $t \in [s_i \wedge t_i,s_i \vee t_i]$, then $tT$ is in the set of transition times $\calE$ as discussed in Section~\ref{sect_asymptotics_sigma} and we find that $\{ Z_{tT}, Z_{\la(t)T} \} \subset \{Z_{\la(t)T}^{\ssup 1}, Z_{\la(t)T}^{\ssup 2}\}$ and also that there exists $\beta > 1 + \frac{1}{\alpha - d}$ such that 
\[ \begin{aligned} \tfrac{1}{a_T} \big| \Phi_{\la(t)T}(Z_{tT})  - \Phi_{\la(t)T}(Z_{\la(t)T}) \big|
& \leq \tfrac{1}{a_T} \big( \Phi_{\la(t)T}(Z_{\la(t)T}^{\ssup 1}) - \Phi_{\la(t)T}(Z_{\la(t)T}^{\ssup 2}) \big)
\\ & \leq \tfrac{a_{\la(t)T}}{a_T} (\log \la(t) T)^{-\beta} \\
& \leq b^q(1+o(1))  (\log a T)^{-\beta}\, , \end{aligned}\]
which tends to zero uniformly in $t\in[a,b]$ completing the discussion of the second component. 

Finally, we consider the \emph{third component}. Using~\eqref{xiest} and that $Z_{tT}= X_{\la^{-1}(t) T}$, we estimate
\[ \bal \big| \tfrac{\Phi_{tT}(Z_{tT})}{a_T} + \tfrac{q}{t}\tfrac{|Z_{tT}|}{r_T} - \tfrac{\xi(X_{\la^{-1}(t)T})}{a_T} \big| 	
& = \big| \tfrac{q}{t}\tfrac{|Z_{tT}|}{r_T} - \tfrac{|Z_{tT}|}{t r_T}\tfrac{\log \xi(Z_{tT})}{\log T} + \tfrac{\eta(Z_{tT})}{t r_T \log T} \big|  \\
& \leq C'\, \tfrac1a\,\sup_{t \in [a, b]}\big\{ \tfrac{|Z_{tT}|}{r_T}\big\}\,\tfrac{\log \log b T}{\log T}  \, , \eal \]
where $C'$ is some constant depending on $a,b$. By Lemma~\ref{max_Z}, the right hand side converges 
in probability to zero,  which completes the proof of the lemma.
\end{proof}

\begin{proof}[Proof of Theorem~\ref{spatial_limit_u}] By a classic result on weak convergence, see e.g.~\cite[Thm. 3.1]{Bil99}, 
the previous lemma ensures that the processes
\[ \Big( \big(\tfrac{X_{tT}}{r_T}, \tfrac{1}{a_T} \tfrac{\log U(tT)}{tT}, \tfrac{\xi(X_{tT})}{a_T}\big) \, : \, t>0 \Big) 
\mbox{ and }\Big( \big(\tfrac{Z_{tT}}{r_T},\tfrac{\Phi_t(Z_{tT})}{a_T},\tfrac{\Phi_{tT}(Z_{tT})}{a_T} + \tfrac{q}{t}\tfrac{|Z_{tT}|}{r_T}\big) \, : \, t>0 \Big) \, , \]
have the same limit, which was identified in Proposition~\ref{spatial_limit}, as 
\[ \Big(\big(Y^{\ssup 1}_t,Y^{\ssup 2}_t+\tfrac{d}{\alpha-d}\big(1-\tfrac{1}{t}\big)|Y^{\ssup 1}_t|,Y^{\ssup 2}_t+\tfrac{d}{\alpha-d} |Y^{\ssup 1}_t|\big)\, : \, t>0 \Big) \, . \]
%\big( Y^{\ssup 2}_t+\tfrac{d}{\alpha-d}\big(1-\tfrac{1}{t}\big)|Y^{\ssup 1}_t|\, : \, t  > 0 \big)
Hence, projecting onto the first and third component proves~(a), and projecting onto 
the second component and noting that all involved processes are continuous proves~(b).
\end{proof}

\begin{proof}[Proof of Proposition~\ref{classical}] We focus on the one-dimensional distributions, as the higher dimensional case works analogously. Fix $t>0$ and let $f$ be a continuous, bounded nonnegative function on $\R^d$. 
%in particular there exists $\kappa > 0$ such that $f(x) \leq \kappa$ for all $x \in \R^d$. In order to show that
Denote
$$\xi_{tT}(f) := \big(\tfrac{T}{\log T}\big)^{\frac{\alpha d}{\alpha-d}} \, \int v\big(tT, \big(\tfrac{T}{\log T}\big)^{\frac{\alpha}{d-\alpha}}x\big) \, f(x) \, \diff x
=  \sum_{y\in \Z^d} v(tT,y) \, f\big(\tfrac{y}{r_T}\big).$$
It suffices to show that the Laplace functionals converge, \ie 
$$\lim_{T\uparrow\infty} \E \big[ e^{-\xi_{tT}(f)} \big] = \E \big[ e^{-f(Y_{t})} \big] \, . \vspace{-4mm}$$
Let $\kappa>0$ be an upper bound for~$f$. For small $\eps > 0$ and $\delta = \frac{\log (1+\frac{\eps}{2})}{\kappa}$, 
consider the event $A_{\delta} = \{ v(tT,Z_{tT}) > 1-\delta \}$.
Since $v(tT, Z_{tT}) \weakconv 1$, we can choose $T_0$ large enough such that $\Prob(A_{\delta}) > 1 -\frac{\eps}{2}$
for all $T\geq T_0$.  We estimate
$$\E[ e^{-f(Y_t)}] - \E\big[e^{-\xi_{tT}(f)}\big] \leq \E[e^{-f(Y_t)}] - \E\big[ e^{-f(\frac{Z_{tT}}{r_T} )} \1_{A_\delta}\big] e^{-\delta\kappa} 
 \leq \E[e^{-f(Y_t)}] - \E\big[ e^{-f(\frac{Z_{tT}}{r_T} )}\big] + \eps \, ,$$
and
$$\E[e^{-f(Y_t)}] - \E[ e^{-\xi_{tT}(f)} ] 
% \geq \E[e^{-f(Y_t)}] - \E\big[ e^{-\sum_{x\in\Z^d} v(tT,x)f(\frac{x}{r_T})} \1_{A_\delta} \big] - \tfrac{\eps}{4} 
\geq \E[ e^{-f(Y_t)}] - \E\big[ e^{-(1-\delta)f(\frac{Z_{tT}}{r_T})} \big] -\tfrac{\eps}{2} 
\geq \E[ e^{-f(Y_t)}] - \E\big[ e^{-f(\frac{Z_{tT}}{r_T})} \big] -\eps.$$
As $\frac{Z_{tT}}{r_T} \Rightarrow Y_t$ by Proposition~\ref{spatial_limit}, the statement follows.
\end{proof}
%\medskip

{\bf Acknowledgments:} We would like to thank Hubert Lacoin for a stimulating discussion and
for suggesting the proof of Lemma~\ref{properties_jumps_of_X}. 
The first author is supported by an Advanced Research Fellowship from
the {\it Engineering and Physical Sciences Research Council} (EPSRC). This article is based on material
from the second author's PhD thesis.
%\medskip

%%%%%%%%%%%%%%%%%%%%%%%%%%%%%%%%%%%%%%%%%%%%%%%%%%%%%%%%%%%%%%%%%%%%%%%%%%%%%%
%															Bibliography 																	 %
%%%%%%%%%%%%%%%%%%%%%%%%%%%%%%%%%%%%%%%%%%%%%%%%%%%%%%%%%%%%%%%%%%%%%%%%%%%%%%

\bibliographystyle{alpha}
\bibliography{ageing}

\end{document}